\documentclass[10pt]{article}

\usepackage{amsmath,amssymb,mathrsfs,amsthm}
\usepackage{amsfonts}
\usepackage[fleqn]{mathtools}
\usepackage{dsfont}
\usepackage{latexsym}

\usepackage[margin=0.83in]{geometry}
\usepackage{setspace}
\usepackage{float}
\usepackage{multicol}
\usepackage{enumitem}
\usepackage{graphicx}
\usepackage[dvipsnames]{xcolor}
\usepackage{ragged2e}
\usepackage[nice]{nicefrac}

\usepackage[english]{babel}
\usepackage{authblk}

\usepackage[sort&compress,numbers]{natbib}

\usepackage[title]{appendix}

\usepackage{setspace}

\newtheorem{theorem}{Theorem}[section]
\newtheorem{lemma}{Lemma}[section]

\newtheorem{corollary}{Corollary}[section]

\newtheorem{remark}{Remark}[section]

\newtheorem{rule-def}[theorem]{Rule}
\theoremstyle{definition}

\numberwithin{equation}{section}

\usepackage[dvipsnames]{xcolor}
\emergencystretch=\maxdimen
\usepackage[
    colorlinks=true,
    linkcolor=magenta,    
    citecolor=green,   
    urlcolor=blue,     
    filecolor=blue,
    linktoc=all
]{hyperref}

\let\originaleqref\eqref
\renewcommand{\eqref}[1]{\textcolor{magenta}{\originaleqref{#1}}}
\providecommand{\keywords}[1]
{
  \textbf{\textit{Keywords:}} #1
}

\providecommand{\msc}[1]
{
  \textbf{\textit{MSC codes (2020):}} #1
}

\title{Global classical solutions for a metastatic cancer invasion model with negative haptotaxis}
\author[1]{Nishith Mohan\thanks{mohan@mathematik.uni-kl.de}}
\author[2]{Dimitrios Katsaounis\thanks{katsaounis@igpm.rwth-aachen.de}}
\affil[1]{RPTU Kaiserslautern-Landau, Department of Mathematics, \authorcr Gottlieb-Daimler-Str. 48, 67663 Kaiserslautern, Germany.}
\affil[2]{Institute for Geometry and Applied Mathematics,  RWTH Aachen, \authorcr  Im S\"usterfeld 2, 52072 Aachen, Germany.}
\date{}
\begin{document}
\makeatletter
\renewcommand\@fnsymbol[1]{}
\makeatother
\maketitle
\begin{abstract}
We analyze a coupled PDE-ODE model for metastatic cancer invasion involving epithelial-like and mesenchymal-like cancer cells, cancer-associated fibroblasts (CAFs), the transforming growth factor $\beta$, matrix-degrading metalloproteinases, and extracellular matrix reconstruction by CAFs. The model incorporates bidirectional phenotypic transitions between epithelial-like and mesenchymal-like cancer cells, cancer-cell haptotaxis, negative fibroblast haptotaxis, as well as cell-specific proliferation. Under explicit parameter restrictions, we establish global classical solvability for the full model in two dimensions using an energy-based approach, and for a three-dimensional variant with constant phenotypic transition rates using coupled $L^p$-estimates. Numerical simulations illustrate primary tumor invasion and metastatic spread to secondary sites.
\end{abstract}

\msc{35A01 (primary), 
     35K57, 
     35A09, 
     35B45, 
     35Q92, 
     92C17 (secondary)  
     }

\section{Introduction}
 In this article, we investigate the global classical solvability of a cancer invasion and metastasis model and further explore its behavior through numerical simulations. The model is a modified version of the one proposed in \cite[(43)]{katsaounis2024genuinely}, where a system of advection-reaction-diffusion PDEs describes cancer growth and tumor progression through cellular migration. The model reads as follows:
\begin{equation}
\label{model1}
\begin{cases}
\partial_t c^{E}
 = \Delta c^{E} -\chi_{E}\nabla\cdot\!\big(c^{E}\nabla v\big)
-\mu_{E}\,\zeta_1(b)\,c^{E} + \mu_{M}\,\zeta_2(b)\,c^{M} \\[7pt]
\hspace{4.5cm}  +\rho_{E}\,c^{E}\big(1-c^{E}-c^{M}-c^{F}-v\big),
& x\in\Omega,\ t>0,
\\[7pt]
\partial_t c^{M}
 = \Delta c^{M}
-\chi_{M}\nabla\cdot\!\big(c^{M}\nabla v\big)
+\mu_{E}\,\zeta_1(b)\,c^{E} - \mu_{M}\,\zeta_2(b)\,c^{M},
& x\in\Omega,\ t>0,
\\[7pt]
\partial_t c^{F}
 = \Delta c^{F}
+\chi_{F}\nabla\cdot\!\big(c^{F}\nabla v\big)+\rho_{F}\,(1+c^{E})\,c^{F}\big(1-c^{E}-c^{M}-c^{F}-v\big)
-\lambda_{F}\,c^{F},
&x\in\Omega,\ t>0,
\\[7pt]
\partial_t b
= \Delta b+\rho_{b}\,c^{F}-\lambda_{b}\,b,
& x\in\Omega,\ t>0,
\\[7pt]
\partial_t m
= \Delta m+\rho_{m}^{F}\,c^{F}+\rho_{m}^{E}\,c^{E}-\lambda_{m}\,m,
& x\in\Omega,\ t>0,
\\[7pt]
\partial_t v
= \rho_{v}\,c^{F}\big(1 - v\big) - \delta_1 mv - \delta_2 c^E v - \delta_3 c^M v,
& x\in\Omega,\ t>0,  
\end{cases}
\end{equation}
where $\Omega \subset \mathbb{R}^N$, $N \in \{2,3\}$, is a bounded domain with sufficiently smooth boundary $\partial\Omega$, and $\nu$ denotes the outward unit normal vector on $\partial\Omega$. The system is supplemented with the zero-flux boundary conditions
\begin{align}
    \label{model2}
     \partial_\nu c^E - \chi_E c^E\partial_\nu v = \partial_\nu c^M - \chi_M c^M\partial_\nu v = \partial_\nu c^F & + \chi_F c^F\partial_\nu v = \partial_\nu b = \partial_\nu m = 0,  ~~ x\in \partial \Omega,\ t>0,
\end{align}
together with the initial conditions
\begin{align}
\label{model3}
\left(c^E, c^M, c^F, b, m, v\right)(x,0)
=
\left(c_0^E, c_0^M, c_0^F, b_0, m_0, v_0\right),
\qquad x\in\Omega.
\end{align}
We denote with $c^E: \Omega \times \mathbb{R}^+ \to \mathbb{R}^+$ the density of Epithelial-like Cancer Cells (ECCs), with $c^M:\Omega\times\mathbb{R^+} \to \mathbb{R}^+$ the density of Mesenchymal-like Cancer Cells (MCCs), with $c^F: \Omega \times \mathbb{R^+} \to \mathbb{R}^+$ the density of Cancer-Associated Fibroblasts (CAFs), with $b: \Omega \times \mathbb{R^+} \to \mathbb{R}^+$ the density of the Transforming Growth Factor $\beta$ (TGF-$\beta$), with $m: \Omega \times \mathbb{R^+} \to \mathbb{R}^+$ the density of the Matrix-degrading Metalloproteinases (MMPs) and with $v: \Omega \times \mathbb{R^+} \to \mathbb{R}^+$ the density of the Extracellular matrix (ECM).

\subsection*{Model description}
Model \eqref{model1} describes the spatio-temporal progression of tumor growth and the creation of metastatic colonies in distant locations away from the primary tumor. In particular, we consider two phenotypes of cancer cells that play different roles in the so-called invasion-metastasis cascade \cite{HW2000, Gupta2006, Hanahan2022}. The ECCs are primarily of a proliferative nature and are responsible for the creation of large tumors, usually observable in histological studies \cite{Talmadge2010}. However, ECCs are also capable of invading the local tissue, either through diffusive spreading or by performing haptotaxis, i.e., following the gradient of high-density regions of the ECM. The ECM is an amalgamation of, primarily, collagen fibers and other smaller proteins and molecules that create the basis of an organ's connective tissue \cite{Friedl2003}. Cancer cells are able to adhere to these fibers to facilitate their motion and hence move to denser regions of the matrix. The second type of cancer cells that play a crucial role in the creation of metastasis are the MCCs. These are primarily of a migratory nature, where they navigate the extracellular environment as solitary cells or in small clusters, and are capable of arriving in distant locations away from the primary tumor site through haptotaxis \cite{Sabeh2009, Friedl2012}. A key underlying process that facilitates cancer invasion and metastasis is the Epithelial-to-Mesenchymal Transition (EMT), which describes the transition of cancer cells with epithelial-like phenotypic characteristics to a mesenchymal-like type of cell \cite{Kalluri2009}. This process allows for a dynamic tumor landscape where ECCs and MCCs work in parallel, with the former promoting local tumor expansion and the latter facilitating cancer invasion in distant areas of the tissue. Equivalently, the reverse process, called Mesenchymal-to-Epithelial Transition (MET), allows for MCCs that have arrived in new cancer-free locations to revert to their original epithelial character and initiate the growth of metastatic tumor colonies \cite{Tsuji2009}. Let us note here that the EMT and MET processes are generally considered to be continuous, where cells of epithelial-like phenotype progressively attain mesenchymal-like characteristics, and vise versa; hence, during the transition, hybrid epithelial-mesenchymal cells arise, leading to a heterogeneous cancer cell population \cite{Pastushenko2019}. In model \eqref{model1} we have simplified the dynamics and focus only on the extreme cases of the epithelial-mesenchymal spectrum.

This dynamic behavior of the two cancer cell phenotypes is described by the first two equations in model \eqref{model1}, where both ECCs and MCCs spread spatially through linear diffusion and due to advection towards the gradient of the ECM. The phenotypic transitions between the two cancer phenotypes are described here as reaction terms, where $\mu_E>0$ is the rate of the EMT process, and $\mu_M>0$ is the rate of the MET process. The phenotypic transitions come through the functions $\zeta_1(s)$ and $\zeta_2(s)$, which depend on the density of TGF-$\beta$, which has been identified as a triggering factor for this process \cite{Oft1998}, and are assumed to satisfy the following:
\begin{equation}
\label{zeta}
\zeta_i(s):\mathbb{R}^+\to\mathbb{R}^+ \text{ is } C^1,\quad \text{and} \quad 0<\zeta_i(s)\le \zeta^\ast \quad \text{for all } s\ge 0,\ \text{for some } \zeta^\ast>0, \; i \in \{1, 2\}.
\end{equation}
These are ``switch'' functions, claiming that in areas of high TGF-$\beta$ concentration, the EMT process will be triggered (through $\zeta_1$) and conversely, in low TGF-$\beta$ density regions the MET process will take place (through $\zeta_2$). A typical choice of these EMT-MET switch functions is given by 
\begin{align}
\label{zeta_Sim}
\zeta_1(s) = \frac{L}{1 + e^{-k(s - b_T)}} \qquad \text{and} \qquad \zeta_2(s) = \frac{L}{1 + e^{k(s - b_T)}}
\end{align}
for some $L > 0$, where $b_T>0$ is the critical TGF-$\beta$ density threshold that will dictate whether EMT or MET will take place in a certain region of the domain. Finally, we assume a logistic-type growth term for the ECCs to account for their proliferative nature, where the cells compete for space with the other cells or the ECM. Note that in contrast to the ECCs, there is no proliferation term included for the MCCs since they are, primarily, migratory cells \cite{Friedl2003, Friedl2012}.

However, cancer cells are highly dependent on external and extracellular factors of the tumor microenvironment for the successful growth and establishment of metastatic tumors. One important player in tumor progression is the CAF cells that are part of the tumor microenvironment \cite{Xing2010CAFs}. CAFs arise from regular fibroblast cells that are recruited by the cancer cells to facilitate tumor progression. In particular, they have the ability to remodel and reconstruct the ECM, thus creating a prosperous environment for cancer cells to use and invade the tissue. Hence, CAFs tend to perform negative haptotaxis, meaning that they move towards regions of low ECM density to repair the matrix \cite{Darby2014CAFs}. The third equation of \eqref{model1} describes the spatiotemporal movement of CAFs through negative haptotaxis and linear diffusion, where we also assume a logistic-type growth that is enhanced in the regions where more ECCs are present, and a constant death rate. These assumptions align with the fact that we observe a larger population of CAFs close to their recruiter cells, the ECCs. Furthermore, CAFs play a crucial role in the EMT process, since they are one of the primary producers of the growth factor TGF-$\beta$ \cite{Yu2014CAFs}. As already mentioned, TGF-$\beta$ has been identified as a triggering factor for the EMT and MET processes \cite{Xu2009}. The fourth equation of \eqref{model1} shows the evolution of the TGF-$\beta$ density, where we assume linear diffusion and constant growth due to CAFs, as well as a constant decay rate. Furthermore, CAFs are one of the main producers of the matrix-degrading enzymes, MMPs, that facilitate, as their name suggests, the remodeling and degradation of the ECM, hence facilitating the movement of the cancer cells through the collagen fibers of the matrix \cite{Erdogan2017CAFs}. There are many different types of MMPs, and ECCs are also able to produce certain types of MMPs that enable their movement through the ECM \cite{Noel2012}. Taking all that into account, the fifth equation in \eqref{model1} models the evolution of MMPs through linear diffusive spreading, constant growth through the CAFs and ECCs, as well as some natural decay. Finally, the sixth equation tracks the density evolution of the ECM. The collagen matrix is assumed to be static, and logistic-type growth is considered due to the presence of CAFs that compete for space with the ECM. The ECM is degraded due to the soluble MMPs that are secreted by the ECCs and CAFs, as well as by the cancer cells themselves. This is due to the fact that cancer cells produce certain matrix-degrading enzymes that stay on their cellular membranes, thus allowing for the local degradation of the matrix \cite{Sabeh2009, Ito2012}.

Overall, model \eqref{model1} includes several key biological components that play a crucial role in the invasion-metastasis cascade. As aforementioned, this is a simplified version of the hybrid model introduced in \cite{katsaounis2024genuinely}, which is an extension of the works in \cite{sfakianakis2020hybrid, franssen2021anovel3d}, and here we highlight the key differences and simplifications made to facilitate the analytical investigation of the phenomenon, as well as some other relevant works on cancer invasion and metastasis. Firstly, we have replaced the hybrid part of the model, where the spatio-temporal evolution of the MCCs was given by an individual-based model of SDEs, with a density formulation for the MCCs. Instead of tracking each individual MCC separately, we now monitor the progression of the entire MCC density. The original modeling assumption in \cite{katsaounis2024genuinely} was justified due to the observation, in histological studies, of much fewer MCCs appearing in cancerous tissues compared to ECCs (approx. $10^3$ per cm$^3$ of MCCs compared to $10^9$ per cm$^3$ of ECCs), thus making a density description for MCCs less feasible. However, when considering larger-scale tissue and organ dynamics, using density formulations again becomes more realistic for all the biological components, thus making the current modeling description a reasonable approximation of the system dynamics. Another simplification made to the current model is the assumption of the linear diffusive spreading of the cells, in comparison to the porous-medium type diffusion considered in \cite{katsaounis2024genuinely}. Thus, we assume that the ensemble of cells behaves more like ``particles'' spreading in the domain, rather than a diffusive fluid-structure. Yet again, when considering a large population of cancer cells spreading in the tissue, linear diffusion becomes a reasonable approximation and has been seen to match fairly well experimental observations \cite{Sturrock2011}. 
A further simplification relates to the phenotypic transition terms that are now modeled in a deterministic fashion. In the original model, the condition of EMT taking place due to high concentrations of TGF-$\beta$ was necessary but not sufficient, in contrast to model \eqref{model1}, and it additionally relied on a random Poisson process generating the phenotypic switching, in order to account for other unknown factors that play a role in the process. The stochastic terms would have brought additional analytical challenges, so we chose to reduce the system to a deterministic level. Finally, we have made some slight modifications on the density evolution of the ECM, where now the proliferation due to the CAFs is modulated by the density of ECM in a logistic-type of growth, as well as the degradation of the matrix being separated into three different terms: (i) due to the presence of all soluble MMPs, (ii) due to the ECCs and (iii) due to the MCCs. The last two terms relate to the membrane-bound MMPs of the cancer cells, which promote the dissemination of the collagen matrix.

Mathematical modeling of cancer invasion and metastasis has been studied extensively over the last decades, and especially the use of macroscopic models of PDEs has proven to provide great insight into the intricate dynamics of the phenomenon \cite{chaplainp2025preziosi, sfakianakis2020mathematical}. Originally, several modeling attempts have focused on simplified descriptions of the invasion-metastasis cascade, where only one population of cancer cells was considered to invade the ECM \cite{anderson1999mathematical, perumpanani1999travelling, marchant2006biphasic, gerisch2008mathematical} and more focus was given on the intracellular processes that modulated the movement and growth of cancer cells \cite{chaplain2005urokinase, andasari2010mathematical} or studying specific cancer cell lines and their interactions with their respective extracellular environment \cite{perumpanani1998extracellular}. One of the first attempts to separate the dynamics of migratory and proliferative cells was done in \cite{perumpanani1996biological}, where the authors divided the cells into invasive and non-invasive, as well as separately described the different components and enzymes of the connective tissue and ECM. In recent years, the EMT process has been studied more extensively, where the dedifferentiation process has been used for haptotaxis or multiple-taxis models \cite{sfakianakis2017multiscale, kolbe2017numerical, franssen2020mathematical, kolbe2021modeling, crossley2024goorgrow}. These developments have led to increasingly detailed models that account for phenotypic heterogeneity and the interaction of different cell populations with the surrounding extracellular environment, while also introducing additional mathematical challenges associated with the resulting taxis-driven dynamics. The increasing complexity of these models has, in turn, motivated considerable interest in understanding their qualitative and analytical properties.

\subsection*{Analytical considerations}

The analytical theory for haptotaxis systems, typically formulated as coupled PDE-ODE systems, has received considerable attention in the last two decades. An early contribution is \cite{walker2007global}, where Walker and Webb studied global classical solutions for a system involving only haptotactic effects. The Chaplain-Lolas chemo-haptotaxis system \cite{chaplain2005urokinase,chaplain2006mathematical} was analyzed in a series of papers by Tao and Wang \cite{tao2008global} and by Tao \cite{tao2009global}, with a focus on global classical solvability. Tao further established global-existence results for a haptotaxis model incorporating tissue remodeling \cite{tao2011global}. Tao and Winkler later investigated the effects of nonlinear diffusion and logistic growth on the solvability of a chemotaxis-haptotaxis system \cite{tao2011chemotaxis}. Particularly relevant to the present study is the work of Tao and Winkler \cite{tao2014energy}, in which they established global solvability for a chemo-haptotaxis system with ECM remodeling (non-diffusible attractant) via energy-based estimates in a two-dimensional spatial domain. Further contributions include the boundedness result of Marciniak-Czochra and Ptashnyk \cite{marciniak2010boundedness}, as well as a series of works by Pang and Wang addressing global existence, boundedness, and asymptotic behavior in related chemotaxis-haptotaxis models \cite{pang2017global,pang2018global,pang2019asymptotic}. In a related direction, Stinner et al.\ \cite{stinner2014global} established global weak solutions for a PDE-ODE system describing multiscale cancer cell invasion.

The models considered in the analytical studies mentioned above primarily focus on invasion by a single cancer-cell population. Recent modeling efforts have moved beyond this limitation by considering multiple interacting cell populations, particularly differentiated cancer cells and cancer stem cells. In this direction, Hellmann et al.\ \cite{hellmann2016mathematical} and Sfakianakis et al.\ \cite{sfakianakis2016multiscale, sfakianakis2017multiscale} developed a deterministic multiscale framework combining three coupled subsystems describing the interactions between the above-mentioned cancer cell populations and their surrounding environment. Global classical solvability for the resulting model in a two-dimensional setting with large initial data was established by Giesselmann et al.\ \cite{giesselmann2018existence}, while further extensions concerning boundedness, three-dimensional settings, more general transition rates, and weak solvability, under different assumptions on the presence or absence of cellular proliferation, were obtained by Dai and Liu in a series of articles \cite{dai2020global,dai2022new,dai2022global} and by Dai and Ma in \cite{dai2023boundedness}. However, all the haptotaxis models discussed above incorporate \emph{positive haptotaxis}, whereby cancer cells migrate toward regions of higher ECM density. As discussed earlier, the model in \cite[(43)]{katsaounis2024genuinely}, on which the present study is based, additionally describes the migration of CAFs toward regions of lower ECM density, a mechanism referred to hereafter as \emph{negative haptotaxis}. The results on mathematical analysis of taxis systems involving negative taxis mechanisms are comparatively limited. For tumor-invasion systems combining regular haptotaxis with negative haptotaxis, Stinner et al.\ \cite{stinner2016global} established global weak solvability for a go-or-grow multiscale model involving migratory and proliferative tumor-cell populations. 
Subsequently, Kolbe et al.\ \cite{kolbe2021modeling} established global weak solutions for a simplified multiscale model with migratory and proliferative phenotypes, combining nonlinear diffusion, attraction toward tissue, repulsion from proliferating cells, and receptor-mediated phenotype transitions. In a related non-biological context, Fuest and Heydari analyzed a cross-diffusion system driven purely by negative haptotaxis for a model of rivaling gangs \cite{fuest2024cross}. Recent analytical studies have addressed reduced variants of the same underlying model in \cite[(43)]{katsaounis2024genuinely} with CAF-mediated tissue reconstruction. Tao and Winkler \cite{TaoWinkler2} established global classical solvability in arbitrary spatial dimensions under suitable conditions on the logistic degradation by developing a Moser-type iteration. In a further study, they treated a four-component variant containing a cubic zero-order nonlinearity and obtained global classical solutions for arbitrarily large initial data in two-dimensional domains \cite{TaoWinklerPreprint1}. Li \cite{Li2026Haptotaxis} subsequently considered a variant without logistic sources and proved global classical solvability in two dimensions under smallness conditions on the haptotactic sensitivities. The above discussion highlights the importance of studying such systems and motivates a rigorous mathematical investigation of their complex nonlinear interactions and the resulting challenges in establishing global solvability. 

\paragraph{Analytical challenges:} Although, as in \cite{TaoWinklerPreprint1, TaoWinkler2, Li2026Haptotaxis}, we replace porous-medium diffusion with standard Laplacian diffusion, the present analysis retains the full coupling structure of the model proposed in \cite[(43)]{katsaounis2024genuinely} and incorporates several features absent from its earlier formulations. In particular, the bidirectional EMT-MET mechanism acts as a differentiation-dedifferentiation coupling and introduces additional transfer terms beyond those treated in \cite{dai2020global,dai2022new,dai2022global,dai2023boundedness, TaoWinklerPreprint1, TaoWinkler2}. Moreover, ECCs and CAFs have different proliferation mechanisms, while MCCs are assumed to be nonproliferative. However, MCCs still undergo haptotactic migration, and the absence of logistic damping in their equation makes it more difficult to obtain suitable uniform estimates. An additional analytical difficulty arises from the MMP component, whose production is driven by the ECC and CAF populations and which, in turn, degrades the ECM. Since ECM gradients determine the haptotactic fluxes - both standard and negative- this coupling creates a feedback loop among the cell populations, MMPs, and the ECM, further complicating the closure of the a priori estimates.

The results in \cite{TaoWinklerPreprint1} and \cite{Li2026Haptotaxis} are restricted to two-dimensional domains, with the latter considering a system without logistic sources, while \cite{TaoWinkler2} obtains arbitrary-dimensional solvability under logistic degradation of the cell populations. In contrast, the present analysis accommodates a nonproliferative but haptotactically migrating MCC population and establishes global solvability in both two and three spatial dimensions. As in \cite{TaoWinklerPreprint1, TaoWinkler2, Li2026Haptotaxis}, we state our classical global solvability conditions explicitly in terms of the model parameters, providing a transparent description of the solution regime covered by the analysis. 

Our work fills several gaps in the analysis of such biologically realistic systems by establishing global classical solvability for the fully coupled model in both two- and three-dimensional domains, while accounting for bidirectional phenotypic transitions, cell population-specific proliferation mechanisms, and ECM-directed haptotactic migration of both proliferative ECCs and nonproliferative MCCs, and ECM-repellent haptotactic migration of CAFs.

\subsection*{Main results} 
We now state the main results of this work, formulated in the theorems that follow. Throughout the discussion, we assume that the initial data satisfy the following conditions for any $\vartheta\in(0,1)$:
\begin{equation}
\label{ic_assum}
\begin{cases}
c_0^E \in \bigcup_{\vartheta \in (0, 1)} C^{2 + \vartheta}(\overline{\Omega}) ~~\text{is such that}~~ c_0^E \geq 0 ~~\text{in $\overline{\Omega}$}~~\text{and}~~\partial_\nu c_0^E = 0~~\text{on}~~\partial \Omega, \\[5pt]
c_0^M \in \bigcup_{\vartheta \in (0, 1)} C^{2 + \vartheta}(\overline{\Omega})~~\text{is such that}~~c_0^M \geq 0~~\text{in $\overline{\Omega}$}~~\text{and}~~\partial_\nu c_0^M = 0~~\text{on}~~\partial \Omega, \\[5pt]
c_0^F \in \bigcup_{\vartheta \in (0, 1)} C^{2 + \vartheta}(\overline{\Omega}) ~~\text{is such that}~~ c_0^F \geq 0~~\text{in $\overline{\Omega}$} ~~\text{and}~~ \partial_\nu c_0^F = 0 ~~\text{on}~~ \partial \Omega, \\[5pt]
b_0 \in \bigcup_{q > 2} W^{1, q}(\Omega) ~\text{is such that}~ b_0 \geq 0, \\[5pt]
m_0 \in \bigcup_{q > 2} W^{1, q}(\Omega) ~\text{is such that}~ m_0 \geq 0, \qquad \text{and} \\[5pt]
v_0 \in \bigcup_{\vartheta \in (0, 1)} C^{2 + \vartheta}(\overline{\Omega}) ~~\text{is such that}~~ v_0 \geq 0~~\text{in $\overline{\Omega}$} ~~\text{and}~~ \partial_\nu v_0 = 0 ~~\text{on}~~ \partial \Omega.
\end{cases}
\end{equation}
Our first main result concerns the global classical solvability of \eqref{model1}-\eqref{model3} in two spatial dimensions and will be established in Section~\ref{sec3}. It is stated as follows.
\begin{theorem}
    \label{2D_Class_Theorem}
     Let $\Omega \subset \mathbb{R}^2$ be a bounded domain with a sufficiently smooth boundary $\partial\Omega$. Let $\delta_1$, $\delta_2$, $\delta_3$, $\lambda_b$, $\lambda_F$, $\lambda_m$, $\mu_E$, $\mu_M$, 
     $\rho_b$, $\rho_E$, $\rho_F$, $\rho_m^E$, $\rho_m^F$, $\rho_v$, $\chi_F$ and $\chi_M$ be fixed positive constants. Suppose that \eqref{zeta} holds and that the initial data satisfy \eqref{ic_assum}. Assume, in addition, that 
     \begin{align*}
     \chi_E<2, \qquad \chi_M<2 \qquad \text{and} \qquad \chi_F<1.
    \end{align*}
  Then the system \eqref{model1}-\eqref{model3} admits a global classical solution.
\end{theorem}

Our second main result, proved in Section~\ref{sec4}, concerns the global classical solvability of \eqref{model3w}, \eqref{model2}, and \eqref{model3} in three spatial dimensions. In this setting, the EMT and MET rates in \eqref{model3w} are prescribed as positive constants $\alpha_1$ and $\alpha_2$, respectively, so that the dynamics of TGF-$\beta$ density $b$ no longer need to be incorporated into the system. The result reads as follows:
\begin{theorem}
    \label{3D_Class_Theorem}
    Let $\Omega\subset\mathbb{R}^3$ be a bounded domain with a sufficiently smooth boundary $\partial\Omega$. Let $\alpha_1$, $\alpha_2$, $\delta_1$, $\delta_2$, $\lambda_F$, $\lambda_m$, $\rho_E$, $\rho_F$, $\rho_m^E$, $\rho_m^F$, $\rho_v$, $\chi_E$, $\chi_F$, and $\chi_M$ be fixed positive constants, and assume that $\delta_3=0$. Suppose that the initial data satisfy \eqref{ic_assum}. Assume, in addition, that
   \begin{align*}
   \chi_E<\frac{\rho_E}{2 \delta_2} \qquad\text{and}\qquad \chi_F<\frac{\rho_F}{\rho_v}.
   \end{align*}
  Then the system \eqref{model3w}, \eqref{model2}, and \eqref{model3} admits a global classical solution.
\end{theorem}

The article is organized as follows. In Section \ref{sec2}, we develop the analytical preliminaries, including an equivalent reformulation of the model, a local existence result, an extensibility criterion, and basic a priori estimates for the solution components. In Section \ref{sec3} we deal with global classical solvability of  \eqref{model1}-\eqref{model3} in two spatial dimensions. Here, our approach relies on a Lyapunov-type functional yielding $L\log L$ estimates for the first three solution components under suitable smallness conditions on the haptotactic sensitivities. These estimates, combined with a modified Gagliardo-Nirenberg interpolation inequality \cite[Lemma A.5.]{tao2014energy}, provide the higher-order control required to derive $L^p \; (p > 1)$ bounds. A subsequent Moser-Alikakos iteration \cite{alikakos1979lp} yields $L^\infty$ estimates, thereby ensuring global classical solvability. This energy-based strategy is well established for two-dimensional haptotaxis systems \cite{dai2022new, tao2020global, li2021boundedness,tao2014energy,surulescu2021does}. 

The subsequent Section \ref{sec4} establishes global classical solution of \eqref{model3w}, \eqref{model2}, and \eqref{model3} in three spatial dimensions. In this setting, the energy-based approach used in the two-dimensional case is no longer available, as the modified Gagliardo-Nirenberg inequality \cite[Lemma A.5.]{tao2014energy} underlying that argument does not apply. We therefore adopt a coupled $L^p$-estimate technique \cite{pang2018global,dai2022new}. Under suitable restrictions on the model parameters, we first establish $L^2$ bounds for the first and third solution components and subsequently bootstrap them to arbitrary $L^p$ bounds for $p> 1$. These estimates, in turn, provide the necessary control to derive an $L^p$ bound for the second solution component. Finally, a Moser-Alikakos iteration yields $L^\infty$ bounds for all three components. Section \ref{sec5} presents numerical simulations illustrating the model dynamics. Finally, Section in \ref{sec6} we discuss both the analytical and numerical results and give some concluding remarks.

Before we begin, we state some notations and conventions that will be used throughout the sequel:
\begin{itemize}
  \item we write $\int_\Omega h(x)$ and $\int_0^T \int_\Omega h(x,t), ~ (T > 0)$ to denote $\int_\Omega h(x)\,dx$ and $\int_0^T \int_\Omega h(x,t)\,dx\,dt$, respectively;
  \item the numbering of the generic constants $A_i, C_i, K_i, \kappa_i, \Lambda_i, \Gamma_i > 0, ~~ i = 1, 2, \cdots$ is restricted to the proof in which they are used. 
\end{itemize}

\section{Preliminaries}
\label{sec2}
This section presents several preliminary analytical considerations that will be used in the sequel. We first introduce a suitable change of variables and derive an equivalent formulation of the model. We then present the local existence result together with a modified extensibility criterion. Finally, we derive several elementary yet crucial estimates that will be instrumental in the subsequent analysis. We emphasize that the results stated in this section remain valid in both two- and three-dimensional spatial domains.


\subsection{Change of variable and equivalent model}
Following a standard procedure in the analysis of haptotaxis systems 
\cite{walker2007global, tao2008global, tao2009global, marciniak2010boundedness, tao2011global, tao2014energy, pang2017global, ke2018note, giesselmann2018existence, tao2020global, dai2020global, li2021boundedness, surulescu2021does, dai2022new, fuest2024cross, TaoWinklerPreprint1, TaoWinkler2, pang2018global, Li2026Haptotaxis, zheng2022global}, we introduce the following substitutions, which transform the first three equation of \eqref{model1} into the divergence form:
\begin{align}
\label{cov}
   a = e^{- \chi_E v} c^E, \qquad  u = e^{- \chi_M v} c^M \qquad \text{and} \qquad w = e^{\chi_F v} c^F.
\end{align}
Using \eqref{cov} we rewrite the initial boundary-value problem in \eqref{model1}-\eqref{model3} in the following form
\begin{align}
    \label{sys1}
    \begin{cases}
        \partial_t a = e^{-\chi_E v} \nabla \cdot \big(e^{\chi_E v} \nabla a\big) + F_1\big(a, u, w, b, m, v\big), & \qquad x \in \Omega, t > 0, \\[7pt]
        \partial_t u = e^{-\chi_M v} \nabla \cdot \big(e^{\chi_M v} \nabla u\big) + F_2\big(a, u, w, b, m, v\big), & \qquad x \in \Omega, t > 0, \\[5pt]
        \partial_t w = e^{\chi_F v} \nabla \cdot \big(e^{-\chi_F v} \nabla w\big) + F_3\big(a, u, w, m, v\big), & \qquad x \in \Omega, t > 0, \\[5pt]
        \partial_t b = \Delta b -\lambda_{b}\,b + F_4\big(v, w\big), & \qquad x \in \Omega, t > 0, \\[5pt]
        \partial_t m = \Delta m - \lambda_m m + F_5\big(a, v, w\big), & \qquad x \in \Omega, t > 0, \\[5pt]
        \partial_t v = F_6\big(a, u, w, m, v\big), & \qquad x \in \Omega, t > 0, \\[5pt]
        \partial_\nu a = \partial_\nu u = \partial_\nu w = \partial_\nu b = \partial_\nu m = 0, & \qquad x \in \partial \Omega, t > 0, \\[5pt]
        a(x, 0) = c^E_0(x)e^{-\chi_E v_0(x)}, \; u(x, 0) = c^M_0(x)e^{-\chi_M v_0(x)}, \\[5pt]
        w(x, 0) = c^F_0(x)e^{\chi_F v_0(x)}, \; b(x, 0) = b_0(x), \;  m(x, 0) = m_0(x),  & \qquad x \in \Omega,
    \end{cases}
\end{align}
where, 
\begin{equation}
\label{functions}
    \begin{cases}
        F_1\big(a, u, w, b, m, v\big) := - \mu_E \zeta_1(b) a + \mu_M \zeta_2(b) e^{(\chi_M - \chi_E) v} u \\[7pt]
        \hspace{3.5cm} + \rho_E a\left(1 - e^{\chi_E v} a - e^{\chi_M v} u - e^{- \chi_F v} w - v\right)  ,  \\[7pt]
         \hspace{3.5cm} - \rho_v \chi_E e^{- \chi_F v} a w (1 - v)+ \delta_1 \chi_E a m v  \\[7pt]
         \hspace{3.5cm} + \delta_2 \chi_E e^{\chi_E v} a^2 v + \delta_3 \chi_E e^{\chi_M v} a u v,\\[7pt]
        F_2\big(a, u, w, b, m, v\big) :=  \mu_E \zeta_1(b) e^{(\chi_E -\chi_M) v} a -\mu_M \zeta_2(b)u +\delta_1 \chi_M mvu\\[5pt]
         \hspace{3.5cm} + \delta_2 \chi_M e^{\chi_E v} a u v + \delta_3\chi_M e^{\chi_M v} u^2 v -\rho_v \chi_Me^{-\chi_F v} uw \big(1-v\big), \\[7pt]
        F_3\big(a, u, w, m, v\big) := \rho_F(1+e^{\chi_E v}a)w\big(1-e^{\chi_E v}a-e^{\chi_M v}u-e^{-\chi_F v}w-v\big) \\[5pt]
        \hspace{3.2cm}-\lambda_F w -\chi_F\delta_1 mvw  -\chi_F\delta_2 e^{\chi_E v}avw \\[5pt]
        \hspace{3.2cm}-\chi_F\delta_3 e^{\chi_M v}uvw +\chi_F\rho_v e^{-\chi_F v}w^2\big(1-v\big), \\[7pt]
       \hspace{1.8cm} F_4\big(v, w\big) :=  \rho_{b}\,e^{-\chi_F v} w, \\[7pt]
        \hspace{1.2cm} F_5\big(a, v, w\big) := \rho_m^F \,e^{-\chi_F v} w + \rho_m^E e^{\chi_E v} a, \\[7pt]
        F_6\big(a, u, w, m, v\big) := \rho_v e^{-\chi_F v} w \big(1 - v\big) - \delta_1 m v - \delta_2 e^{\chi_E v} a v - \delta_3 e^{\chi_M v} uv.
    \end{cases}
\end{equation}
\begin{remark}
    As pointed out in \cite{tao2011global, tao2014energy, li2021boundedness, surulescu2021does}, the systems \eqref{model1}-\eqref{model3} and \eqref{sys1} are equivalent in the sense of classical solutions. Estimates that do not involve spatial derivatives can therefore be transferred directly via \eqref{cov}, since $a$ corresponds to $c^E$, $u$ to $c^M$, and $w$ to $c^F$ respectively. However, estimates involving gradient terms require additional care, as highlighted in \cite{surulescu2021does}. Throughout the paper, we will tactically switch between these variables using \eqref{cov}.
\end{remark}

\subsection{Local existence and a refined extensibility criterion}
Using well-established strategies for constructing solutions in haptotaxis systems \cite{tao2014energy, pang2017global}, we obtain the following result concerning the local existence of classical solutions to \eqref{sys1}.
\begin{lemma}
    \label{local_existence}
Let $\Omega \subset \mathbb{R}^N$, $N \in \{2,3\}$, be a bounded domain with sufficiently smooth boundary $\partial\Omega$. Let $\delta_1$, $\delta_2$, $\delta_3$, $\lambda_b$, $\lambda_F$, $\lambda_m$, $\mu_E$, $\mu_M$, $\rho_b$, $\rho_E$, $\rho_F$, $\rho_m^E$, $\rho_m^F$, $\rho_v$, $\chi_E$, $\chi_F$ and $\chi_M$ be fixed positive constants, and assume that \eqref{zeta} holds true and the initial data satisfy \eqref{ic_assum}. Then there exist $T_{\max}\in(0,\infty)$ and a classical solution $(a,u,w,b,m,v)$ to \eqref{sys1} in $\Omega\times(0,T_{\max})$ such that
\begin{equation}
\begin{cases}
\label{sol1}
a \in \bigcup_{\vartheta \in (0, 1)} C^{2+\vartheta,\,1+\frac{\vartheta}{2}}(\overline{\Omega}\times[0,T_{\max})), \\[5pt]
u \in \bigcup_{\vartheta \in (0, 1)} C^{2+\vartheta,\,1+\frac{\vartheta}{2}}(\overline{\Omega}\times[0,T_{\max})), \\[5pt]
w \in \bigcup_{\vartheta \in (0, 1)} C^{2+\vartheta,\,1+\frac{\vartheta}{2}}(\overline{\Omega}\times[0,T_{\max})), \\[5pt]
b \in \bigcup_{q > 2}  C^0\big([0,T_{\max});W^{1,q}(\Omega)\big)
\cap \bigcup_{\vartheta \in (0, 1)} C^{2+\vartheta,\,1+\frac{\vartheta}{2}}(\overline{\Omega}\times(0,T_{\max})),   \\[5pt]
m \in \bigcup_{q > 2}  C^0\big([0,T_{\max});W^{1,q}(\Omega)\big)
\cap \bigcup_{\vartheta \in (0, 1)} C^{2+\vartheta,\,1+\frac{\vartheta}{2}}(\overline{\Omega}\times(0,T_{\max})) \qquad \text{and} \\[5pt]
v \in \bigcup_{\vartheta \in (0, 1)} C^{2+\vartheta,\,1+\frac{\vartheta}{2}}(\overline{\Omega}\times[0,T_{\max})).
\end{cases}
\end{equation}
Moreover, the components $a,u,w,b,m,v$ are nonnegative in $\Omega\times(0,T_{\max})$, and if $T_{\max}<\infty$, then for all $\vartheta\in(0,1)$,
\begin{align}
\label{extensibility_criterion1}
\limsup_{t\nearrow T_{\max}}
\Big(
\nonumber \|a(\cdot,t)\|_{C^{2+\vartheta}(\overline{\Omega})}
& +\|u(\cdot,t)\|_{C^{2+\vartheta}(\overline{\Omega})}
+\|w(\cdot,t)\|_{C^{2+\vartheta}(\overline{\Omega})} \\
&+\|b(\cdot,t)\|_{C^{2+\vartheta}(\overline{\Omega})}
+\|m(\cdot,t)\|_{C^{2+\vartheta}(\overline{\Omega})}
+\|v(\cdot,t)\|_{C^{2+\vartheta}(\overline{\Omega})}
\Big)
= \infty.
\end{align}
\end{lemma}
\begin{proof}
For the proof of the local existence of classical solutions with regularity as specified in \eqref{sol1}, together with the extensibility criterion \eqref{extensibility_criterion1}, we refer to \cite[Lemma~2.1]{tao2014energy}.
\end{proof}

Lemma \ref{local_existence} ensures the local existence of solutions to the original problem \eqref{model1}-\eqref{model3} in the following sense.
\begin{lemma}
\label{local_existence2}
Let $T_{\max}$ and $(a, u, w, b, m, v)$ be as in Lemma~\ref{local_existence}, and define
\begin{align*}
c^E = e^{\chi_E v} a, \qquad   c^M = e^{\chi_M v} u \qquad \text{and} \qquad c^F = e^{-\chi_F v} w . 
\end{align*}
Then
\begin{align*}
\begin{cases}
c^E \in C^{2, 1}\left(\overline{\Omega} \times [0, T_{\max})\right), \\[5pt]
c^M \in C^{2, 1}\left(\overline{\Omega} \times [0, T_{\max})\right), \\[5pt]
c^F \in C^{2, 1}\left(\overline{\Omega} \times [0, T_{\max})\right), \\[5pt]
b \in \bigcup_{q > 2}  C^0\big([0,T_{\max});W^{1,q}(\Omega)\big)
\cap C^{2, 1}(\overline{\Omega}\times(0,T_{\max})), \\[5pt]
m \in \bigcup_{q > 2}  C^0\big([0,T_{\max});W^{1,q}(\Omega)\big)
\cap C^{2, 1}(\overline{\Omega}\times(0,T_{\max})) \qquad \text{and}  \\[5pt]
v \in C^{2, 1}\left(\overline{\Omega} \times [0, T_{\max})\right).
\end{cases}
\end{align*}
and $(c^E,c^M,c^F,b,m,v)$ solves \eqref{model1}-\eqref{model3} in the classical sense in $\Omega \times (0,T_{\max})$.
\end{lemma}
\begin{proof}
    This follows directly from \eqref{sys1}-\eqref{functions}.
\end{proof}
In a recent contribution, Tao and Winkler established in \cite[Proposition~3.1]{TaoWinklerPreprint1} a relaxed version of the extensibility criterion in \eqref{extensibility_criterion1}. This result yields the following lemma, a crucial ingredient in the proof of global existence of classical solutions to \eqref{model1}-\eqref{model3} in both two- and three-dimensional cases.

%
\begin{lemma}
\label{extensibility_imp}
Let \eqref{ic_assum} hold. Then the solution $(a,u,w,b,m,v)$ of \eqref{sys1} from Lemma~\ref{local_existence} has the following extensibility alternative: either $T_{\max}=\infty$, or
\begin{align}
\label{extensibility_criterion2}
\limsup_{t\nearrow T_{\max}}
\Big(
\|a(\cdot,t)\|_{L^\infty(\Omega)} 
+
\|u(\cdot,t)\|_{L^\infty(\Omega)}
+
\|w(\cdot,t)\|_{L^\infty(\Omega)}
\Big)
= \infty.
\end{align}
\end{lemma}

\begin{remark}
 Lemma \ref{local_existence}, and consequently Lemma \ref{local_existence2} as well as the modified extensibility criterion stated in Lemma \ref{extensibility_imp}, remain valid for \eqref{model3w}, \eqref{model2}, and \eqref{model3} when $\alpha_1,\alpha_2>0$. Therefore, all subsequent results established in this section remain valid for  \eqref{model3w}, \eqref{model2}, and \eqref{model3}.
\end{remark}

\subsection{Some basic estimates}
We are now in a position to collect several elementary regularity properties of the solution components that will be crucial in the subsequent analysis.
\begin{lemma}
\label{lemmav_bound}
We have the estimate
\begin{align}
\label{v_es1}
0 \leq v(x,t) \leq 1 
\qquad \text{for all } x\in\Omega \text{ and } t\in(0,T_{\max}).
\end{align}
\end{lemma}
\begin{proof}
 \eqref{v_es1} follows directly from the non-negativity of the solution components, the assumption $0 < v_0 \leq 1$, and an ODE comparison argument applied to the sixth equation in \eqref{sys1}.
\end{proof}

The next lemma provides $L^1$ estimates for the first three solution components, together with space-time $L^2$ estimates for $c^E$ and $c^F$.
\begin{lemma}
    \label{lemma:l1bound}
    Let $T > 0$. Then there exists constants  $M_1(T) > 0$ and $M_2(T) > 0$ such that
    \begin{align}
    \label{L1ess}
    \int_\Omega c^i(\cdot,t)\leq M_1(T)
    \qquad \text{for all } i\in\{E,M,F\},\quad t\in(0,\widehat T),
\end{align}
    and
    \begin{align}
\label{sqc}
\int_0^{\widehat{T}}\int_\Omega (c^j)^2
\leq M_2(T)
\qquad \text{for all } j\in\{E,F\}.
\end{align}
    where $\widehat{T} := \min\{T, T_{\max}\}$. 
\end{lemma}
\begin{proof}
   We integrate the first three equations of \eqref{model1} over $\Omega$ and add them. Owing to the non-negativity of the solution components, this yields the following differential inequality:
   \begin{align}
       \label{L11}
       \frac{d}{dt} \left\{\int_\Omega c^E + \int_\Omega c^M + \int_\Omega c^F\right\} + \rho_E \int_\Omega \left(c^E\right)^2 +  \rho_F \int_\Omega \left(c^F\right)^2 \leq \rho \left\{\int_\Omega c^E + \int_\Omega c^M + \int_\Omega c^F\right\}
   \end{align}
   for all $t \in (0, T_{\max})$, where $\rho = \max\{\rho_E, 1, \rho_F\}$. Dropping the non-negative quadratic terms from the left-hand side of \eqref{L11}, we infer that
\begin{align}
\label{L111}
\frac{d}{dt}\left\{\int_\Omega c^E +\int_\Omega c^M +\int_\Omega c^F\right\} \leq \rho \left\{\int_\Omega c^E +\int_\Omega c^M +\int_\Omega c^F\right\}
\end{align}
for all $t\in(0,T_{\max})$. An application of Gronwall's inequality to \eqref{L111} on $(0,\widehat T)$ therefore yields
\begin{align*}
\int_\Omega c^E(\cdot,t) +\int_\Omega c^M(\cdot,t) +\int_\Omega c^F(\cdot,t)
&\leq e^{\rho t} \left\{\int_\Omega c^E_0 +\int_\Omega c^M_0 +\int_\Omega c^F_0\right\}\\[5pt]
&\leq e^{\rho T}\left\{\int_\Omega c^E_0+\int_\Omega c^M_0+\int_\Omega c^F_0 \right\}
\end{align*}
for all $t\in(0,\widehat T)$. Since $c^E$, $c^M$ and $c^F$ are non-negative, each individual mass is bounded by the total mass. Hence, upon defining
\begin{align*}
M_1(T) := e^{\rho T} \left\{\int_\Omega c^E_0 +\int_\Omega c^M_0 +\int_\Omega c^F_0\right\},
\end{align*}
we directly obtain \eqref{L1ess}. Set, 
$$y_1(t):=\sum_{i\in\{E,M,F\}}\int_\Omega c^i(\cdot,t).$$
Integrating \eqref{L11}, over $(0,\widehat{T})$ in view of \eqref{L1ess}, we further obtain 
\begin{align*} 
\rho_E \int_0^{\widehat{T}}\int_\Omega (c^E)^2 + \rho_F \int_0^{\widehat{T}}\int_\Omega (c^F)^2 & \leq  y_1(0) +\rho\int_0^{\widehat{T}}y_1(t)\,dt\\[5pt]
&\leq  y_1(0)+\rho \widehat{T} M_1(T) \\[5pt]
& \leq  y_1(0)+\rho T M_1(T) 
\end{align*} 
Consequently,
\begin{align*}
\int_0^{\widehat{T}}\int_\Omega (c^E)^2 \leq \frac{y_1(0)+\rho T M_1(T)}{\rho_E}, \qquad
\int_0^{\widehat{T}}\int_\Omega (c^F)^2 \leq \frac{y_1(0)+\rho T M_1(T)}{\rho_F}.
\end{align*}
Hence, upon defining
\begin{align*}
M_2(T):= \max\left\{\frac{y_1(0)+\rho T M_1(T)}{\rho_E}, \frac{y_1(0)+\rho T M_1(T)}{\rho_F}\right\},
\end{align*}
we conclude \eqref{sqc}
\end{proof}

As a consequence of Lemma \ref{lemma:l1bound} and \eqref{v_es1} we have the following corollary. 
\begin{corollary}
\label{L1_z}
 Let $T > 0$. Then there exist constants $a^\ast(T),u^\ast(T), w^\ast(T) > 0$ such that
\begin{equation}
\label{zL1_bound}
\int_\Omega a(\cdot,t) \leq a^\ast(T), \quad 
\int_\Omega u(\cdot,t) \leq u^\ast(T) 
\quad \text{and} \quad 
\int_\Omega w(\cdot,t) \leq w^\ast(T) 
\quad \text{for all } t \in (0,\widehat{T}),
\end{equation}
where $\widehat{T} := \min\{T,T_{\max}\}$.
\end{corollary}
\begin{proof}
Using the definitions of $a$, $u$ and $w$ from \eqref{cov}, together with \eqref{v_es1}, we obtain
\begin{align*}
& \int_\Omega a =  \int_\Omega e^{-\chi_E v} c^E \leq \int_\Omega c^E \leq M_1(T),\\[5pt]
& \int_\Omega u 
= \int_\Omega e^{-\chi_M v}\, c^M 
\leq \int_\Omega c^M  \leq M_1(T) \qquad \text{and}  \\[5pt]
& \int_\Omega w 
= \int_\Omega e^{\chi_F v}\, c^F 
\leq e^{\chi_F} \int_\Omega c^F  \leq e^{\chi_F} M_1(T)
\end{align*}
for all $t\in(0, \widehat{T})$.  Hence, \eqref{zL1_bound} holds with 
\begin{align*}
    a^\ast(T) := M_1(T), \qquad
u^\ast(T):=M_1(T)
\qquad \text{and} \qquad
w^\ast(T) :=e^{\chi_F}M_1(T).
\end{align*}
\end{proof}
Next, we establish an $L^1$ estimate for the fifth solution component. 
\begin{lemma}
\label{lemma_mess}
Let $T>0$. Then there exists a constant $M_m(T)>0$ such that
\begin{align}
\label{mes1}
\int_\Omega m(\cdot,t)\leq M_m(T)
\qquad \text{for all }t\in(0,\widehat{T}),
\end{align}
where $\widehat{T} := \min\{T,T_{\max}\}$.
\end{lemma}
\begin{proof}
Integrating the fifth equation of \eqref{model1} over $\Omega$ and using \eqref{L1ess}, we obtain
\begin{align*}
\frac{d}{dt}\int_\Omega m+\lambda_m\int_\Omega m = \rho_m^F\int_\Omega c^F +\rho_m^E\int_\Omega c^E \leq (\rho_m^F+\rho_m^E)M_1(T)
\end{align*}
for all $t\in(0,T_{\max})$. An application of the comparison principle for ODEs yields
\begin{align*}
    \int_\Omega m(\cdot,t) \leq \max\left\{\int_\Omega m_0,\,\frac{(\rho_m^F+\rho_m^E)M_1(T)}{\lambda_m}\right\}=:M_m(T)
\end{align*}
for all $t \in (0, \widehat{T})$. This establishes \eqref{mes1}.
\end{proof}

We also have some additional estimates for $m$ presented in the following lemma. 
\begin{lemma}
  \label{lemma_mbound1w}
   There exist a positive constant $M_2^\ast(T)$ such that
  \begin{align}
   \label{lemma_mnab}
    \int_\Omega |\nabla m(\cdot, t)|^2 \leq M_2^\ast(T), \quad \text{for all}~ t \in (0, \widehat{T}),
  \end{align}
  and
  \begin{equation}
    \label{lemma_mdel}
    \int_{0}^{\widehat{T}} \int_\Omega |\Delta m|^2 \leq M_2^\ast(T)
  \end{equation}
  where $\widehat{T} := \min\{T,T_{\max}\}$.
\end{lemma}
\begin{proof}
  We test the fifth equation of \eqref{model1} with $- \Delta m$ and apply Young's inequality to have
  \begin{align}
  \label{mnab1}
    \frac{1}{2} \frac{d}{dt} \int_\Omega |\nabla m|^2 + \frac{1}{2} \int_\Omega |\Delta m|^2 + \lambda_m \int_\Omega |\nabla m|^2 \leq (\rho_m^F)^2\int_\Omega (c^F)^2 + (\rho_m^E)^2\int_\Omega (c^E)^2
  \end{align}
  for all $t \in (0, T_{\max})$. From \eqref{mnab1} we observe that $y_2(t) = \int_\Omega |\nabla m(\cdot, t)|^2, t \in (0, T_{\max})$ satisfies the following ordinary differential inequality:
  \begin{align}
  \label{m233}
      y_2'(t) + 2 \lambda_m y_2(t) \leq f(t) := 2(\rho_m^F)^2\int_\Omega (c^F)^2 + 2(\rho_m^E)^2\int_\Omega (c^E)^2.
  \end{align}
 Thanks to Lemma \ref{lemma:l1bound} we have a constant $C^\ast > 0$ such that
  \begin{align*}
     \int_0^{\widehat{T}} f(s) \leq C^\ast \quad \text{for all} ~ t \in (0, \widehat{T}).
  \end{align*}
Consequently, an application of \cite[Lemma 3.4]{stinner2014global} to \eqref{m233} ensures that there exists a constant $M_2(T)^\ast>0$ such that
  \begin{equation*}
    \int_\Omega |\nabla m(\cdot, t)|^2 \leq  \max \left\{\int_\Omega |\nabla m_0|^2 + C^\ast, \; \frac{C^\ast}{ 2\lambda_m} + 2 C^\ast\right\} =: M_2^\ast(T) \quad \text{for all}~ t \in (0, \widehat{T}),
  \end{equation*}
  this yields \eqref{lemma_mnab}. \eqref{lemma_mdel} can then be directly obtained by an integration of \eqref{mnab1} in view of \eqref{lemma_mnab}.
\end{proof}
In the sequel, we will frequently use the notations
\begin{align*}
e^{\chi_E} := \Lambda_E, \qquad e^{\chi_M}:=\Lambda_M
\qquad \text{and} \qquad
e^{\chi_F}:=\Lambda_F.
\end{align*}


\section{Global classical solutions to \texorpdfstring{\eqref{sys1}}{the system} in a two-dimensional domain}
\label{sec3}
This section aims to establish global classical solutions to \eqref{model1}-\eqref{model3} in a two-dimensional spatial domain. To this end, we derive a Lyapunov-type functional involving the logarithmic entropy of $c^E$, $c^M$, and $c^F$. This functional plays a crucial role in establishing an $L^\infty$ estimate for these components via a modified Gagliardo-Nirenberg interpolation inequality \cite[Lemma A.5]{tao2014energy}. 

We begin by invoking the following lemma, previously employed in the analysis of oncolytic virotherapy models with positive haptotaxis \cite[Lemma 5.2]{tao2020global}, \cite[Lemma 2.1]{li2021boundedness}. The lemma may be viewed as a special case of \cite[Lemma A.5]{tao2014energy}.
\begin{lemma}\textnormal{(\cite[Lemma 5.2]{tao2020global}; \cite[Lemma 2.1]{li2021boundedness}).}
\label{Lemma_2EDLOg}
Let $\Omega \subset \mathbb{R}^2$ be a bounded domain with smooth boundary, and let $p \in (1,\infty)$ and $\varepsilon > 0$. Then there exists a constant $K(p,\varepsilon) > 0$ such that
\[
\|\varphi\|_{L^{\frac{2(p+1)}{p}}(\Omega)}^{\frac{2(p+1)}{p}}
\leq \varepsilon \|\nabla \varphi\|_{L^2(\Omega)}^2 \cdot \int_{\Omega} |\varphi|^{\frac{2}{p}} |\ln |\varphi||
+ K(p,\varepsilon)\left( \|\varphi\|_{L^{\frac{2}{p}}(\Omega)}^{\frac{2(p+1)}{p}} + 1 \right)
\]
for all $\varphi \in W^{1,2}(\Omega)$.
\end{lemma}
We begin the analysis by stating an important result concerning the fifth solution component.
\begin{lemma}
\label{lemma_mboundpq}
Let $T>0$ and $p>1$. Then there exists a constant $C(p,T)>0$ such that
\begin{align}
\label{mp1}
\int_\Omega m^p(\cdot,t)\leq C(p,T)
\qquad \text{for all }t\in(0,\widehat{T}),
\end{align}
where $\widehat{T}:=\min\{T,T_{\max}\}$.
\end{lemma}
\begin{proof}
  Since $W^{1, 2}(\Omega) \hookrightarrow L^p(\Omega)$ (note that $\Omega$ in this section is a two-dimensional domain),  \eqref{mp1} follows directly by combining the $L^1$ bound for $m$ established in \eqref{mes1} with \eqref{lemma_mnab}.
\end{proof}

Similar arguments to those employed for the fifth equation in \eqref{model1} for $m$, when applied to the fourth equation for $b$, yield the following estimates.
\begin{lemma}
  \label{lemma_mbound1b}
 Let $T>0$. Then, for any $p\geq 1$, there exist constants $C(p,T)>0$ and $C(T)>0$ such that
\begin{align*}
    \int_\Omega b^p(\cdot,t) \leq C(p,T)
    \qquad \text{for all } t\in(0,\widehat{T}),
\end{align*}
\begin{align*}
    \int_\Omega |\nabla b(\cdot,t)|^2 \leq C(T)
    \qquad \text{for all } t\in(0,\widehat{T}),
\end{align*}
and
\begin{align*}
    \int_0^{\widehat{T}}\int_\Omega |\Delta b|^2 \leq C(T),
\end{align*}
where $\widehat{T}:=\min\{T,T_{\max}\}$.
\end{lemma}


We next establish several auxiliary lemmas that are essential for the development of an entropy-type functional for \eqref{model1}.

\begin{lemma}
    \label{cwwe_log11}
    For every $\eta > 0$ the following holds true:
    \begin{align}
        \label{celog_main}
        \nonumber \frac{d}{dt} \int_\Omega c^E \ln c^E & + \int_\Omega \frac{\left|\nabla c^E\right|^2}{c^E}  \leq  \chi_E \int_\Omega \nabla c^E \cdot \nabla v + \frac{\rho_E}{2}\int_\Omega (c^E)^2  \\[5pt]
        & + \frac{ \eta \mu_M \zeta^\ast}{2}\int_\Omega (c^M)^2 + B_1 \int_\Omega c^E + \left(\frac{\rho_E}{e} +  \mu_M \zeta^\ast\right) \int_\Omega c^M + \frac{\rho_E}{e} \int_\Omega c^F + B_2
    \end{align}
   for all $t\in(0,T_{\max})$ where $B_1 := \frac{2}{e^2} \left(\frac{\mu_M \zeta^\ast}{\eta} + \rho_E\right) + \rho_E \left(1 + \frac{1}{e}\right)$ and $B_2 := \frac{1}{e}\left(\mu_E \zeta^\ast + \rho_E \right)|\Omega|$.
\end{lemma}
\begin{proof}
    Owing to the positivity of the solution components, we can directly obtain from the $c^E$ equation
    \begin{align}
        \label{cel12}
        \nonumber & \frac{d}{dt} \int_\Omega c^E \ln c^E  = \int_\Omega \partial_t c^E \ln c^E + \int_\Omega \partial_t c^E  \\[5pt]
        \nonumber & = - \int_\Omega \frac{\left|\nabla c^E\right|^2}{c^E} + \chi_E \int_\Omega \nabla c^E \cdot \nabla v  + \int_\Omega \left\{-\mu_{E}\,\zeta_1(b)\,c^{E} + \mu_{M}\,\zeta_2(b)\,c^{M}\right\}\ln c^E \\[5pt]
        & + \int_\Omega \left\{\rho_{E}\,c^{E}\big(1-c^{E}-c^{M}-c^{F}-v\big)\right\} \ln c^E + \int_\Omega \partial_t c^E 
    \end{align}
    for all $t \in (0, T_{\max})$. An application of Young's inequality, together with \eqref{zeta} and the inequality $(\ln s)^2 \leq \frac{4}{e^2}s$ for $s>0$, yields for any $\eta > 0$
    \begin{align}
        \label{ce_alg12}
        \nonumber \mu_M \int_\Omega \zeta_2(b) c^M \ln c^E \leq \mu_M \int_{\{c^E > 1\}} \zeta_2(b) c^M \ln c^E & \leq \mu_M \zeta^\ast \left\{\frac{\eta}{2} \int_\Omega (c^M)^2 + \frac{1}{2\eta} \int_{\{c^E > 1\}} (\ln c^E)^2\right\} \\[5pt]
        & \leq \frac{\mu_M \zeta^\ast}{2} \left\{\eta \int_\Omega (c^M)^2 + \frac{4}{\eta e^2} \int_\Omega c^E\right\}
    \end{align}
   for all $t \in (0, T_{\max})$.  Using the inequality $-s\ln s\le\frac{1}{e}$ for $s>0$, together with \eqref{zeta} and \eqref{v_es1}, we obtain
   \begin{align}
       \label{alh32}
   - \mu_E \int_\Omega \zeta_1(b) c^E \ln c^E  = \mu_E \int_\Omega \zeta_1(b) (- c^E \ln c^E) \leq \frac{\mu_E}{e} \int_\Omega \zeta_1(b)  \leq \frac{\mu_E}{e} \zeta^\ast  |\Omega|
   \end{align}
  for all $t \in (0, T_{\max})$. Using the logarithmic inequalities mentioned above together with \eqref{v_es1}, we obtain the following estimates:
  \begin{align}
      \label{alsg3}
      & \rho_E \int_\Omega c^E \ln c^E \leq \frac{\rho_E}{2} \left\{\int_\Omega (c^E)^2 + \int_{\{c^E > 1\}} (\ln c^E)^2\right\} \leq \frac{\rho_E}{2} \left\{\int_\Omega (c^E)^2 + \frac{4}{e^2} \int_\Omega c^E\right\}, \\[5pt]
      \label{alsg4}
      & - \rho_E \int_\Omega (c^E)^2 \ln c^E = \rho_E \int_\Omega c^E (- c^E \ln c^E) \leq \frac{\rho_E}{e} \int_\Omega c^E, \\[5pt]
      \label{alsg77}
      & - \rho_E \int_\Omega c^M c^E \ln c^E = \rho_E \int_\Omega c^M (-c^E \ln c^E) \leq \frac{\rho_E}{e} \int_\Omega c^M, \\[5pt]
      \label{alsg99}
      & -\rho_E \int_\Omega c^F c^E \ln c^E = \rho_E \int_\Omega c^F (-c^E \ln c^E) \leq \frac{\rho_E}{e} \int_\Omega c^F, \\[5pt]
      \label{alsg95}
      &  -\rho_E \int_\Omega v c^E \ln c^E = \rho_E \int_\Omega v (-c^E \ln c^E) \leq \frac{\rho_E}{e} |\Omega| 
  \end{align}
    for all $t \in (0, T_{\max})$. Next, we estimate $\int_\Omega \partial_t c^E$. Owing to the nonnegativity of the solution components and \eqref{zeta}, we have
    \begin{align}
    \label{reasc}
   \nonumber \int_\Omega \big\{-\mu_{E}\,\zeta_1(b)\,c^{E} + \mu_{M}\,\zeta_2(b)\,c^{M} & +\rho_{E}\,c^{E}\big(1-c^{E} -c^{M}-c^{F}-v\big)\big\} \\[5pt]
   & \leq \mu_M \zeta^\ast \int_\Omega c^M + \rho_E \int_\Omega c^E
    \end{align}
   for all $t \in (0, T_{\max})$.  Inserting \eqref{ce_alg12}-\eqref{reasc} into \eqref{cel12} results in \eqref{celog_main}.
\end{proof}
\begin{lemma}
    \label{cm_log11}
      The following holds true:
    \begin{align}
        \label{cmlog_main}
        \nonumber \frac{d}{dt} \int_\Omega c^M \ln c^M & + \int_\Omega \frac{\left|\nabla c^M\right|^2}{c^M}  \leq  \chi_M \int_\Omega \nabla c^M \cdot \nabla v + \frac{\mu_E \zeta^\ast}{2} \int_\Omega (c^E)^2 \\[5pt] 
        &  + \mu_E \zeta^\ast \int_\Omega c^E+ \frac{2\mu_E \zeta^\ast }{e^2} \int_\Omega c^M + \frac{\mu_M}{e} \zeta^\ast |\Omega|
    \end{align}
   for all $t\in(0,T_{\max})$.
\end{lemma}
\begin{proof}
    Owing to the positivity of solution components, we can directly have from the $c^M$ equation
    \begin{align}
        \label{cml12}
        \nonumber & \frac{d}{dt} \int_\Omega c^M \ln c^M  = \int_\Omega \partial_t c^M \ln c^M + \int_\Omega \partial_t c^M   \\[5pt]
        & = - \int_\Omega \frac{\left|\nabla c^M\right|^2}{c^M} + \chi_M \int_\Omega \nabla c^M \cdot \nabla v  + \int_\Omega \left\{\mu_{E}\,\zeta_1(b)\,c^{E} - \mu_{M}\,\zeta_2(b)\,c^{M}\right\}\ln c^M + \int_\Omega \partial_t c^M
    \end{align}
    for all $t \in (0, T_{\max})$.  An application of Young's inequality, together with \eqref{zeta} and the inequality $(\ln s)^2 \leq \frac{4}{e^2}s$ for $s>0$, yields
    \begin{align}
        \label{cm_alg12}
        \nonumber \mu_E \int_\Omega \zeta_1(b) c^E \ln c^M \leq \mu_E \int_{\{c^M > 1\}} \zeta_1(b) c^E \ln c^M & \leq \mu_E \zeta^\ast \left\{\frac{1}{2} \int_\Omega (c^E)^2 + \frac{1}{2} \int_{\{c^M > 1\}} (\ln c^M)^2\right\} \\[5pt]
        & \leq \frac{\mu_E \zeta^\ast}{2} \left\{ \int_\Omega (c^E)^2 + \frac{4}{e^2} \int_\Omega c^M\right\} 
    \end{align}
   for all $t \in (0, T_{\max})$.  Using the inequality $-s\ln s\le\frac{1}{e}$ for $s>0$, together with \eqref{zeta} and \eqref{v_es1}, we obtain
   \begin{align}
       \label{mlh32}
   - \mu_M \int_\Omega \zeta_2(b) c^M \ln c^M  = \mu_M \int_\Omega \zeta_2(b) (- c^M \ln c^M) \leq \frac{\mu_M}{e} \int_\Omega \zeta_2(b)  \leq \frac{\mu_M}{e} \zeta^\ast  |\Omega|,
   \end{align}
  for all $t \in (0, T_{\max})$. Next, we estimate $\int_\Omega \partial_t c^M$. Owing to the nonnegativity of the solution components and \eqref{zeta}, we have
    \begin{align}
    \label{remsc}
   \int_\Omega \left\{\mu_{E}\,\zeta_1(b)\,c^{E} - \mu_{M}\,\zeta_2(b)\,c^{M}\right\}\leq \mu_E \zeta^\ast \int_\Omega c^E
    \end{align}
   for all $t \in (0, T_{\max})$.  Inserting \eqref{cm_alg12}-\eqref{remsc} into \eqref{cml12} results in \eqref{cmlog_main}. 
\end{proof}
\begin{lemma}
    \label{wwloge}
    For every $\eta > 0$ the following holds true:
    \begin{align}
        \label{cwlog_main}
        \nonumber \frac{d}{dt} \int_\Omega c^F \ln c^F & + \int_\Omega \frac{\left|\nabla c^F\right|^2}{c^F} + \chi_F \int_\Omega \nabla c^F \cdot \nabla v \leq R_1 \int_\Omega (c^E)^2 + \frac{\eta\rho_F}{2e}\int_\Omega(c^M)^2\\[5pt]
        & + R_2 \int_\Omega (c^F)^2 + \frac{\rho_F}{e}\int_\Omega c^E +\frac{\rho_F}{e}\int_\Omega c^M + R_3 \int_\Omega c^F +\frac{\rho_F+\lambda_F}{e}|\Omega|
    \end{align}
   for all $t\in(0,T_{\max})$ where $R_1 := \frac{\rho_F}{2}(1 + \frac{3}{e} + \frac{1}{\eta e})$, $R_2 := \rho_F + \frac{\rho_F}{2e}$, $R_3 := \rho_F+\frac{\rho_F}{e}+\frac{2\rho_F}{e^2}$.
\end{lemma}
\begin{proof}
    Owing to the positivity of the solution components, we directly obtain from the $c^F$ equation
    \begin{align}
        \label{cwl12}
        \nonumber & \frac{d}{dt} \int_\Omega c^F \ln c^F = \int_\Omega \partial_t c^F \ln c^F + \int_\Omega \partial_t c^F   \\[5pt]
        \nonumber & = - \int_\Omega \frac{\left|\nabla c^F\right|^2}{c^F} - \chi_F \int_\Omega \nabla c^F \cdot \nabla v  + \int_\Omega \partial_t c^F   \\[5pt]
        & + \int_\Omega \left\{\rho_{F}\,(1+c^{E})\,c^{F}\big(1-c^{E}-c^{M}-c^{F}-v\big) -\lambda_{F}\,c^{F}\right\} \ln c^F 
    \end{align}
    for all $t \in (0, T_{\max})$.  Using Young's inequality and the logarithmic inequalities mentioned before, together with \eqref{v_es1}, results in the following
    \begin{align}
        \label{ff1}
        & \rho_F \int_\Omega c^F \ln c^F \leq  \frac{\rho_F}{2} \left\{\int_\Omega (c^F)^2 + \frac{4}{e^2} \int_\Omega c^F\right\}, \\[5pt]
        & - \rho_F \int_\Omega c^F c^M \ln c^F = \rho_F \int_\Omega c^M \left(-c^F \ln c^F\right) \leq \frac{\rho_F}{e} \int_\Omega c^M, \\[5pt]
        & - \rho_F \int_\Omega (c^F)^2 \ln c^F =  \rho_F \int_\Omega c^F \left(- c^F \ln c^F\right) \leq \frac{\rho_F}{e} \int_\Omega c^F, \\[5pt]
        & - \rho_F \int_\Omega v c^F \ln c^F = \rho_F \int_\Omega v \left(- c^F \ln c^F\right) \leq \frac{\rho_F}{e} |\Omega|, \\[5pt]
       & - \rho_F \int_\Omega c^E c^F c^M \ln c^F = \rho_F \int_\Omega (c^E c^M ) \left(- c^F \ln c^F\right) \leq \frac{\rho_F}{2e} \left\{\frac{1}{\eta}\int_\Omega (c^E)^2 + \eta \int_\Omega (c^M)^2\right\}, \\[5pt]
        & - \rho_F \int_\Omega (c^E)^2 c^F \ln c^F = \rho_F \int_\Omega (c^E)^2 \left(- c^F \ln c^F\right) \leq \frac{\rho_F}{e} \int_\Omega (c^E)^2,  \\[5pt]
        & - \rho_F \int_\Omega c^E (c^F)^2 \ln c^F = \rho_F \int_\Omega \left(c^E c^F \right) \left(- c^F \ln c^F\right) \leq \frac{\rho_F}{2e} \left\{\int_\Omega (c^E)^2 + \int_\Omega (c^F)^2\right\}, \\[5pt]
        & - \rho_F \int_\Omega c^E c^F v \ln c^F = \rho_F \int_\Omega c^E v \left(- c^F \ln c^F\right) \leq \frac{\rho_F}{e} \int_\Omega c^E, \\[5pt]
         \label{ff2}
        & - \lambda_F \int_\Omega c^F \ln c^F = \lambda_F \int_\Omega \left(- c^F \ln c^F\right) \leq \frac{\lambda_F}{e} |\Omega|
    \end{align}
      for all $t \in (0, T_{\max})$. Next, we estimate $\int_\Omega \partial_t c^F$. Owing to the nonnegativity of the solution components and Young's inequality, we have
    \begin{align}
    \label{refsc}
     \nonumber \int_\Omega \left\{\rho_{F}\,(1+c^{E})\,c^{F}\big(1-c^{E}-c^{M}-c^{F}-v\big) -\lambda_{F}\,c^{F}\right\} &\leq \rho_F \int_\Omega (1 + c^E)c^F \\[5pt]
    &\leq \rho_F \int_\Omega c^F + \frac{\rho_F}{2} \left\{\int_\Omega (c^E)^2 + \int_\Omega (c^F)^2 \right\}
    \end{align}
   for all $t \in (0, T_{\max})$.  Inserting \eqref{ff1}-\eqref{refsc} into \eqref{cwl12} results in \eqref{cwlog_main}. 
\end{proof}
\begin{lemma}
    \label{mloge}
   The following holds true:
    \begin{align}
        \label{mlog_main}
         \nonumber \frac{d}{dt} \int_\Omega m \ln m & + \int_\Omega \frac{\left|\nabla m\right|^2}{m} \leq \frac{\rho_m^F}{2} \int_\Omega (c^F)^2 +  \frac{\rho_m^E}{2} \int_\Omega (c^E)^2  \\[5pt]
         & + \rho_{m}^{F} \int_\Omega c^F + \rho_{m}^{E} \int_\Omega c^E + \frac{2}{e^2}(\rho_{m}^{F} + \rho_{m}^{E}) \int_\Omega m + \frac{\lambda_m}{e}|\Omega|
    \end{align}
   for all $t\in(0,T_{\max})$.
\end{lemma}
\begin{proof}
    Owing to the positivity of the solution components, we can directly obtain from the $m$ equation
    \begin{align}
        \label{ml12}
        \nonumber & \frac{d}{dt} \int_\Omega m \ln m  = \int_\Omega \partial_t m \ln m + \int_\Omega \partial_t m \\[5pt]
        & = - \int_\Omega \frac{\left|\nabla m\right|^2}{m} + \int_\Omega \left\{\rho_{m}^{F}\,c^{F}+\rho_{m}^{E}\,c^{E}-\lambda_{m}\,m\right\} \ln m + \int_\Omega \partial_t m 
    \end{align}
    for all $t \in (0, T_{\max})$.  Using Young's inequality and the logarithmic inequalities mentioned before results in 
    \begin{align}
        \label{maa1}
        & \rho_m^F \int_\Omega c^F \ln m \leq \frac{\rho_m^F}{2} \left\{ \int_\Omega (c^F)^2 + \frac{4}{e^2} \int_\Omega m\right\}, \\[5pt]
        & \rho_m^E\int_\Omega  c^E \ln m \leq \frac{\rho_m^E}{2} \left\{ \int_\Omega (c^E)^2 + \frac{4}{e^2} \int_\Omega m\right\}, \\[5pt]
         \label{maa2}
        & \lambda_m \int_\Omega (-m \ln m) \leq \frac{\lambda_m}{e}|\Omega|
    \end{align}
    for all $t \in (0, T_{\max})$.  Next, we estimate $\int_\Omega \partial_t m$. Owing to the nonnegativity of the solution components, we have
    \begin{align}
        \label{maa34}
        \int_\Omega \left\{\rho_{m}^{F}\,c^{F}+\rho_{m}^{E}\,c^{E}-\lambda_{m}\,m\right\} \leq \rho_{m}^{F} \int_\Omega c^F + \rho_{m}^{E} \int_\Omega c^E
    \end{align}
    for all $t \in (0, T_{\max})$. Inserting \eqref{maa1}-\eqref{maa34} into \eqref{ml12} results in \eqref{mlog_main}. 
\end{proof}
To control the terms arising from the haptotactic interactions, we study the evolution of two gradient-type functionals associated with $v$, namely $\frac{1}{2}\int_\Omega \frac{|\nabla v|^2}{v}$ and $\int_\Omega |\nabla v|^2$. The corresponding results are stated in the following lemmas.
\begin{lemma}
    \label{vloge}
   The following holds true:
    \begin{align}
        \label{v_main}
         \nonumber  & \frac{1}{2} \frac{d}{dt} \int_\Omega \frac{|\nabla v|^2}{v} + \rho_v \int_\Omega \nabla c^F \cdot \nabla v + \delta_2 \int_\Omega \nabla c^E \cdot \nabla v + \delta_3 \int_\Omega \nabla c^M \cdot \nabla v + \frac{\rho_v}{2} \int_\Omega c^F \frac{|\nabla v|^2}{v} \\[5pt]
         & +  \frac{\delta_2}{2} \int_\Omega c^E \frac{|\nabla v|^2}{v}  + \frac{\delta_3}{2} \int_\Omega c^M \frac{|\nabla v|^2}{v} \leq \frac{\delta_1}{2} \int_\Omega \frac{|\nabla m|^2}{m} + \frac{\rho_v}{2} \int_\Omega \frac{\left|\nabla c^F\right|^2}{c^F}
    \end{align}
   for all $t\in(0,T_{\max})$.
\end{lemma}
\begin{proof}
    We can directly calculate from the $v$ equation
    \begin{align}
        \label{v_mai}
        \nonumber & \frac{1}{2} \frac{d}{dt} \int_\Omega \frac{|\nabla v|^2}{v} = \int_\Omega \frac{1}{v} \nabla v \cdot \nabla v_t - \frac{1}{2} \int_\Omega \frac{|\nabla v|^2}{v^2}v_t \\[5pt]
        \nonumber & =  \int_\Omega \frac{1}{v} \nabla v \cdot \nabla (\rho_{v}c^{F} - \rho_{v}c^{F}v - \delta_1 mv - \delta_2 c^E v - \delta_3 c^M v) \\[5pt]
        \nonumber & - \frac{1}{2} \int_\Omega \frac{|\nabla v|^2}{v^2} \left(\rho_{v}c^{F} - \rho_{v}c^{F}v - \delta_1 mv - \delta_2 c^E v - \delta_3 c^M v\right) \\[5pt]
        \nonumber & = \rho_v \int_\Omega \frac{1}{v} \nabla v \cdot \nabla c^F - \rho_v \int_\Omega \nabla c^F \cdot \nabla v - \rho_v \int_\Omega c^F \frac{|\nabla v|^2}{v} - \delta_1 \int_\Omega \nabla m \cdot \nabla v - \delta_1 \int_\Omega m \frac{|\nabla v|^2}{v} \\[5pt]
        \nonumber & - \delta_2 \int_\Omega \nabla c^E \cdot \nabla v - \delta_2 \int_\Omega c^E \frac{|\nabla v|^2}{v} - \delta_3 \int_\Omega \nabla c^M \cdot \nabla v - \delta_3 \int_\Omega c^M \frac{|\nabla v|^2}{v} - \frac{\rho_v}{2} \int_\Omega c^F \frac{|\nabla v|^2}{v^2}       \\[5pt]
        & + \frac{\rho_v}{2} \int_\Omega c^F \frac{|\nabla v|^2}{v} + \frac{\delta_1}{2} \int_\Omega m \frac{|\nabla v|^2}{v}  + \frac{\delta_2}{2} \int_\Omega c^E \frac{|\nabla v|^2}{v}  + \frac{\delta_3}{2} \int_\Omega c^M \frac{|\nabla v|^2}{v}
    \end{align}
    for all $t \in (0, T_{\max})$. By means of Young's inequality and \eqref{v_es1} we have
    \begin{align}
        \label{vw2}
        &\rho_v \int_\Omega \frac{1}{v} \nabla v \cdot \nabla c^F \leq \frac{\rho_v}{2} \int_\Omega c^F \frac{|\nabla v|^2}{v^2} + \frac{\rho_v}{2} \int_\Omega \frac{\left|\nabla c^F\right|^2}{c^F} \qquad \text{and} \\[5pt]
         \label{vw3}
        & -\delta_1 \int_\Omega \nabla m \cdot \nabla v \leq \frac{\delta_1}{2} \int_\Omega v \frac{|\nabla m|^2}{m} + \frac{\delta_1}{2} \int_\Omega m \frac{|\nabla v|^2}{v} \leq  \frac{\delta_1}{2} \int_\Omega \frac{|\nabla m|^2}{m} + \frac{\delta_1}{2} \int_\Omega m \frac{|\nabla v|^2}{v}
    \end{align}
     for all $t \in (0, T_{\max})$. Inserting \eqref{vw2} and \eqref{vw3} into \eqref{v_mai} results in \eqref{v_main}.
\end{proof}

\begin{lemma}
    \label{vloge2}
   The following holds true:
    \begin{align}
        \label{v_main2}
         \frac{d}{dt} \int_\Omega |\nabla v|^2 \leq 2 \rho_v \int_\Omega \nabla v \cdot \nabla c^F + \frac{\rho_v}{2} \int_\Omega \frac{\left|\nabla c^F\right|^2}{c^F} & + \frac{\delta_1}{2} \int_\Omega \frac{\left|\nabla m\right|^2}{m}+ \frac{\delta_2}{2} \int_\Omega \frac{\left|\nabla c^E\right|^2}{c^E} + \frac{\delta_3}{2} \int_\Omega \frac{\left|\nabla c^M\right|^2}{c^M}
    \end{align}
   for all $t\in(0,T_{\max})$.
\end{lemma}
\begin{proof}
    We can directly calculate from the $v$ equation
    \begin{align}
        \label{v_mai2}
        \nonumber & \frac{d}{dt} \int_\Omega |\nabla v|^2 = 2 \int_\Omega \nabla v \cdot \nabla (\rho_{v}c^{F} - \rho_{v}c^{F}v - \delta_1 mv - \delta_2 c^E v - \delta_3 c^M v)\\[5pt] 
        \nonumber & = 2 \rho_v \int_\Omega \nabla v \cdot \nabla c^F - 2 \rho_v \int_\Omega v \nabla c^F \cdot \nabla v - 2 \rho_v \int_\Omega c^F |\nabla v|^2 - 2 \delta_1 \int_\Omega v \nabla m \cdot \nabla v - 2 \delta_1 \int_\Omega m |\nabla v|^2 \\[5pt]
        &  - 2 \delta_2 \int_\Omega v \nabla c^E \cdot \nabla v - 2 \delta_2 \int_\Omega c^E |\nabla v|^2  - 2 \delta_3 \int_\Omega v \nabla c^M \cdot \nabla v - 2 \delta_3 \int_\Omega c^M |\nabla v|^2  
    \end{align}
    for all $t \in (0, T_{\max})$. By means of Young's inequality and \eqref{v_es1} we have the following estimates
    \begin{align}
    \label{v432}
     & - 2  \rho_v \int_\Omega v \nabla c^F \cdot \nabla v \leq 2\rho_v \int_\Omega \left|\nabla c^F \cdot \nabla v\right| \leq \frac{\rho_v}{2} \int_\Omega \frac{\left|\nabla c^F\right|^2}{c^F} + 2 \rho_v  \int_\Omega c^F |\nabla v|^2, \\[5pt]
     & - 2  \delta_1 \int_\Omega v \nabla m \cdot \nabla v \leq 2 \delta_1 \int_\Omega \left|\nabla m \cdot \nabla v\right| \leq \frac{\delta_1}{2} \int_\Omega \frac{\left|\nabla m\right|^2}{m} + 2 \delta_1  \int_\Omega m |\nabla v|^2, \\[5pt]
     & -2  \delta_2 \int_\Omega v \nabla c^E \cdot \nabla v \leq 2 \delta_2 \int_\Omega \left|\nabla c^E \cdot \nabla v\right| \leq \frac{\delta_2}{2} \int_\Omega \frac{\left|\nabla c^E\right|^2}{c^E} + 2 \delta_2  \int_\Omega c^E |\nabla v|^2 \quad \text{and} \\[5pt]
     \label{b323}
     & -2  \delta_3 \int_\Omega v \nabla c^M \cdot \nabla v \leq 2 \delta_3 \int_\Omega \left|\nabla c^M \cdot \nabla v\right| \leq \frac{\delta_3}{2} \int_\Omega \frac{\left|\nabla c^M\right|^2}{c^M} + 2 \delta_3  \int_\Omega c^M |\nabla v|^2
    \end{align}
for all $t \in (0, T_{\max})$. Inserting \eqref{v432}-\eqref{b323} into \eqref{v_mai2} results in \eqref{v_main2}. 
\end{proof}
With these lemmas at hand, we now derive $L\log L$ estimates for $a$, $u$, and $w$, which will play a crucial role in establishing the corresponding $L^p$ estimates for arbitrary $p>1$.
\begin{lemma}
\label{entropy_lemma_ablog}
Let $T>0$. Suppose that
\begin{align}
\label{asd99}
\chi_E<2, \qquad \chi_M<2 \qquad \text{and} \qquad \chi_F<1.
\end{align}
Then there exists a constants $\widehat{C}_E(T), \widehat{C}_M(T), \widehat{C}_F(T) >0$ such that
\begin{align}
\label{cEln_bound}
&\int_\Omega a(\cdot,t)\left|\ln a(\cdot,t)\right|
\leq \widehat{C}_E(T),
\\[5pt]
\label{cMln_bound}
&\int_\Omega u(\cdot,t)\left|\ln u(\cdot,t)\right|
\leq \widehat{C}_M(T) \qquad \text{and}
\\[5pt]
\label{cFln_bound}
&\int_\Omega w(\cdot,t)\left|\ln w(\cdot,t)\right|
\leq \widehat{C}_F(T)
\end{align}
for all $t\in(0,\widehat{T})$, where
$\widehat{T}:=\min\{T,T_{\max}\}$.
\end{lemma}
\begin{proof}
  As a consequence of Lemmas~\ref{cwwe_log11}, \ref{cm_log11}, \ref{wwloge}, and \ref{mloge}, together with the $L^1$ estimates in \eqref{L1ess} and \eqref{mes1}, we obtain the following differential inequalities for some appropriate constant $C_1(T)>0$:
  \begin{align}
  \label{op88}
      \frac{d}{dt} \int_\Omega c^E \ln c^E + \int_\Omega \frac{\left|\nabla c^E\right|^2}{c^E}  \leq  \chi_E \int_\Omega \nabla c^E \cdot \nabla v & + \frac{\rho_E}{2}\int_\Omega (c^E)^2  + \frac{\eta \mu_M \zeta^\ast}{2}\int_\Omega (c^M)^2 + C_1(T),
  \end{align}
  and,
  \begin{align}
      \label{piu}
      \frac{d}{dt} \int_\Omega c^M \ln c^M + \int_\Omega \frac{\left|\nabla c^M\right|^2}{c^M}  \leq  \chi_M \int_\Omega \nabla c^M \cdot \nabla v + \frac{\mu_E \zeta^\ast}{2} \int_\Omega (c^E)^2 + C_1(T),
  \end{align}
  and,
  \begin{align}
      \label{oaa2}
      \nonumber \frac{d}{dt} \int_\Omega c^F \ln c^F + \int_\Omega \frac{\left|\nabla c^F\right|^2}{c^F} & + \chi_F \int_\Omega \nabla c^F \cdot \nabla v \leq R_1 \int_\Omega (c^E)^2 \\[5pt] 
       & + \frac{\eta \rho_F}{2e}\int_\Omega(c^M)^2+ R_2 \int_\Omega (c^F)^2 + C_1(T),
  \end{align}
as well as
     \begin{align}
        \label{q34}
         \frac{d}{dt} \int_\Omega m \ln m + \int_\Omega \frac{\left|\nabla m\right|^2}{m} \leq \frac{\rho_m^F}{2} \int_\Omega (c^F)^2 +  \frac{\rho_m^E}{2} \int_\Omega (c^E)^2 + C_1(T),
    \end{align}
for all $t \in (0, \widehat{T})$. Also, combining the outcomes of Lemmas \ref{vloge} and \ref{vloge2} and taking into account the positivity of $\frac{1}{2}\int_\Omega \frac{|\nabla v|^2}{v}$ and $\int_\Omega |\nabla v|^2$ we have
\begin{align}
    \label{op99}
    \nonumber & \frac{d}{dt} \left\{\frac{1}{2} \int_\Omega \frac{|\nabla v|^2}{v} + \int_\Omega |\nabla v|^2 \right\} + \delta_2 \int_\Omega \nabla c^E \cdot \nabla v + \delta_3 \int_\Omega \nabla c^M \cdot \nabla v + \frac{\rho_v}{2} \int_\Omega c^F \frac{|\nabla v|^2}{v}  \\[5pt]  
    \nonumber & +  \frac{\delta_2}{2} \int_\Omega c^E \frac{|\nabla v|^2}{v} + \frac{\delta_3}{2} \int_\Omega c^M \frac{|\nabla v|^2}{v} \leq \frac{1}{2} \int_\Omega \frac{|\nabla v|^2}{v} + \int_\Omega |\nabla v|^2 + \rho_v \int_\Omega \nabla v \cdot \nabla c^F  \\[5pt]
    & + \delta_1 \int_\Omega \frac{|\nabla m|^2}{m} + \rho_v \int_\Omega \frac{\left|\nabla c^F\right|^2}{c^F}+ \frac{\delta_2}{2} \int_\Omega \frac{\left|\nabla c^E\right|^2}{c^E} + \frac{\delta_3}{2} \int_\Omega \frac{\left|\nabla c^M\right|^2}{c^M} 
\end{align}
for all $t \in (0, T_{\max})$. Multiplying \eqref{op88} by $\frac{\delta_2}{\chi_E}$, \eqref{piu} by $\frac{\delta_3}{\chi_M}$, \eqref{oaa2} by $\frac{\rho_v}{\chi_F}$, and \eqref{q34} by $\delta_1$, and then adding the resulting inequalities to \eqref{op99}, we obtain
\begin{align}
    \label{pbtt}
    \nonumber & \frac{d}{dt} \left\{\frac{\delta_2}{\chi_E} \int_\Omega c^E \ln c^E  + \frac{\delta_3}{\chi_M}  \int_\Omega c^M \ln c^M + \frac{\rho_v}{\chi_F}  \int_\Omega c^F \ln c^F + \delta_1 \int_\Omega m \ln m + \frac{1}{2} \int_\Omega \frac{|\nabla v|^2}{v} + \int_\Omega |\nabla v|^2\right\}  \\[5pt]
    \nonumber & + \frac{\delta_2}{\chi_E} \int_\Omega \frac{\left|\nabla c^E\right|^2}{c^E} + \frac{\delta_3}{\chi_M}  \int_\Omega \frac{\left|\nabla c^M\right|^2}{c^M} + \frac{\rho_v}{\chi_F}\int_\Omega \frac{\left|\nabla c^F\right|^2}{c^F} + \frac{\rho_v}{2} \int_\Omega c^F \frac{|\nabla v|^2}{v} +  \frac{\delta_2}{2} \int_\Omega c^E \frac{|\nabla v|^2}{v} \\[5pt]
    \nonumber & + \frac{\delta_3}{2} \int_\Omega c^M \frac{|\nabla v|^2}{v}  \leq \frac{1}{2} \int_\Omega \frac{|\nabla v|^2}{v} + \int_\Omega |\nabla v|^2 + \rho_v \int_\Omega \frac{\left|\nabla c^F\right|^2}{c^F}+ \frac{\delta_2}{2} \int_\Omega \frac{\left|\nabla c^E\right|^2}{c^E} + \frac{\delta_3}{2} \int_\Omega \frac{\left|\nabla c^M\right|^2}{c^M} \\[5pt]
    & + \Lambda_1 \int_\Omega (c^E)^2 + \eta \Lambda_2 \int_\Omega (c^M)^2 + \Lambda_3 \int_\Omega (c^F)^2 + \left(\frac{\delta_2}{\chi_E} + \frac{\delta_3}{\chi_M}  +  \frac{\rho_v}{\chi_F} + \delta_1\right)C_1(T)
\end{align}
for all $t \in (0, \widehat{T})$, where 
\begin{align*}
    \Lambda_1 := \frac{\delta_2\rho_E}{2\chi_E} + \frac{\delta_3 \mu_E \zeta^\ast}{2 \chi_M} + \frac{\delta_1 \rho_m^E}{2} + \frac{\rho_v R_1}{\chi_F}, \quad \Lambda_2 := \frac{\delta_2 \mu_M \zeta^\ast}{2 \chi_E} + \frac{\rho_v \rho_F}{2e \chi_F} \quad \text{and} \quad \Lambda_3 := \frac{\delta_1 \rho_m^F}{2} + \frac{\rho_v R_2}{\chi_F}.
\end{align*}
 Let
  \begin{align*}
    \mathcal{E}(t) := & \frac{\delta_2}{\chi_E} \int_\Omega c^E(\cdot, t)\ln c^E(\cdot, t) + \frac{\delta_3}{\chi_M} \int_\Omega c^M(\cdot, t) \ln c^M(\cdot, t) + \frac{\rho_v}{\chi_F} \int_\Omega c^F(\cdot, t) \ln c^F(\cdot, t) \\[5pt]
    &  + \delta_1\int_\Omega m(\cdot, t)\ln m(\cdot, t) + \frac{1}{2} \int_\Omega \frac{|\nabla v(\cdot, t)|^2}{v(\cdot, t)} + \int_\Omega |\nabla v(\cdot, t)|^2
  \end{align*}
for all $t \in (0, T_{\max})$. Thanks to \eqref{asd99} we have $\sigma_1, \sigma_2, \sigma_3 > 0$ such that
\begin{align}
\label{po999a}
     \nonumber & \frac{d}{dt} \mathcal{E}(t) + \sigma_1 \int_\Omega \frac{\left|\nabla c^E\right|^2}{c^E} + \sigma_2 \int_\Omega \frac{\left|\nabla c^M\right|^2}{c^M} + \sigma_3 \int_\Omega \frac{\left|\nabla c^F\right|^2}{c^F} + \frac{\rho_v}{2} \int_\Omega c^F \frac{|\nabla v|^2}{v} + \frac{\delta_2}{2} \int_\Omega c^E \frac{|\nabla v|^2}{v}  \\[5pt]
     \nonumber & + \frac{\delta_3}{2} \int_\Omega c^M \frac{|\nabla v|^2}{v}  \leq \frac{1}{2} \int_\Omega \frac{|\nabla v|^2}{v}  + \int_\Omega |\nabla v|^2 + \Lambda_1 \int_\Omega (c^E)^2 + \eta \Lambda_2 \int_\Omega (c^M)^2 + \Lambda_3 \int_\Omega (c^F)^2 \\[5pt]
      & + \left(\frac{\delta_2}{\chi_E} + \frac{\delta_3}{\chi_M}  +  \frac{\rho_v}{\chi_F} + \delta_1\right)C_1(T)
\end{align}
for all $t \in (0, \widehat{T})$, where
\begin{align*}
    & \sigma_1 := \delta_2 \left(\frac{1}{\chi_E} - \frac{1}{2}\right), \quad \sigma_2 := \delta_3 \left(\frac{1}{\chi_M} - \frac{1}{2}\right) \quad \text{and} \quad
    \sigma_3 := \rho_v \left(\frac{1}{\chi_F} - 1\right). 
\end{align*}
Next, we estimate $\eta \Lambda_2 \int_\Omega (c^M)^2$ using the Gagliardo-Nirenberg inequality \cite[(3.33)]{surulescu2021does}
\begin{align}
\label{gaagge}
    \int_\Omega \varphi^4 
    \leq C_{\mathrm{GN}}^{(2)}\left\{\int_\Omega |\nabla \varphi|^2\right\} \cdot
    \left\{\int_\Omega \varphi^2\right\}+ C_{\mathrm{GN}}^{(2)}\left\{\int_\Omega \varphi^2\right\}^2 \qquad \text{for all } \varphi\in W^{1,2}(\Omega).
\end{align}
For fixed $T > 0$, choose $\eta \equiv \eta(T) > 0$ sufficiently small such that 
\begin{align}
\label{aakk}
    \frac{\eta \Lambda_2 C_{\mathrm{GN}}^{(2)} M_1(T)}{4} = \frac{\sigma_2}{2}.
\end{align}
In view of \eqref{aakk} and \eqref{L1ess}, it follows from \eqref{gaagge} that
\begin{align}
    \label{GNqq}
    \nonumber
    \eta \Lambda_2 \int_\Omega (c^M)^2 &\leq \frac{\eta \Lambda_2 C_{\mathrm{GN}}^{(2)}}{4} \left\{\int_\Omega \frac{\left|\nabla c^M\right|^2}{c^M}\right\} \cdot \left\{\int_\Omega c^M\right\}
    + \eta \Lambda_2 C_{\mathrm{GN}}^{(2)}
    \left\{\int_\Omega c^M\right\}^2 \\[5pt]
    \nonumber
    &\leq
    \frac{\eta \Lambda_2 C_{\mathrm{GN}}^{(2)}
    M_1(T)}{4}
    \int_\Omega \frac{\left|\nabla c^M\right|^2}{c^M}
    + \eta \Lambda_2 C_{\mathrm{GN}}^{(2)}
    M_1(T)^2 \\[5pt]
    &=
    \frac{\sigma_2}{2}
    \int_\Omega \frac{\left|\nabla c^M\right|^2}{c^M}
    + \eta \Lambda_2 C_{\mathrm{GN}}^{(2)}
    M_1(T)^2
\end{align}
for all $t\in(0,\widehat{T})$. Inserting \eqref{GNqq} into \eqref{po999a} results in
\begin{align}
\label{o999aw}
     \nonumber & \frac{d}{dt} \mathcal{E}(t) + \sigma_1 \int_\Omega \frac{\left|\nabla c^E\right|^2}{c^E} + \frac{\sigma_2}{2} \int_\Omega \frac{\left|\nabla c^M\right|^2}{c^M} + \sigma_3 \int_\Omega \frac{\left|\nabla c^F\right|^2}{c^F}  + \frac{\rho_v}{2} \int_\Omega c^F \frac{|\nabla v|^2}{v}  + \frac{\delta_2}{2} \int_\Omega c^E \frac{|\nabla v|^2}{v}\\[5pt]
     &    + \frac{\delta_3}{2} \int_\Omega c^M \frac{|\nabla v|^2}{v}  \leq \frac{1}{2} \int_\Omega \frac{|\nabla v|^2}{v} + \int_\Omega |\nabla v|^2 + \Lambda_1 \int_\Omega (c^E)^2 + \Lambda_3 \int_\Omega (c^F)^2 + \overline{\Lambda}(T)
\end{align}
for all $t\in(0,\widehat{T})$, where 
\begin{align*}
    \overline{\Lambda}(T) := \left(\frac{\delta_2}{\chi_E} + \frac{\delta_3}{\chi_M}  +  \frac{\rho_v}{\chi_F} + \delta_1\right)C_1(T) + \eta \Lambda_2 C_{\mathrm{GN}}^{(2)}
    M_1(T)^2.
\end{align*}
From \eqref{o999aw} we have the following differential inequality
\begin{align*}
    \nonumber \frac{d}{dt} \mathcal{E}(t) & \leq  \mathcal{E}(t) - \frac{\delta_2}{\chi_E} \int_\Omega c^E \ln c^E - \frac{\delta_3}{\chi_M}  \int_\Omega c^M \ln c^M - \frac{\rho_v}{\chi_F}  \int_\Omega c^F \ln c^F - \delta_1 \int_\Omega m \ln m \\[5pt]
     & + \Lambda_1 \int_\Omega (c^E)^2 + \Lambda_3 \int_\Omega (c^F)^2 + \overline{\Lambda}(T)
\end{align*}
for all $t\in(0,\widehat{T})$. Thanks to the logarithmic inequality $-s\ln s\leq e^{-1}$ for all $s>0$, we obtain
\begin{align*}
\frac{d}{dt}\mathcal{E}(t)
\leq \mathcal{E}(t)
+ \left(
\frac{\delta_2}{\chi_E}
+\frac{\delta_3}{\chi_M}
+\frac{\rho_v}{\chi_F}
+\delta_1
\right)\frac{|\Omega|}{e}
+\overline{\Lambda}(T)
+\Lambda_1\int_\Omega (c^E)^2
+\Lambda_3\int_\Omega (c^F)^2
\end{align*}
for all $t\in(0,\widehat{T})$. Hence, by an application of Gronwall's lemma, together with \eqref{sqc}, there exists a constant $\Lambda^\ast(T)>0$ such that
\begin{align*}
\mathcal{E}(t)
& \leq e^{T}\left\{
\mathcal{E}(0)+\left(\frac{\delta_2}{\chi_E}+\frac{\delta_3}{\chi_M}+\frac{\rho_v}{\chi_F} + \delta_1
\right)\frac{|\Omega|T}{e}+\overline{\Lambda}(T)T
+\Lambda_1\int_0^{\widehat{T}}\int_\Omega (c^E)^2
+\Lambda_3\int_0^{\widehat{T}}\int_\Omega (c^F)^2
\right\} \\[5pt]
& \leq  e^{T}\left\{
\mathcal{E}(0)+\left(\frac{\delta_2}{\chi_E}+\frac{\delta_3}{\chi_M}+\frac{\rho_v}{\chi_F} + \delta_1
\right)\frac{|\Omega|T}{e}+\overline{\Lambda}(T)T
+(\Lambda_1+\Lambda_3)M_2(T)
\right\}
=: \Lambda^\ast(T)
\end{align*}
for all $t\in(0,\widehat{T})$. We next derive separate upper bounds for the first three logarithmic terms constituting $\mathcal{E}(t)$. Since $s\ln s\geq -e^{-1}$ for all $s>0$, we have 
\begin{align*}
    \frac{\delta_2}{\chi_E} \int_\Omega c^E \ln c^E  \leq \mathcal{E}(t) + \left(\frac{\delta_3}{\chi_M}
+\frac{\rho_v}{\chi_F} + \delta_1\right)\frac{|\Omega|}{e} \leq \Lambda^\ast(T) + \left(\frac{\delta_3}{\chi_M}
+\frac{\rho_v}{\chi_F} + \delta_1\right)\frac{|\Omega|}{e},
\end{align*}
for all $t\in(0,\widehat{T})$ and hence
\begin{align}
\label{p8x}
     \int_\Omega c^E \ln c^E \leq \frac{\chi_E}{\delta_2} \Lambda^\ast(T) + \frac{\chi_E}{\delta_2} \left(\frac{\delta_3}{\chi_M}
+\frac{\rho_v}{\chi_F} + \delta_1\right)\frac{|\Omega|}{e} =: C_E(T).
\end{align}
for all $t\in(0,\widehat{T})$. Similarly, we have 
\begin{align}
     \label{pybaa}
      & \int_\Omega c^M \ln c^M \leq \frac{\chi_M}{\delta_3} \Lambda^\ast(T) + \frac{\chi_M}{\delta_3}\left(\frac{\delta_2}{\chi_E} + \frac{\rho_v}{\chi_F} + \delta_1\right)\frac{|\Omega|}{e} =: C_M(T), \\[5pt]
    \label{pybb}
     & \int_\Omega c^F \ln c^F \leq \frac{\chi_F}{\rho_v} \Lambda^\ast(T) + \frac{\chi_F}{\rho_v}\left(\frac{\delta_2}{\chi_E} + \frac{\delta_3}{\chi_M} + \delta_1\right)\frac{|\Omega|}{e} =: C_F(T).
\end{align}
for all $t\in(0,\widehat{T})$. In view of \eqref{cov}, \eqref{v_es1} and \eqref{p8x} observe that
\begin{align}
\label{alpo2}
    \nonumber \int_\Omega a \ln a = \int_\Omega e^{-\chi_E v} c^E \ln \left(e^{-\chi_E v} c^E\right) & =  \int_\Omega e^{-\chi_E v} c^E (- \chi_E v + \ln c^E)\\[5pt] & \leq \int_\Omega c^E \ln c^E \leq C_E(T)
\end{align}
for all $t\in(0,\widehat{T})$. Similarly, using \eqref{cov}, \eqref{v_es1} and \eqref{pybaa} we have
\begin{align}
\label{uloh3}
    \nonumber \int_\Omega u \ln u = \int_\Omega e^{-\chi_M v} c^M \ln \left(e^{-\chi_M v} c^M\right) & =  \int_\Omega e^{-\chi_M v} c^M (- \chi_M v + \ln c^M)\\[5pt] & \leq \int_\Omega c^M \ln c^M \leq C_M(T)
\end{align}
for all $t\in(0,\widehat{T})$. Also, using \eqref{cov}, \eqref{v_es1}, \eqref{L1ess} and \eqref{pybb} we have
\begin{align}
\label{wloh3}
    \nonumber \int_\Omega w \ln w = \int_\Omega e^{\chi_F v} c^F \ln \left(e^{\chi_F v} c^F\right) & =  \int_\Omega e^{\chi_F v} c^F ( \chi_F v + \ln c^F)\\[5pt] 
    \nonumber & \leq \chi_F \Lambda_F M_1(T) + \Lambda_F \int_\Omega c^F \ln c^F \\[5pt]
    &\leq  \chi_F \Lambda_F M_1(T) + \Lambda_F C_F(T)
\end{align}
for all $t\in(0,\widehat{T})$. Finally, using once again $s\ln s\geq -e^{-1}$ for all $s>0$, we obtain for each $f \in \{a, u, w\}$
 \begin{align}
 \label{lllii}
    \nonumber \int_\Omega f(\cdot, t) |\ln f(\cdot, t)| & = \int_\Omega f(\cdot, t) \ln f(\cdot, t) - 2 \int_{\{f < 1\}} f(\cdot, t) \ln f(\cdot, t)                  \\[5pt]
     & \leq \int_\Omega f(\cdot, t) \ln f(\cdot, t) + 2 e^{-1} |\Omega| 
\end{align}
for all $t\in(0,\widehat{T})$.  In view of \eqref{alpo2}-\eqref{wloh3} and \eqref{lllii}, we have constants  $\widehat{C}_E(T), \widehat{C}_M(T), \widehat{C}_F(T) >0$ such that 
\begin{align*}
     & \int_\Omega a(\cdot, t) |\ln a(\cdot, t)| \leq C_E(T) +  2 e^{-1} |\Omega| =: \widehat{C}_E(T), \\[5pt]
     &\int_\Omega u(\cdot, t) |\ln u(\cdot, t)| \leq C_M(T) +  2 e^{-1} |\Omega| =: \widehat{C}_M(T) \qquad \text{and} \\[5pt]
     & \int_\Omega w(\cdot, t) |\ln w(\cdot, t)| \leq  \chi_F \Lambda_F M_1(T) + \Lambda_F C_F(T)  +  2 e^{-1} |\Omega|=: \widehat{C}_F(T)
\end{align*}
for all $t\in(0,\widehat{T})$. Thus establishing the desired outcomes in \eqref{cEln_bound}-\eqref{cFln_bound}.
\end{proof}
We will now proceed to derive an a priori $L^\infty$ estimate for $a$, $u$ and $w$, and hence for $c^E$, $c^M$ and $c^F$. This will be achieved by combining the outcome of Lemma \ref{entropy_lemma_ablog} with a standard $L^p$ testing procedure, followed by applying the well-established Moser-type iteration argument.
\begin{lemma}
\label{lemma_ablinfty}
Let $T > 0$. Suppose that the assumptions of Lemma \ref{entropy_lemma_ablog} are satisfied. Then there exists a constant $C(T) > 0$ such that
\begin{align}
\label{c3llnfty12}
& \|a(\cdot,t)\|_{L^\infty(\Omega)} \leq C(T), \\[5pt]
\label{c1llnfty12}
& \|u(\cdot,t)\|_{L^\infty(\Omega)} \leq C(T) \qquad \text{and} \\[5pt]
\label{c2llnfty12}
& \|w(\cdot,t)\|_{L^\infty(\Omega)} \leq C(T)
\end{align}
for all $t \in (0,\widehat{T})$, where $\widehat{T} := \min\{T, T_{\max}\}$.
\end{lemma}
\begin{proof}
  Thanks to Lemma \ref{lemma_mboundpq} for any $p+ 1$ we have
    \begin{align}
    \label{m_bound}
        \int_\Omega m^{p + 1} \leq K_1(p, T)
    \end{align}
    for all $t \in (0, \widehat{T})$. Testing the $a$ equation by $e^{\chi_E v} a^{p - 1}$ with $p > 4$ and integrating by parts results in 
   \begin{align}
     \label{app1}
      \nonumber & \frac{d}{dt} \int_\Omega e^{\chi_E v} a^p + \int_\Omega e^{\chi_E v} a^p= p \int_\Omega e^{\chi_E v} a^{p - 1} a_t + \chi_E \int_\Omega e^{\chi_E v} a^p v_t + \int_\Omega e^{\chi_E v} a^p\\[5pt]
     \nonumber & = p \int_\Omega e^{\chi_E v} a^{p - 1} \big\{ e^{-\chi_E v} \nabla \cdot \big(e^{\chi_E v} \nabla a\big)- \mu_E \zeta_1(b) a + \mu_M \zeta_2(b) e^{(\chi_M - \chi_E) v} u  - \rho_v \chi_E e^{- \chi_F v} a w (1 - v)\\[5pt]
  \nonumber & + \rho_E a\left(1 - e^{\chi_E v} a - e^{\chi_M v} u - e^{- \chi_F v} w - v\right)  + \delta_1 \chi_E a m v + \delta_2 \chi_E e^{\chi_E v} a^2 v + \delta_3 \chi_E e^{\chi_M v} a u v\big\} \\[5pt]
  & + \chi_E \int_\Omega e^{\chi_E v} a^p\big\{\rho_v e^{-\chi_F v} w \big(1 - v\big) - \delta_1 m v - \delta_2 e^{\chi_E v} a v - \delta_3 e^{\chi_M v} uv\big\} + \int_\Omega e^{\chi_E v} a^p
\end{align}
for all $t \in (0, T_{\max})$. Since 
\begin{align*}
    -(p - 1) \rho_v \chi_E \int_\Omega e^{(\chi_E - \chi_F)v} a^p w (1 - v) \leq 0 \qquad \text{for all}~ t \in (0, T_{\max}).
\end{align*}
We will now estimate the remaining terms. To this end, first consider the term $ p \mu_M \zeta_2(b) \int_\Omega e^{\chi_M v} a^{p - 1} u$. A few applications of Young's inequality, together with \eqref{zeta}, yields
\begin{align}
    \label{iue12}
    \nonumber p \mu_M \int_\Omega \zeta_2(b) e^{\chi_M v} a^{p - 1} u & \leq p\Lambda_M \mu_M \zeta^\ast\left\{\frac{p - 1}{p + 1} \int_\Omega a^{p + 1} + \frac{2}{p + 1} \int_\Omega u^{\frac{p + 1}{2}}\right\} \\[5pt]
    & \leq p\Lambda_M \mu_M \zeta^\ast\left\{\frac{p - 1}{p + 1} \int_\Omega a^{p + 1} + \frac{1}{p + 1} \int_\Omega u^{p + 1} + \frac{1}{p + 1}|\Omega|\right\}
\end{align}
for all $t \in (0, T_{\max})$. Thanks to \eqref{v_es1}, \eqref{m_bound} and Young's inequality we have
\begin{align}
    \label{er33}
     \nonumber (p - 1) \delta_1 \chi_E \int_\Omega e^{\chi_E v} a^p m v & \leq (p - 1) \delta_1 \chi_E  \Lambda_E \int_\Omega a^p m \\[5pt]
     \nonumber & \leq (p - 1) \delta_1 \chi_E  \Lambda_E \left\{\frac{p}{p + 1} \int_\Omega a^{p + 1} + \frac{1}{p + 1} \int_\Omega m^{p + 1}\right\} \\[5pt]
     & \leq (p - 1) \delta_1 \chi_E  \Lambda_E \left\{\frac{p}{p + 1} \int_\Omega a^{p + 1} + \frac{1}{p + 1} K_1(p, T)\right\}
\end{align}
for all $t \in (0, \widehat{T})$. Moreover, thanks to \eqref{v_es1}, we have
    \begin{align}
        \label{opi2sa1}
        (p - 1) \delta_2 \chi_E \int_\Omega e^{2\chi_E v} a^{p + 1} v \leq p \delta_2 \chi_E \Lambda_E^2 \int_\Omega a^{p + 1}
    \end{align}
    for all $t \in (0, T_{\max})$. Again, thanks to \eqref{v_es1} and Young's inequality we obtain
    \begin{align*}
        (p - 1) \delta_3 \chi_E \int_\Omega e^{(\chi_E + \chi_M)v} a^p u v & \leq  (p - 1) \delta_3 \chi_E \int_\Omega e^{(\chi_E + \chi_M)v} a^p u \\[5pt]
        & \leq  (p - 1) \delta_3 \chi_E \Lambda_E \Lambda_M \left\{\frac{p}{p + 1} \int_\Omega a^{p + 1} + \frac{1}{p + 1} \int_\Omega u^{p+1}\right\}
    \end{align*}
  for all $t \in (0, T_{\max})$.  Also, by means of Young's inequality, we have
  \begin{align}
     \label{p9bbaa}
     (p \rho_E + 1)\int_\Omega e^{\chi_E v} a^p \leq p \int_\Omega a^{p + 1} + \Lambda_E^{p+1}\frac{(p \rho_E + 1)^{p + 1}}{(p + 1)^{p + 1}} |\Omega|
 \end{align}
for all $t \in (0, T_{\max})$. Inserting \eqref{iue12}-\eqref{p9bbaa} into \eqref{app1} and discarding the negative terms, we obtain
\begin{align}
     \label{muzz22}
      \nonumber \frac{d}{dt} \int_\Omega e^{\chi_E v} a^p + \int_\Omega e^{\chi_E v} a^p & + \frac{4(p - 1)}{p} \int_\Omega e^{\chi_E v} |\nabla a^{\frac{p}{2}}|^2 \\[5pt]
      &\leq p K_2(p) \int_\Omega a^{p + 1}  + p K_3(p) \int_\Omega u^{p + 1} + K_4(p, T)
\end{align}
 for all $t \in (0, \widehat{T})$, where
 \begin{align*}
     &  K_2(p) := \left(\frac{p - 1}{p + 1}\right) \left\{\Lambda_M \mu_M \zeta^\ast + \delta_1 \chi_E \Lambda_E + \delta_3 \chi_E \Lambda_E\Lambda_M\right\} + \delta_2 \chi_E \Lambda_E^2 + 1,  \\[5pt]
     & K_3(p) :=  \frac{\Lambda_M \mu_M \zeta^\ast}{p + 1} + \frac{1}{p} \cdot \left(\frac{p - 1}{p + 1}\right) \delta_3 \chi_E \Lambda_E \Lambda_M,  \\[5pt]
     & K_4(p, T) := \left(\frac{p}{p + 1}\right) \Lambda_M \mu_M \zeta^\ast |\Omega|  + \Lambda_E^{p+1}\frac{(p \rho_E + 1)^{p + 1}}{(p + 1)^{p + 1}} |\Omega|+ \left(\frac{p - 1}{p + 1}\right) \delta_1 \chi_E \Lambda_E K_1(p, T).
 \end{align*}
   Testing the $u$ equation by $e^{\chi_M v} u^{p - 1}$ with $p > 4$ and integrating by parts results in 
    \begin{align}
    \label{pu002}
     \nonumber & \frac{d}{dt} \int_\Omega e^{\chi_M v} u^p + \int_\Omega e^{\chi_M v} u^p = p \int_\Omega e^{\chi_M v} u^{p - 1} u_t + \chi_M \int_\Omega e^{\chi_M v} u^p v_t + \int_\Omega e^{\chi_M v} u^p   \\[5pt]
     \nonumber & = p \int_\Omega e^{\chi_M v} u^{p - 1} \Big\{ e^{-\chi_M v} \nabla \cdot \big(e^{\chi_M v} \nabla u\big) + \mu_E \zeta_1(b) e^{(\chi_E -\chi_M) v} a -\mu_M\zeta_2(b)u \\[5pt]
     \nonumber &  +\delta_1 \chi_M mvu  + \delta_2 \chi_M e^{\chi_E v} a u v + \delta_3\chi_M e^{\chi_M v} u^2 v -\rho_v \chi_Me^{-\chi_F v} uw \big(1-v\big)\Big\} \\[5pt]
    &+  \chi_M \int_\Omega e^{\chi_M v} u^p \Big\{ \rho_v e^{-\chi_F v} w \big(1 - v\big) - \delta_1 m v - \delta_2 e^{\chi_E v} a v - \delta_3 e^{\chi_M v} uv\Big\}  + \int_\Omega e^{\chi_M v} u^p
\end{align}
for all $t \in (0, T_{\max})$. Proceeding analogously as above, we have
\begin{align}
    \label{iuewss}
     p \mu_E \int_\Omega \zeta_1(b) e^{\chi_E v} u^{p - 1} a \leq p\Lambda_E \mu_E \zeta^\ast\left\{\frac{p - 1}{p + 1} \int_\Omega u^{p + 1} + \frac{1}{p + 1} \int_\Omega a^{p + 1} + \frac{1}{p + 1}|\Omega|\right\},
\end{align}
for all $t \in (0, T_{\max})$, and
\begin{align}
\label{ao88}
    (p - 1) \delta_1 \chi_M \int_\Omega e^{\chi_M v} u^p m v  \leq (p - 1) \delta_1 \chi_M  \Lambda_M \left\{\frac{p}{p + 1} \int_\Omega u^{p + 1} + \frac{1}{p + 1} K_1(p, T)\right\}
\end{align}
for all $t \in (0, \widehat{T})$. Thanks to \eqref{v_es1} and Young's inequality we obtain
\begin{align}
    \label{pnee}
    (p - 1) \delta_2 \chi_M \int_\Omega e^{\chi_M v} u^p a v \leq (p - 1) \delta_2 \chi_M \Lambda_M\left\{\frac{p}{p + 1} \int_\Omega u^{p + 1} + \frac{1}{p + 1} \int_\Omega a^{p + 1}\right\}
\end{align}
    for all $t \in (0, T_{\max})$. Again, thanks to \eqref{v_es1} and Young's inequality, we have
    \begin{align}
    \label{pn11ee}
    (p - 1) \delta_3 \chi_M \int_\Omega e^{2\chi_M v} u^{p + 1} v \leq (p - 1) \delta_3 \chi_M \Lambda_M^2 \int_\Omega u^{p + 1}
\end{align}
 for all $t \in (0, T_{\max})$. Also, we have
 \begin{align}
     \label{opii}
     \int_\Omega e^{\chi_M v} u^{p} \leq \frac{p}{p + 1}\Lambda_M \int_\Omega u^{p + 1} + \frac{\Lambda_M}{p + 1} |\Omega|
 \end{align}
 for all $t \in (0, T_{\max})$. Inserting \eqref{iuewss}-\eqref{opii} into \eqref{pu002} results in 
\begin{align}
     \label{muzzuu}
      \nonumber \frac{d}{dt} \int_\Omega e^{\chi_M v} u^p + \int_\Omega e^{\chi_M v} u^p & + \frac{4(p - 1)}{p} \int_\Omega e^{\chi_M v} |\nabla u^{\frac{p}{2}}|^2 \\[5pt]
      &\leq p K_5(p) \int_\Omega a^{p + 1}  + p K_6(p) \int_\Omega u^{p + 1} + K_7(p, T)
\end{align}
 for all $t \in (0, \widehat{T})$, where
 \begin{align*}
     &  K_5(p) := \frac{\Lambda_E \mu_E \zeta^\ast}{p + 1} + \left(1 - \frac{1}{p}\right) \frac{1}{p + 1}\delta_2 \chi_M \Lambda_M,  \\[5pt]
     & K_6(p) := \frac{p - 1}{p + 1}\left\{\Lambda_E \mu_E \zeta^\ast + (\delta_1 + \delta_2)\chi_M \Lambda_M \right\} + \left(1 - \frac{1}{p}\right) \delta_3 \chi_M \Lambda_M^2 + \frac{1}{p + 1} \Lambda_M ,  \\[5pt]
     & K_7(p, T) := \frac{p }{p + 1}\Lambda_E \mu_E \zeta^\ast|\Omega| + \frac{p - 1}{p + 1} \delta_1 \chi_M \Lambda_M K_1(p, T) + \frac{\Lambda_M}{p + 1}|\Omega|.
 \end{align*}
   Testing the $w$ equation by $e^{-\chi_F v} w^{p - 1}$ with $p > 4$ and integrating by parts results in (For details of this calculation, see Lemma \ref{3dp_lemma1})
   \begin{align}
    \label{wplkw}
    \nonumber \frac{d}{dt} \int_\Omega e^{-\chi_F v} w^p + \int_\Omega e^{-\chi_F v} w^p & + \frac{4(p-1)e^{-\chi_F}}{p}\int_\Omega |\nabla w^{\frac p2}|^2 \\[5pt] & \leq p \Pi_3 \int_\Omega a^{p + 1} + p \Pi_4 \int_\Omega w^{p + 1} + \Pi_5 
 \end{align}
 for all $t \in (0, T_{\max})$. Combining \eqref{muzz22}, \eqref{muzzuu} and \eqref{wplkw} results in the following ordinary differential inequality for $p > 4$
 \begin{align}
     \label{po866}
     \nonumber & \frac{d}{dt} \left\{\int_\Omega e^{\chi_E v} a^p + \int_\Omega e^{\chi_M v} u^p +  \int_\Omega e^{-\chi_F v} w^p \right\} +  \left\{\int_\Omega e^{\chi_E v} a^p + \int_\Omega e^{\chi_M v} u^p +  \int_\Omega e^{-\chi_F v} w^p \right\} \\[5pt]
    \nonumber &+ \frac{4(p - 1)}{p} \left\{\int_\Omega e^{\chi_E v} |\nabla a^{\frac{p}{2}}|^2 +  \int_\Omega e^{\chi_M v} |\nabla u^{\frac{p}{2}}|^2 + e^{-\chi_F}\int_\Omega |\nabla w^{\frac p2}|^2\right\} \\[5pt]
    & \leq K_8(p) \int_\Omega a^{p + 1} +  K_9(p) \int_\Omega u^{p + 1} +  K_{10}(p) \int_\Omega w^{p + 1} + K_{11}(p, T) 
 \end{align}
 for all $t \in (0, \widehat{T})$, where $K_8(p) := p K_2(p) + pK_5(p) + p\Pi_3$, $K_9(p) := p K_3(p) + p K_6(p)$ and $K_{10}(p) = p \Pi_4$ and $K_{11}(p, T) := K_4(p, T) + K_7(p, T) + \Pi_5$.  
 
 Next, we control $\int_\Omega a^{p+1}$,  $\int_\Omega u^{p+1}$ and $\int_\Omega w^{p+1}$ by the dissipation terms $\int_\Omega |\nabla a^{\frac p2}|^2$, $\int_\Omega |\nabla u^{\frac p2}|^2$ and $\int_\Omega |\nabla w^{\frac p2}|^2$ respectively by means of Lemma \ref{Lemma_2EDLOg}. Let $\varphi = a^{\frac{p}{2}}$ then by means of Lemma \ref{Lemma_2EDLOg} we have
\begin{align}
\label{ins1}
\nonumber K_8(p)\int_\Omega a^{p+1} &= K_8(p)\|a^{\frac p2}\|_{L^{\frac{2(p+1)}p}(\Omega)}^{\frac{2(p+1)}p} \\[5pt]
\nonumber &\leq K_8(p) \varepsilon\|\nabla a^{\frac{p}{2}}\|^2_{L^2(\Omega)} \cdot \int_\Omega a |\ln a^{\frac{p}{2}}| + K_8(p) K(p, \varepsilon) \cdot \left(\|a^{\frac{p}{2}}\|^{\frac{2(p + 1)}{p}}_{L^{\frac{2}{p}}(\Omega)} + 1\right) \\[5pt]
& = \frac{pK_8(p) \varepsilon}{2} \|\nabla a^{\frac{p}{2}}\|^2_{L^2(\Omega)} \cdot \int_\Omega a |\ln a| + K_8(p) K(p, \varepsilon) \cdot \left\{\left(\int_\Omega a\right)^{p + 1} + 1\right\}
\end{align}
for all $t \in (0, T_{\max})$. In view of \eqref{zL1_bound} and \eqref{cEln_bound}, it follows directly from \eqref{ins1} that
\begin{align}
    \label{ins2}
    K_8(p)\int_\Omega a^{p+1}\leq  \frac{pK_8(p)\widehat{C}_E(T)}{2} \varepsilon \|\nabla a^{\frac{p}{2}}\|^2_{L^2(\Omega)} + K_8(p) K(p, \varepsilon) \cdot \left\{(a^\ast(T))^{p + 1} + 1\right\}
\end{align}
for all $t \in (0, \widehat{T})$. For fixed $T > 0$, choose $\varepsilon \equiv \varepsilon(p, T) > 0$ sufficiently small such that with some constant $K_{12}(p,T, \varepsilon)>0$ 
\begin{align}
    \label{ins3}
     K_8(p)\int_\Omega a^{p+1} \leq \frac{4(p-1)}{p}\int_\Omega |\nabla a^{\frac p2}|^2 + K_{12}(\varepsilon, p, T)
\end{align}
for all $t \in (0, \widehat{T})$. Proceeding similarly as above, we obtain a constant $K_{13}(\varepsilon, p, T) > 0$ such that, for some $\varepsilon\equiv\varepsilon(p,T)>0$, we have
\begin{align}
    \label{ins4}
    K_9(p)\int_\Omega u^{p+1} \leq \frac{4(p-1)}{p}\int_\Omega |\nabla u^{\frac p2}|^2 + K_{13}(\varepsilon, p, T)
\end{align}
for all $t \in (0, \widehat{T})$ and a $K_{14}(\varepsilon, p, T) > 0$ such that
\begin{align}
    \label{ins00}
     K_{10}(p)\int_\Omega w^{p+1} \leq \frac{4(p-1)e^{-\chi_F}}{p}\int_\Omega |\nabla w^{\frac p2}|^2 + K_{14}(\varepsilon, p, T)
\end{align}
for all $t \in (0, \widehat{T})$. Inserting \eqref{ins3}-\eqref{ins00} into \eqref{po866} results in 
\begin{align}
     \label{pou6}
     \nonumber \frac{d}{dt} \bigg\{\int_\Omega e^{\chi_E v} a^p + \int_\Omega e^{\chi_M v} u^p & +  \int_\Omega e^{-\chi_F v} w^p \bigg\} \\[5pt]
     & +  \left\{\int_\Omega e^{\chi_E v} a^p + \int_\Omega e^{\chi_M v} u^p +  \int_\Omega e^{-\chi_F v} w^p \right\} \leq K_{15}(p, T) 
 \end{align}
 for all $t \in (0, \widehat{T})$. An ODE comparison argument then yields 
 \begin{align*}
     & \int_\Omega e^{\chi_E v} a^p + \int_\Omega e^{\chi_M v} u^p
    + \int_\Omega e^{-\chi_F v} w^p \\[5pt]
    & \leq \max \bigg\{K_{15}(p, T), \int_\Omega e^{-(p - 1)\chi_E v_0} \cdot \left(c^E_0\right)^p + \int_\Omega e^{-(p - 1)\chi_M v_0} \cdot \left(c^M_0\right)^p + \int_\Omega e^{(p - 1)\chi_F v_0} \cdot \left(c^F_0\right)^p\bigg\} \\[5pt]
    & =: K^\ast(p, T)
 \end{align*}
 for all $t \in (0, \widehat{T})$. By applying the $L^p$-$L^q$ estimates for the Neumann heat semigroup \cite[Lemma 1.3]{winkler_aggregation_2010} to the fourth equation in \eqref{sys1}, and in view of \eqref{zeta}, \eqref{v_es1}, and the $L^p$ estimate for $c^F$, we obtain a constant $K_{16}(T)>0$ such that
\begin{align}
\label{ceinf}
\|b(\cdot, t)\|_{L^\infty(\Omega)} \leq K_{16}(T) \quad \text{for all}~ t \in (0, \widehat{T}).
\end{align}
Similarly, applying the same semigroup estimates to the fifth equation in \eqref{sys1}, now using the $L^p$ estimates for both $c^E$ and $c^F$, yields a constant $K_{17}(T)>0$ such that
\begin{align}
\label{ceinf2}
\|m(\cdot, t)\|_{L^\infty(\Omega)} \leq K_{17}(T) \quad \text{for all}~ t \in (0, \widehat{T}).
\end{align}
We can now adapt the well-established Moser-Alikakos iteration technique to the present system to upgrade the $L^p$ estimates for $a$, $u$, and $w$ to $L^\infty$ bounds. In view of \eqref{ceinf}-\eqref{ceinf2}, there exist constants $K_{19},K_{20}>0$, independent of $p$, and, upon setting $K_{18}:=\min\{1,e^{-\chi_F}\}$, we have
  \begin{align}
 \label{inew2w}
    \nonumber & \frac{d}{dt}\bigg\{ \int_\Omega e^{\chi_E v} a^p
   +\int_\Omega e^{\chi_M v} u^p
    + \int_\Omega e^{-\chi_F v} w^p \bigg\} + \bigg\{\int_\Omega e^{\chi_E v} a^p +
   \int_\Omega e^{\chi_M v} u^p
    + \int_\Omega e^{-\chi_F v} w^p \bigg\} \\[5pt]
     & \leq\; - 2 K_{18} \left\{\int_\Omega |\nabla a^{\frac{p}{2}}|^2 + \int_\Omega |\nabla u^{\frac{p}{2}}|^2 + \int_\Omega |\nabla w^{\frac{p}{2}}|^2\right\} + p K_{19}\left\{\int_\Omega  a^{p+1} + \int_\Omega u^{p + 1} + \int_\Omega w^{p + 1}\right\} + K_{20}
\end{align}
for all $p \geq 4$ and $t \in (0, \widehat{T})$. Since $2 \leq \frac{2(p+1)}{p} \leq \frac52$ for $p \geq 4$, we apply the following Gagliardo-Nirenberg inequality \cite[pp.~110]{li2021boundedness}:
\begin{align}
    \label{Gag_Nin11}
    \|\varphi\|_{L^{\frac{2(p+1)}{p}}(\Omega)}
    \leq C_{\mathrm{GN}}^{(2)} \|\nabla \varphi\|_{L^2(\Omega)}^{\frac{p+2}{2(p+1)}} \; \cdot \|\varphi\|_{L^1(\Omega)}^{\frac{p}{2(p+1)}}
    + C_{\mathrm{GN}}^{(2)} \|\varphi\|_{L^1(\Omega)}
    \quad \text{for all } \varphi \in W^{1,2}(\Omega)
\end{align}
 with some constant $C_{\mathrm{GN}}^{(2)} = C_{\mathrm{GN}}^{(2)}(\Omega) > 0$, in conjunction with Young's inequality, to have
\begin{align}
\label{oiu}
    \nonumber p K_{19}\int_\Omega a^{p+1}
    &= K_{19} \|a^{\frac{p}{2}}\|_{L^{\frac{2(p+1)}{p}}(\Omega)}^{\frac{2(p+1)}{p}} \\[5pt]
    \nonumber &\leq p K_{19}\left(
    C_{\mathrm{GN}}^{(2)} \|\nabla a^{\frac{p}{2}}\|_{L^2(\Omega)}^{\frac{p+2}{2(p+1)}} \; \cdot
    \|a^{\frac{p}{2}}\|_{L^1(\Omega)}^{\frac{p}{2(p+1)}}
    + C_{\mathrm{GN}}^{(2)} \|a^{\frac{p}{2}}\|_{L^1(\Omega)}
    \right)^{\frac{2(p+1)}{p}} \\[5pt]
    \nonumber &\leq 8 p K_{19} (C_{\mathrm{GN}}^{(2)})^3
    \|\nabla a^{\frac{p}{2}}\|_{L^2(\Omega)}^{\frac{p+2}{p}} \; \cdot
    \|a^{\frac{p}{2}}\|_{L^1(\Omega)}
    + 8p K_{19} (C_{\mathrm{GN}}^{(2)})^3
    \|a^{\frac{p}{2}}\|_{L^1(\Omega)}^{\frac{2(p+1)}{p}} \\[5pt]
    &\leq K_{18} \|\nabla a^{\frac{p}{2}}\|_{L^2(\Omega)}^2
    + K_{21} p^4 \max \left\{1, \|a^{\frac{p}{2}}\|_{L^1(\Omega)}\right\}^{\frac{2p}{p-2}},
\end{align}
where we have used the fact that
$$
\frac{2(p+1)}{p} \leq \frac{2p}{p-2} \leq 4 ~\text{for any}~ p \geq 4.
$$
Similarly, we obtain
\begin{align}
\label{aa87}
 & pK_{19} \int_\Omega u^{p+1}
    \leq K_{18} \|\nabla u^{\frac{p}{2}}\|_{L^2(\Omega)}^2
    + K_{21} p^4 \max \left\{1, \|u^{\frac{p}{2}}\|_{L^1(\Omega)}\right\}^{\frac{2p}{p-2}}, \\[5pt]
    \label{opi}
    &  pK_{19} \int_\Omega w^{p+1}
    \leq K_{18} \|\nabla w^{\frac{p}{2}}\|_{L^2(\Omega)}^2
    + K_{21} p^4 \max \left\{1, \|w^{\frac{p}{2}}\|_{L^1(\Omega)}\right\}^{\frac{2p}{p-2}}.
\end{align}
Substituting \eqref{oiu}-\eqref{opi} into \eqref{inew2w}, we obtain, for some $C_{28} > 0$,
\begin{align}
    \label{pkj}
    \nonumber &\frac{d}{dt}\bigg\{ \int_\Omega e^{\chi_E v} a^p
   +\int_\Omega e^{\chi_M v} u^p
    + \int_\Omega e^{-\chi_F v} w^p \bigg\} + \bigg\{\int_\Omega e^{\chi_E v} a^p +
   \int_\Omega e^{\chi_M v} u^p
    + \int_\Omega e^{-\chi_F v} w^p \bigg\} \\[5pt]
    & \leq K_{22} p^4 \max\left\{1, \|a^{\frac{p}{2}}\|_{L^1(\Omega)} + \|u^{\frac{p}{2}}\|_{L^1(\Omega)} + \|w^{\frac{p}{2}}\|_{L^1(\Omega)}\right\}^{\frac{2p}{p-2}}
\end{align}
for all $t \in (0, \widehat{T})$. Let $p_j = 4 \cdot 2^j$ and define
$$
M_j := \max \left\{1, \sup_{t \in (0, \widehat{T})}
\left(
\int_\Omega a^{p_j}(\cdot,t) + \int_\Omega u^{p_j}(\cdot,t) + \int_\Omega w^{p_j}(\cdot,t)
\right)\right\}
\quad \text{for } j=0,1,2,\ldots.
$$
Then \eqref{pkj} implies that for all $j = 1,2,\ldots$,
\begin{align}
    \label{pkj1}
    \nonumber  \frac{d}{dt}\bigg\{ \int_\Omega e^{\chi_E v} a^{p_j} + 
   \int_\Omega e^{\chi_M v} u^{p_j}
   & + \int_\Omega e^{-\chi_F v} w^{p_j} \bigg\} \\[5pt]
   & + \bigg\{ \int_\Omega e^{\chi_E v} a^{p_j} + 
   \int_\Omega e^{\chi_M v} u^{p_j}
    + \int_\Omega e^{-\chi_F v} w^{p_j} \bigg\}
    \leq K_{22} p_j^4 M_{j-1}^4
\end{align}
for all $t \in (0,\widehat{T})$. Hence, \eqref{pkj1} entails the existence of $B > 1$, independent of $j$, such that
$$
M_j \leq \max\left\{
B^j M_{j-1}^4,\,
|\Omega|\left(\|a_0\|_{L^\infty(\Omega)}^{p_j} + \|u_0\|_{L^\infty(\Omega)}^{p_j}
+ \|w_0\|_{L^\infty(\Omega)}^{p_j}\right)
\right\}
$$
for all $j \geq 1$. Therefore, by a standard recursive argument (see, e.g., \cite{tao2014energy, pang2018global}), we distinguish two cases.

If
\begin{align*}
B^j M_{j-1}^4 \leq |\Omega|\left(\|a_0\|_{L^\infty(\Omega)}^{p_j} + \|u_0\|_{L^\infty(\Omega)}^{p_j} +\|w_0\|_{L^\infty(\Omega)}^{p_j}\right) \quad \text{for infinitely many } j\geq 1.
\end{align*}
Then one directly obtains a uniform bound for $M_j^{\frac{1}{p_j}}$. In the opposite case, there exists a constant $K_{23}>0$ such that
$$
M_j^{\frac{1}{p_j}}\leq K_{23}
\qquad \text{for all } j\geq 1.
$$
Hence, in either case, $M_j^{\frac{1}{p_j}}$ remains bounded, and passing to the limit as $j\to\infty$, we readily obtain \eqref{c3llnfty12}-\eqref{c2llnfty12}.
\end{proof}

We are now in a position to establish the main result of this section asserting 2D global classical solvability for \eqref{model1}-\eqref{model3}.

\begin{proof}[Proof of Theorem \ref{2D_Class_Theorem}]
In view of the extensibility criterion \eqref{extensibility_criterion2} in Lemma \ref{extensibility_imp}, it suffices to invoke the conclusion of Lemma \ref{lemma_ablinfty} to obtain the desired result.
\end{proof}


\section{Global classical solutions to \texorpdfstring{\eqref{model3w}, \eqref{model2} and \eqref{model3}}{the system} in a three-dimensional domain}
\label{sec4}
In this section, we explore global classical solutions to \eqref{model1}-\eqref{model3} in a three-dimensional domain. To facilitate analysis in this higher-dimensional setting, we assume the EMT and MET rates in \eqref{model1} are positive constants, denoted by $\alpha_1>0$ and $\alpha_2>0$, respectively. Thus, we assume that phenotypic transitions occur at fixed rates, independently of TGF-$\beta$. This allows us to retain the mathematical features of bidirectional EMT-MET transitions while eliminating the additional coupling introduced by the transition functions $\zeta_1$ and $\zeta_2$, which constitute the only dependence on TGF-$\beta$ in \eqref{model1}. Consequently, we can omit the equation for $b$, and the system reduces to the following form, which will serve as the basis for the analysis in this section.
\begin{equation}
\label{model3w}
\begin{cases}
\partial_t c^{E}
 = \Delta c^{E} -\chi_{E}\nabla\cdot\!\big(c^{E}\nabla v\big)
-\alpha_1\,c^{E} + \alpha_2\,c^{M} +\,\rho_{E}\,c^{E}\big(1-c^{E}-c^{M}-c^{F}-v\big),
& x\in\Omega,\ t>0,
\\[7pt]
\partial_t c^{M}
 = \Delta c^{M}
-\chi_{M}\nabla\cdot\!\big(c^{M}\nabla v\big)
+\alpha_1 \,c^{E} - \alpha_2\,c^{M},
& x\in\Omega,\ t>0,
\\[7pt]
\partial_t c^{F}
 = \Delta c^{F}
+\chi_{F}\nabla\cdot\!\big(c^{F}\nabla v\big) +\rho_{F}\,(1+c^{E})\,c^{F}\big(1-c^{E}-c^{M}-c^{F}-v\big)
-\lambda_{F}\,c^{F},
&x\in\Omega,\ t>0,
\\[7pt]
\partial_t m
= \Delta m+\rho_{m}^{F}\,c^{F}+\rho_{m}^{E}\,c^{E}-\lambda_{m}\,m,
& x\in\Omega,\ t>0,
\\[7pt]
\partial_t v
= \rho_{v}\,c^{F}\big(1 - v\big) - \delta_1 mv - \delta_2 c^E v - \delta_3 c^M v,
& x\in\Omega,\ t>0,  
\end{cases}
\end{equation}
subject to the boundary conditions \eqref{model2} and satisfying the initial conditions \eqref{model3}. 

\begin{remark}
\label{rem_delta3}
The restriction $\delta_3=0$ in Theorem~\ref{3D_Class_Theorem} is not required for the derivation of the weighted $L^2$-estimates for the solution components $a$ and $w$ and, consequently, $L^p$ estimates for the same. Indeed, these estimates remain valid for $\delta_3>0$, provided that the additional condition
\begin{align*}
    \delta_3\chi_E< \frac{\rho_E}{2}
\end{align*}
is satisfied. The assumption $\delta_3=0$ is imposed only at the subsequent stage, where we derive higher-order $L^p$-estimates for $u$. 
\end{remark}
We first establish the most crucial estimate ensuring an a priori $L^2$-bound for $a$ and $w$.
\begin{lemma}
    \label{3d_lemma1}
  Let $T>0$. Suppose that all the assumptions of Theorem~\ref{3D_Class_Theorem}, except for the restriction $\delta_3=0$, are satisfied, and assume in addition the condition stated in Remark~\ref{rem_delta3}. Then there exists a constant $C(T)>0$ such that
 \begin{align}
    \label{uwsql2}
    \|a(\cdot,t)\|_{L^2(\Omega)} + \|w(\cdot,t)\|_{L^2(\Omega)}
   \leq C(T)
   \qquad \text{for all } t\in(0,\widehat{T}),
   \end{align}
     where $\widehat{T}:=\min\{T,T_{\max}\}$.
    \end{lemma}
\begin{proof}
  In view of the above discussion, we have from the $a$ equation
  \begin{align}
  \label{9u6}
  \nonumber & \frac{d}{dt} \int_\Omega e^{\chi_E v} a^2 = 2 \int_\Omega e^{\chi_E v} a a_t + \chi_E \int_\Omega e^{\chi_E v} a^2 v_t \\[5pt]
  \nonumber & = 2 \int_\Omega e^{\chi_E v} a \big\{ e^{-\chi_E v} \nabla \cdot \big(e^{\chi_E v} \nabla a\big) -\alpha_1 a + \alpha_2 e^{(\chi_M - \chi_E) v} u - \rho_v \chi_E e^{- \chi_F v} a w (1 - v) \\[5pt]
  \nonumber & + \rho_E a\left(1 - e^{\chi_E v} a - e^{\chi_M v} u - e^{- \chi_F v} w - v\right) + \delta_1 \chi_E a m v + \delta_2 \chi_E e^{\chi_E v} a^2 v + \delta_3 \chi_E e^{\chi_M v} a u v\big\} \\[5pt] 
  & + \chi_E \int_\Omega e^{\chi_E v} a^2\big\{\rho_v e^{-\chi_F v} w \big(1 - v\big) - \delta_1 m v - \delta_2 e^{\chi_E v} a v - \delta_3 e^{\chi_M v} uv\big\} 
  \end{align}
for all $t \in (0, T_{\max})$. To estimate the right-hand side, we treat each term separately. We first turn to
$2\alpha_2\int_\Omega e^{\chi_M v}au$, an application of Young's inequality in view of \eqref{L1ess} results in:
\begin{align}
    \label{eo8}
    \nonumber 2\alpha_2\int_\Omega e^{\chi_M v}au & \leq \rho_E \int_\Omega e^{(\chi_E + \chi_M)v} a^2 u + \frac{\alpha_2^2}{\rho_E} \int_\Omega e^{(\chi_M - \chi_E)v} u  \qquad \text{for all}~ t > 0 \\[5pt]
    & \leq \rho_E \int_\Omega e^{(\chi_E + \chi_M)v} a^2 u + \frac{\alpha_2^2}{\rho_E} \Lambda_M u^\ast(T) \qquad \text{for all}~ t \in (0, \widehat{T}). 
\end{align}
As straightforward consequences of \eqref{v_es1} we have the following three potentially unfavorable terms
\begin{align}
    \label{ob1}
    2 \delta_1 \chi_E \int_\Omega e^{\chi_E v} a^2 m v - \delta_1 \chi_E \int_\Omega e^{\chi_E v} a^2 m v = \delta_1 \chi_E \int_\Omega e^{\chi_E v} a^2 m v \leq \delta_1 \chi_E \int_\Omega e^{\chi_E v} a^2 m, \\[5pt]
    \label{ob2}
    2 \delta_2 \chi_E \int_\Omega e^{2\chi_E v} a^3 v - \delta_2 \chi_E \int_\Omega e^{2\chi_E v} a^3 v =  \delta_2 \chi_E \int_\Omega e^{2\chi_E v} a^3 v \leq  \delta_2 \chi_E \int_\Omega e^{2\chi_E v} a^3, \\[5pt]
    \label{ob3}
    2 \delta_3 \chi_E \int_\Omega e^{(\chi_E + \chi_M) v} a^2 u v - \delta_3 \chi_E \int_\Omega e^{(\chi_E + \chi_M) v} a^2 u v \leq  \delta_3 \chi_E \int_\Omega e^{(\chi_E + \chi_M) v} a^2 u
\end{align}
for all $t \in (0, T_{\max})$, and we also have
\begin{align}
    \label{po88}
    \nonumber & - 2 \rho_v \chi_E \int_\Omega e^{(\chi_E - \chi_F) v} a^2 w (1 - v) + \rho_v \chi_E \int_\Omega e^{(\chi_E - \chi_F) v} a^2 w (1 - v)  \\[5pt]
    & = -\rho_v \chi_E \int_\Omega e^{(\chi_E - \chi_F) v} a^2 w (1 - v) \leq 0
\end{align}
for all $t \in (0, T_{\max})$. We next estimate the term in \eqref{ob1}. An application of Young's inequality yields
\begin{align}
    \label{m1ess2}
    \nonumber \delta_1 \chi_E \int_\Omega e^{\chi_E v} a^2 m & = \int_\Omega \left(\frac{3}{4}\right)^{\frac{2}{3}} \rho_E^{2/3}e^{\frac{4}{3}\chi_E v}a^2 \times  \left(\frac{4}{3}\right)^{\frac{2}{3}} \frac{\delta_1 \chi_E}{\rho_E^{2/3}} e^{-\frac{\chi_E}{3} v} m \\[5pt]
    & \leq \frac{\rho_E}{2} \int_\Omega e^{2 \chi_E v} a^3 + \underbrace{\frac{16  \delta_1^3 \chi_E^3}{27\rho_E^2}}_{=:A_1} \int_\Omega e^{-\chi_E v} m^3 
\end{align}
for all $t \in (0, T_{\max})$. We next estimate the second term on the right-hand side of \eqref{m1ess2}. By means of the Gagliardo-Nirenberg inequality  \cite[(4.8)-(4.10)]{dai2022new} with some constant $C_{\mathrm{GN}}^{(3)} = C_{\mathrm{GN}}^{(3)}(\Omega) > 0$, \eqref{mes1} and \eqref{lemma_mnab}, we have
\begin{align}
    \label{mgagn}
    \nonumber A_1\|m\|_{L^3(\Omega)}^3 & \leq A_1 C_{\mathrm{GN}}^{(3)} \|\nabla m\|^{\frac{12}{5}}_{L^2(\Omega)} \|m\|^{\frac{3}{5}}_{L^1(\Omega)} + A_1 C_{\mathrm{GN}}^{(3)} \|m\|^3_{L^1(\Omega)} \\[5pt]
    & \leq A_1 C_{\mathrm{GN}}^{(3)} (M_2^\ast(T))^{\frac{6}{5}} \cdot (M_m(T))^{\frac{3}{5}} + A_1 C_{\mathrm{GN}}^{(3)} (M_m(T))^3 =: M^\ast_3(T)
\end{align}
for all $t \in (0, \widehat{T})$. Inserting \eqref{mgagn} into \eqref{m1ess2} results in 
\begin{align}
    \label{n1ess2}
   \delta_1 \chi_E \int_\Omega e^{\chi_E v} a^2 m \leq \frac{\rho_E}{2} \int_\Omega e^{2 \chi_E v} a^3 + M^\ast_3(T)
\end{align}
for all $t \in (0, \widehat{T})$. We now estimate the term $2\rho_E\int_\Omega e^{\chi_Ev}a^2.$
Using Young's inequality and \eqref{zL1_bound}, we obtain
\begin{align}
    \label{asw33}
    \nonumber 2 \rho_E \int_\Omega e^{\chi_E v} a^2  = \rho_E \int_\Omega e^{\chi_E v} a^{\frac{3}{2}} \times 2 a^{\frac{1}{2}} & \leq \frac{\rho_E}{2} \int_\Omega e^{2 \chi_E v} a^3 + 2 \rho_E \int_\Omega a \\[5pt]
    & \leq \frac{\rho_E}{2} \int_\Omega e^{2 \chi_E v} a^3 + 2 \rho_E a^\ast(T)
\end{align}
for all $t \in (0, \widehat{T})$. Inserting \eqref{ob2}, \eqref{ob3}, \eqref{n1ess2} and \eqref{asw33} into \eqref{9u6} results in 
\begin{align}
    \label{mainu}
    \nonumber & \frac{d}{dt} \int_\Omega e^{\chi_E v} a^2 + 2 \alpha_1 \int_\Omega e^{\chi_E v} a^2 + 2 \int_\Omega e^{\chi_E v} |\nabla a|^2  \leq (\delta_3 \chi_E - \rho_E) \int_\Omega e^{(\chi_E + \chi_M)v} a^2 u \\[5pt]
    & + \left(\delta_2 \chi_E - \frac{\rho_E}{2}\right) \int_\Omega e^{2\chi_E v} a^3 + \underbrace{\frac{\alpha_2^2}{\rho_E} \Lambda_M u^\ast(T) +  M^\ast_3(T) +  2 \rho_E a^\ast(T)}_{=:A_3(T)}
\end{align}
for all $t \in (0, \widehat{T})$. The assumptions $\delta_3\chi_E<\frac{\rho_E}{2}$ and $\delta_2\chi_E<\frac{\rho_E}{2}$ allow us to infer from \eqref{mainu} that
\begin{align}
    \label{lopp}
   \frac{d}{dt} \int_\Omega e^{\chi_E v} a^2 + 2 \alpha_1 \int_\Omega e^{\chi_E v} a^2 + 2 \int_\Omega e^{\chi_E v} |\nabla a|^2  \leq  A_3(T)
\end{align}
for all $t \in (0, \widehat{T})$. Proceeding analogously for the $w$-equation, we compute, 
\begin{align}
    \label{wuooo}
    \nonumber & \frac{d}{dt} \int_\Omega e^{-\chi_F v} w^2 + \int_\Omega e^{-\chi_F v} w^2   = 2 \int_\Omega e^{-\chi_F v} w w_t - \chi_F \int_\Omega e^{-\chi_F v} w^2 v_t + \int_\Omega e^{-\chi_F v} w^2 \\[5pt]
    \nonumber & = 2  \int_\Omega e^{-\chi_F v} w \Big\{ e^{\chi_F v} \nabla \cdot \big(e^{-\chi_F v} \nabla w\big) + \rho_F(1+e^{\chi_E v}a)w\big(1-e^{\chi_E v}a-e^{\chi_M v}u-e^{-\chi_F v}w-v\big)\\[5pt] 
    \nonumber &  -\lambda_F w -\chi_F\delta_1 mvw  -\chi_F\delta_2 e^{\chi_E v}avw 
        -\chi_F\delta_3 e^{\chi_M v}uvw +\chi_F\rho_v e^{-\chi_F v}w^2\big(1-v\big)\Big\} \\[5pt]
    & - \chi_F \int_\Omega e^{-\chi_F v} w^2 \big\{\rho_v e^{-\chi_F v} w \big(1 - v\big) - \delta_1 m v - \delta_2 e^{\chi_E v} a v - \delta_3 e^{\chi_M v} uv\big\} + \int_\Omega e^{-\chi_F v} w^2
\end{align}
for all $t \in (0, T_{\max})$. Using the positivity of the solution components, we readily have
\begin{align}
    \label{1we}
-2 \delta_1 \chi_F \int_\Omega e^{-\chi_F v} m v w^2 + \delta_1 \chi_F \int_\Omega e^{-\chi_F v} m v w^2  = - \delta_1 \chi_F \int_\Omega e^{-\chi_F v} m v w^2  \leq 0, \\[5pt]
   \label{1we1}
- 2 \delta_2 \chi_F \int_\Omega e^{(\chi_E - \chi_F)v} a v w^2 + \delta_2 \chi_F \int_\Omega e^{(\chi_E - \chi_F)v} a v w^2 = - \delta_2 \chi_F \int_\Omega e^{(\chi_E - \chi_F)v} a v w^2 \leq 0, \\[5pt]
   \label{1w77}
- 2 \delta_3 \chi_F \int_\Omega e^{(\chi_M - \chi_F) v} u vw^2 + \delta_3 \chi_F \int_\Omega e^{(\chi_M - \chi_F) v} u vw^2 = - \delta_3 \chi_F \int_\Omega e^{(\chi_M - \chi_F) v} u vw^2 \leq  0
\end{align}
for all $t \in (0, T_{\max})$. Again, as a straightforward consequence of \eqref{v_es1} we have     
\begin{align}
    \label{up77}
    \nonumber 2\rho_v \chi_F \int_\Omega e^{-2 \chi_F v} w^3(1-v) - \rho_v \chi_F \int_\Omega e^{-2 \chi_F v} w^3(1-v) & = \rho_v \chi_F \int_\Omega e^{-2 \chi_F v} w^3(1-v) \\[5pt]
    & \leq \rho_v \chi_F \int_\Omega  e^{-2 \chi_F v} w^3
\end{align}
for all $t \in (0, T_{\max})$. Inserting \eqref{1we}-\eqref{up77} into \eqref{wuooo} results in 
\begin{align}
    \label{weqz1}
    \nonumber& \frac{d}{dt} \int_\Omega e^{-\chi_F v} w^2 + \int_\Omega e^{-\chi_F v} w^2 + 2 e^{-\chi_F} \int_\Omega |\nabla w|^2 \leq (2 \rho_F + 1) \int_\Omega e^{-\chi_F v} w^2 +  \rho_v \chi_F \int_\Omega  e^{-2 \chi_F v} w^3 \\[5pt]
    \nonumber & -2\rho_F\int_\Omega e^{(2\chi_E-\chi_F)v}a^2w^2  -2\rho_F\int_\Omega e^{(\chi_M-\chi_F)v}uw^2 -2\rho_F\int_\Omega e^{(\chi_E+\chi_M-\chi_F)v}auw^2  \\[5pt]
    &-2\rho_F\int_\Omega e^{-2\chi_Fv}w^3 -2\rho_F\int_\Omega e^{(\chi_E-2\chi_F)v}aw^3  -2\rho_F\int_\Omega ve^{-\chi_Fv}w^2-2\rho_F\int_\Omega e^{(\chi_E-\chi_F)v}avw^2
\end{align}
for all $t \in (0, T_{\max})$.  Consider the term $(2 \rho_F + 1) \int_\Omega e^{-\chi_F v} w^2$. Using Young's inequality and \eqref{zL1_bound}, we obtain 
\begin{align}
    \label{w1ess2}
    \nonumber (2 \rho_F + 1) \int_\Omega e^{-\chi_F v} w^2  & =   \int_\Omega \sqrt{2 \rho_F} e^{-\chi_F v} w^{\frac{3}{2}} \times \frac{(2 \rho_F + 1) }{\sqrt{2 \rho_F}} w^{\frac{1}{2}} \\[5pt]
    \nonumber & \leq \rho_F \int_\Omega e^{-2 \chi_F v} w^3 + \frac{(2 \rho_F + 1)^2}{4 \rho_F} \int_\Omega w \\[5pt]
    & \leq \rho_F \int_\Omega e^{-2 \chi_F v} w^3 +  \underbrace{\frac{(2 \rho_F + 1)^2}{4 \rho_F}w^\ast(T)}_{=:A_4(T)}
\end{align}
for all $t \in (0, \widehat{T})$. Combining \eqref{w1ess2} with \eqref{weqz1}, and using the assumption $\rho_v\chi_F<\rho_F$, we infer, after dropping the nonpositive terms, that
\begin{align}
    \label{we33}
    & \frac{d}{dt} \int_\Omega e^{-\chi_F v} w^2 + \int_\Omega e^{-\chi_F v} w^2 + 2 e^{-\chi_F} \int_\Omega |\nabla w|^2 \leq  A_4(T)
\end{align}
for all $t \in (0, \widehat{T})$. On adding \eqref{lopp} and \eqref{we33}, we obtain the following ordinary differential inequality, where $\hat{\alpha}:=\min\{2\alpha_1,1\}$
\begin{align*}
    \frac{d}{dt} \left\{\int_\Omega e^{\chi_E v} a^2 + \int_\Omega e^{-\chi_F v} w^2\right\} + \widehat{\alpha} \left\{\int_\Omega e^{\chi_E v} a^2  + \int_\Omega e^{-\chi_F v} w^2\right\} \leq A_5(T) = A_3(T) + A_4(T)
\end{align*}
for all $t \in (0, \widehat{T})$. An ODE comparison argument then yields
\begin{align*}
    \int_\Omega e^{\chi_E v} a^2 + \int_\Omega e^{-\chi_F v} w^2 \leq A_6(T) := \max \left\{\frac{A_5(T)}{\widehat{\alpha}}, \int_\Omega e^{-\chi_E v_0} \cdot \left(c^E_0\right)^2 + \int_\Omega e^{\chi_F v_0} \cdot \left( c^F_0\right)^2 \right\}
\end{align*}
for all $t \in (0, \widehat{T})$. This readily yields \eqref{uwsql2} in view of \eqref{v_es1}.
\end{proof}
We now derive $L^p(\Omega)$ estimates for $a$ and $w$ for arbitrary $p>1$ by applying standard testing procedures together with suitable interpolation inequalities. This allows us to bootstrap the $L^2(\Omega)$ estimates established in the previous lemma to higher-order $L^p(\Omega)$ bounds.
\begin{lemma}
    \label{3dp_lemma1}
  Let $T>0$. Suppose that all the assumptions of Theorem~\ref{3D_Class_Theorem}, except for the restriction $\delta_3=0$, are satisfied, and assume in addition the condition stated in Remark~\ref{rem_delta3}. Then there exists a constant $C(p, T)>0$ such that for all $p > 1$
 \begin{align}
    \label{awlp}
    \|a(\cdot,t)\|_{L^p(\Omega)} + \|w(\cdot,t)\|_{L^p(\Omega)}
   \leq C(p, T)
   \qquad \text{for all } t\in(0,\widehat{T}),
   \end{align}
     where $\widehat{T}:=\min\{T,T_{\max}\}$.
\end{lemma}
\begin{proof}
   An application of the smoothing properties of the Neumann heat semigroup $\{e^{t\Delta}\}_{t\geq0}$ on $\Omega$ (see \cite[Lemma~1.3]{winkler_aggregation_2010}) to the fourth equation in \eqref{model3w} in view of \eqref{uwsql2} immediately yields
\begin{align}
\label{mlinf}
\|m(\cdot,t)\|_{L^\infty(\Omega)}
\leq M_{\infty}(T)
\qquad\text{for all }t\in(0,\widehat{T}).
\end{align}
We can easily calculate from the $a$ equation
\begin{align}
    \label{p111}
    \nonumber & \frac{d}{dt} \int_\Omega e^{\chi_E v} a^p + \int_\Omega e^{\chi_E v} a^p= p \int_\Omega e^{\chi_E v} a^{p - 1} a_t + \chi_E \int_\Omega e^{\chi_E v} a^p v_t + \int_\Omega e^{\chi_E v} a^p\\[5pt]
     \nonumber & = p \int_\Omega e^{\chi_E v} a^{p - 1} \big\{ e^{-\chi_E v} \nabla \cdot \big(e^{\chi_E v} \nabla a\big) -\alpha_1 a + \alpha_2 e^{(\chi_M - \chi_E) v} u  - \rho_v \chi_E e^{- \chi_F v} a w (1 - v)\\[5pt]
  \nonumber & + \rho_E a\left(1 - e^{\chi_E v} a - e^{\chi_M v} u - e^{- \chi_F v} w - v\right)  + \delta_1 \chi_E a m v + \delta_2 \chi_E e^{\chi_E v} a^2 v + \delta_3 \chi_E e^{\chi_M v} a u v\big\} \\[5pt]
  & + \chi_E \int_\Omega e^{\chi_E v} a^p\big\{\rho_v e^{-\chi_F v} w \big(1 - v\big) - \delta_1 m v - \delta_2 e^{\chi_E v} a v - \delta_3 e^{\chi_M v} uv\big\} + \int_\Omega e^{\chi_E v} a^p
\end{align}
for all $t \in (0, T_{\max})$. Consider the term $\alpha_2 p \int_\Omega e^{\chi_M v} a^{p - 1} u$. Thanks to Young's inequality and \eqref{zL1_bound} we have
\begin{align}
    \label{iue1}
    \nonumber \alpha_2 p \int_\Omega e^{\chi_M v} a^{p - 1} u & \leq \frac{p \rho_E}{2} \int_\Omega e^{(\chi_E + \chi_M) v} a^p u + \alpha_2^p \left(\frac{2(p - 1)}{p \rho_E}\right)^{p - 1}\int_\Omega e^{\chi_M v} e^{-(p - 1) \chi_E v} u \\[5pt]
    & \leq \frac{p \rho_E}{2} \int_\Omega e^{(\chi_E + \chi_M) v} a^p u + \alpha_2^p \left(\frac{2(p - 1)}{p \rho_E}\right)^{p - 1} \Lambda_M u^\ast(T)
\end{align}
    for all $t \in (0, \widehat{T})$. The first positive term on the right-hand side of \eqref{iue1} can be controlled by the proliferative term as follows
    \begin{align}
    \label{ie43}
        \nonumber & \frac{p \rho_E}{2} \int_\Omega e^{(\chi_E + \chi_M) v} a^p u +  \alpha_2^p \left(\frac{2(p - 1)}{p \rho_E}\right)^{p - 1}\Lambda_M u^\ast(T) - p\rho_E \int_\Omega e^{(\chi_E + \chi_M) v} a^p u \\[5pt]
        & = -\frac{p \rho_E}{2} \int_\Omega e^{(\chi_E + \chi_M) v} a^p u +  \alpha_2^p \left(\frac{2(p - 1)}{p \rho_E}\right)^{p - 1}\Lambda_M u^\ast(T) 
    \end{align}
    for all $t \in (0, \widehat{T})$. Thanks to \eqref{v_es1}, \eqref{mlinf} and Young's inequality we have
    \begin{align}
        \label{eq22}
        \nonumber (p - 1) \delta_1 \chi_E \int_\Omega e^{\chi_E v} a^p m v & \leq (p - 1) \delta_1 \chi_E M_\infty(T) \Lambda_E \int_\Omega a^p \\[5pt] & \leq  p \int_\Omega a^{p + 1} + \left(\frac{p}{p + 1}\right)^{p}  ((p - 1)\delta_1 \chi_E M_\infty(T) \Lambda_E)^{p + 1}|\Omega|
    \end{align}
    for all $t \in (0, \widehat{T})$. Also we have
    \begin{align}
        \label{opi21}
        (p - 1) \delta_2 \chi_E \int_\Omega e^{2\chi_E v} a^{p + 1} v \leq p \delta_2 \chi_E \Lambda_E^2 \int_\Omega a^{p + 1}
    \end{align}
    for all $t \in (0, T_{\max})$. Again, thanks to \eqref{v_es1} we obtain
    \begin{align*}
        (p - 1) \delta_3 \chi_E \int_\Omega e^{(\chi_E + \chi_M)v} a^p u v \leq  (p - 1) \delta_3 \chi_E \int_\Omega e^{(\chi_E + \chi_M)v} a^p u
    \end{align*}
  for all $t \in (0, T_{\max})$. The above term can be absorbed into a proliferative term by virtue of the assumption $\delta_3\chi_E<\frac{\rho_E}{2}$.
  \begin{align}
  \label{ooo22}
      p\left(\delta_3 \chi_E - \frac{\rho_E}{2}\right)  \int_\Omega e^{(\chi_E + \chi_M)v} a^p u - \delta_3 \chi_E  \int_\Omega e^{(\chi_E + \chi_M)v} a^p u < 0
  \end{align}
 for all $t \in (0, T_{\max})$. Also, we have 
 \begin{align}
 \label{new1}
     -(p - 1) \rho_v \chi_E \int_\Omega e^{(\chi_E - \chi_F)v}a^p w(1 - v) \leq 0
 \end{align}
 for all $t \in (0, T_{\max})$, and 
 \begin{align}
     \label{p9bb}
    (p \rho_E + 1)\int_\Omega e^{\chi_E v} a^p \leq p \int_\Omega a^{p + 1} + \Lambda_E^{p+1}\frac{(p \rho_E + 1)^{p + 1}}{(p + 1)^{p + 1}} |\Omega|
 \end{align}
for all $t \in (0, T_{\max})$. Inserting \eqref{ie43}-\eqref{p9bb} into \eqref{p111}, and neglecting the nonpositive terms, we arrive at
 \begin{align}
     \label{muzz}
      \frac{d}{dt} \int_\Omega e^{\chi_E v} a^p + \int_\Omega e^{\chi_E v} a^p & + \frac{4(p - 1)}{p} \int_\Omega e^{\chi_E v} |\nabla a^{\frac{p}{2}}|^2 \leq p \Pi_1 \int_\Omega a^{p + 1} + \Pi_2(p, T)
\end{align}
 for all $t \in (0, \widehat{T})$, where
 \begin{align*}
     & \Pi_1 := 2 + \delta_2 \chi_E \Lambda_E^2, \\[5pt]
     & \Pi_2(p, T) :=  \left(\frac{p}{p + 1}\right)^{p}  ((p - 1)\delta_1 \chi_E M_\infty(T) \Lambda_E)^{p + 1}|\Omega|  \\[5pt]
     & \hspace{1.5cm}  + \alpha_2^p \left(\frac{2(p - 1)}{p \rho_E}\right)^{p - 1}\Lambda_Mu^\ast(T) + \Lambda_E^{p+1}\frac{(p \rho_E + 1)^{p + 1}}{(p + 1)^{p + 1}} |\Omega|.
 \end{align*}
 Proceeding analogously for the $w$-equation, we compute, 
\begin{align}
    \label{wu98}
    \nonumber & \frac{d}{dt} \int_\Omega e^{-\chi_F v} w^p + \int_\Omega e^{-\chi_F v} w^p  = p \int_\Omega e^{-\chi_F v} w^{p - 1} w_t - \chi_F \int_\Omega e^{-\chi_F v} w^p v_t + \int_\Omega e^{-\chi_F v} w^p \\[5pt]
    \nonumber & = p  \int_\Omega e^{-\chi_F v} w^{p - 1} \Big\{ e^{\chi_F v} \nabla \cdot \big(e^{-\chi_F v} \nabla w\big) + \rho_F(1+e^{\chi_E v}a)w\big(1-e^{\chi_E v}a-e^{\chi_M v}u-e^{-\chi_F v}w-v\big)\\[5pt] 
    \nonumber &  -\lambda_F w -\chi_F\delta_1 mvw  -\chi_F\delta_2 e^{\chi_E v}avw 
        -\chi_F\delta_3 e^{\chi_M v}uvw +\chi_F\rho_v e^{-\chi_F v}w^2\big(1-v\big)\Big\} \\[5pt]
    & - \chi_F \int_\Omega e^{-\chi_F v} w^p \big\{\rho_v e^{-\chi_F v} w \big(1 - v\big) - \delta_1 m v - \delta_2 e^{\chi_E v} a v - \delta_3 e^{\chi_M v} uv\big\} + \int_\Omega e^{-\chi_F v} w^p
\end{align}
for all $t \in (0, T_{\max})$. Clearly, in \eqref{wu98} owing to the positivity of the solution components, we have
\begin{align}
    \label{o988}
    & -(p - 1) \delta_1 \chi_F \int_\Omega e^{-\chi_F v} m v w^p \leq 0, \\[5pt]
    & -(p - 1) \delta_2 \chi_F \int_\Omega e^{(\chi_E - \chi_F) v} a v w^p \leq 0  \qquad \text{and}\\[5pt]
    \label{o988ss}
    & -(p - 1) \delta_3 \chi_F \int_\Omega e^{(\chi_M - \chi_F) v} u v w^p \leq 0
\end{align}
for all $t \in (0, T_{\max})$. Also, we have the following control thanks to \eqref{v_es1} and $\rho_v \chi_F < \rho_F$
\begin{align}
\label{o9bb}
\nonumber & (p-1)\rho_v\chi_F\int_\Omega e^{-2\chi_Fv}w^{p+1}(1-v) -p\rho_F\int_\Omega e^{-2\chi_Fv}w^{p+1} \\[5pt]
&\leq -\Big(p(\rho_F-\rho_v\chi_F)+\rho_v\chi_F\Big) \int_\Omega e^{-2\chi_Fv}w^{p+1} \leq 0.
\end{align}

for all $t \in (0, T_{\max})$. Moreover, we can estimate $ p \rho_F \int_\Omega e^{(\chi_E - \chi_F)v} a w^p $ by means of Young's inequality
\begin{align}
    \label{png}
    p \rho_F \int_\Omega e^{(\chi_E - \chi_F)v} a w^p \leq p \int_\Omega w^{p + 1} + p p^p \frac{1}{(p + 1)^{p + 1}}\rho_F^{p + 1} \Lambda_E^{p + 1} \int_\Omega a^{p + 1}
\end{align}
for all $t \in (0, T_{\max})$. Also, we have
\begin{align}
     \label{99bb}
      & p \rho_F \int_\Omega e^{\chi_F v} w^p \leq  p \rho_F \int_\Omega  w^p \leq p^2 \rho_F \int_\Omega w^{p + 1} + \frac{p \rho_F}{(p + 1)^{p + 1}} |\Omega|, \\[5pt]
     \label{9bb}
     & \int_\Omega e^{-\chi_F v} w^p \leq \int_\Omega w^p \leq p \int_\Omega w^{p + 1} + \frac{1}{(p + 1)^{p + 1}} |\Omega|.
 \end{align}
for all $t \in (0, T_{\max})$. Inserting \eqref{o988}-\eqref{9bb} into \eqref{wu98} results in the following ordinary differential inequality after neglecting the nonpositive terms,
\begin{align}
    \label{wplk}
    \frac{d}{dt} \int_\Omega e^{-\chi_F v} w^p + \int_\Omega e^{-\chi_F v} w^p & + \frac{4(p-1)e^{-\chi_F}}{p}\int_\Omega |\nabla w^{\frac p2}|^2 \leq p \Pi_3 \int_\Omega a^{p + 1} + p \Pi_4 \int_\Omega w^{p + 1} + \Pi_5 
 \end{align}
 for all $t \in (0, T_{\max})$, where 
 \begin{align*}
     & \Pi_3 := p^p \frac{1}{(p + 1)^{p + 1}}\rho_F^{p + 1} \Lambda_E^{p + 1}, \\[5pt]
     & \Pi_4 := p \rho_F + 2, \\[5pt]
     & \Pi_5 := \frac{p \rho_F}{(p + 1)^{p + 1}} |\Omega| + \frac{1}{(p + 1)^{p + 1}} |\Omega|.
 \end{align*}
Combining \eqref{muzz} and \eqref{wplk} results in
 \begin{align}
     \label{mainuw1}
     \nonumber & \frac{d}{dt} \bigg\{\int_\Omega e^{\chi_E v} a^p + \int_\Omega e^{-\chi_F v} w^p\bigg\} + \bigg\{\int_\Omega e^{\chi_E v} a^p + \int_\Omega e^{-\chi_F v} w^p\bigg\} \\[5pt]
     \nonumber & + \frac{4(p - 1)}{p} \int_\Omega e^{\chi_E v} |\nabla a^{\frac{p}{2}}|^2 +  \frac{4(p - 1) e^{-\chi_F}}{p}  \int_\Omega |\nabla w^{\frac p2}|^2 \\[5pt] 
     &\leq   p \Pi_6 \int_\Omega a^{p+1} + p \Pi_4 \int_\Omega w^{p+1} + \Pi_7(p, T)
 \end{align}
 for all $t \in (0, \widehat{T})$, where
 \begin{align*}
      \Pi_6 := \Pi_1 + \Pi_3 \qquad \text{and} \qquad \Pi_7(p, T) := \Pi_2(p, T) + \Pi_5. 
 \end{align*}
 We are now in a position to apply the Gagliardo-Nirenberg inequality \cite[(4.24)]{dai2022new}. Thus, for some constant $C_{\mathrm{GN}}^{(3)}=C_{\mathrm{GN}}^{(3)}(p,\Omega)>0$, and in view of \eqref{uwsql2}, for any $\widehat{\Pi}>0$ we obtain
\begin{align}
\label{gag_3d33}
\nonumber p\widehat{\Pi} \|y\|_{L^{p+1}(\Omega)}^{p+1}
& = p\widehat{\Pi}\|y^{\frac p2}\|_{L^{\frac{2(p+1)}{p}}(\Omega)}^{\frac{2(p+1)}{p}} \\[5pt]
\nonumber &\leq p\widehat{\Pi}C_{\mathrm{GN}}^{(3)}\left(
\|\nabla y^{\frac p2}\|_{L^2(\Omega)}^{\frac{2\theta(p+1)}{p}}
\|y^{\frac p2}\|_{L^{\frac{4}{p}}(\Omega)}^{\frac{2(1-\theta)(p+1)}{p}}
+\|y^{\frac p2}\|_{L^{\frac{4}{p}}(\Omega)}^{\frac{2(p+1)}{p}}
\right) \\[5pt]
&\leq p\widehat{\Pi}\left(\|\nabla y^{\frac p2}\|_{L^2(\Omega)}^{\frac{2\theta(p+1)}{p}}+1\right),
\end{align}
where, $y \in \{a, w\}$ and 
$$
\theta=\frac{p(3p-3)}{(p+1)(3p-2)}\in(0,1),
$$
due to $p>1$, and hence,
$$
\theta\frac{2(p+1)}{p}=2-\frac{2}{3p-2}\in(0,2),
$$
for any $p>1$. By means of Young's inequality, we have the following estimates from \eqref{gag_3d33} for some $\Pi_8(p, T), \Pi_{9}(p, T) > 0$
\begin{align}
    \label{you1}
     & p \Pi_6\int_\Omega a^{p+1} \leq  \frac{4(p-1)}{p}\int_\Omega e^{\chi_E v} |\nabla a^{\frac p2}|^2 + \Pi_8(p, T) \qquad \text{and} \\[5pt]
     \label{you2}
     & 2 p \int_\Omega w^{p+1} \leq  \frac{4(p-1)e^{-\chi_F}}{p}\int_\Omega |\nabla w^{\frac p2}|^2 + \Pi_{9}(p, T),
\end{align}
for all $t \in (0, \widehat{T})$. Inserting \eqref{you1}-\eqref{you2} into \eqref{mainuw1} results in the following differential inequality with $\Pi_{10}(p, T) = \Sigma_{i \in \{7, 8, 9\}} \Pi_i(p, T)$
\begin{align*}
    \frac{d}{dt}\left\{
   \int_\Omega e^{\chi_E v} a^p
    + \int_\Omega e^{-\chi_F v} w^p
    \right\} + \left\{
   \int_\Omega e^{\chi_E v} a^p
    + \int_\Omega e^{-\chi_F v} w^p
    \right\} \leq  \Pi_{10}(p, T)
\end{align*}
 for all $t \in (0, \widehat{T})$.  An ODE comparison argument then yields 
 \begin{align}
     \label{pmain1}
     \int_\Omega e^{\chi_E v} a^p
    + \int_\Omega e^{-\chi_F v} w^p \leq \max \left\{\Pi_{10}(p, T), \int_\Omega e^{-(p - 1)\chi_E v_0} \cdot \left(c^E_0\right)^p  + \int_\Omega e^{(p - 1)\chi_F v_0} \cdot \left(c^F_0\right)^p\right\}
 \end{align}
 for all $t \in (0, \widehat{T})$. The desired estimate \eqref{awlp} now follows directly from combining \eqref{pmain1} and \eqref{v_es1}.
\end{proof}    
We now derive $L^p$ estimates for $u$ for arbitrary $p>2$ under the additional assumption $\delta_3=0$. Since $\rho_E>0$, the condition imposed in Remark~\ref{rem_delta3} holds true. Consequently, the conclusions of Lemmas~\ref{3d_lemma1} and \ref{3dp_lemma1} remain valid and will be used throughout the subsequent analysis.
\begin{lemma}
\label{lemma_uLp}
Let $T>0$. Suppose that all the assumptions of Theorem~\ref{3D_Class_Theorem} are satisfied. Then there exists a constant $C(p, T)>0$ such that for all $p > 2$
\begin{align}
\label{ulp11}
\|u(\cdot,t)\|_{L^p(\Omega)}
\leq C(p,T)
\qquad\text{for all }t\in(0,\widehat{T}),
\end{align}
where $\widehat{T}:=\min\{T,T_{\max}\}$.
\end{lemma}
\begin{proof}
   We can directly calculate from the $u$ equation
    \begin{align}
    \label{pu00}
     \nonumber & \frac{d}{dt} \int_\Omega e^{\chi_M v} u^p = p \int_\Omega e^{\chi_M v} u^{p - 1} u_t + \chi_M \int_\Omega e^{\chi_M v} u^p v_t   \\[5pt]
     \nonumber & = p \int_\Omega e^{\chi_M v} u^{p - 1} \Big\{ e^{-\chi_M v} \nabla \cdot \big(e^{\chi_M v} \nabla u\big) + \alpha_1 e^{(\chi_E -\chi_M) v} a - \alpha_2 u  \\[5pt]
     \nonumber &  +\delta_1 \chi_M mvu  + \delta_2 \chi_M e^{\chi_E v} a u v -\rho_v \chi_Me^{-\chi_F v} uw \big(1-v\big)\Big\} \\[5pt]
    &+  \chi_M \int_\Omega e^{\chi_M v} u^p \Big\{ \rho_v e^{-\chi_F v} w \big(1 - v\big) - \delta_1 m v - \delta_2 e^{\chi_E v} a v \Big\}
\end{align}
for all $t \in (0, T_{\max})$. Consider the term $p \alpha_1 \int_\Omega e^{\chi_E v} a u^{p - 1}$. Applying Young's inequality in view of \eqref{zL1_bound} results in 
\begin{align}
    \label{uew2}
\nonumber p \alpha_1 \int_\Omega e^{\chi_E v} a u^{p - 1} & \leq p \alpha_1 \int_\Omega  e^{\chi_E v} a u^p + \alpha_1 \left(\frac{p - 1}{p}\right)^{p - 1}  \int_\Omega  e^{\chi_E v} a \\[5pt]
& \leq p \alpha_1 \int_\Omega  e^{\chi_E v} a u^p + \alpha_1 \left(\frac{p - 1}{p}\right)^{p - 1} \Lambda_E a^\ast(T)
\end{align}
for all $t \in (0, \widehat{T})$. Thanks to \eqref{v_es1} and \eqref{mlinf} we have
\begin{align}
    \label{mz1}
    & (p - 1) \delta_1 \chi_M \int_\Omega e^{\chi_M v}m v u^p \leq p \delta_1 \chi_M M_\infty(T) \int_\Omega e^{\chi_M v} u^p
\end{align}
for all $t \in (0, T_{\max})$. Also, we have
\begin{align}
     \label{mz2}
    & (p - 1) \delta_2 \chi_M \int_\Omega e^{(\chi_E + \chi_M) v}a u^p v \leq  p \delta_2 \chi_M \int_\Omega e^{(\chi_E + \chi_M) v}a u^p 
\end{align}
for all $t \in (0, T_{\max})$. Moreover, 
\begin{align}
    \label{mz3}
-(p - 1) \rho_v \chi_M \int_\Omega e^{(\chi_M - \chi_F)v}wu^p(1 - v) \leq 0
\end{align}
for all $t \in (0, T_{\max})$. Combining \eqref{pu00} with \eqref{uew2}-\eqref{mz3}, we infer that,
\begin{align}
    \label{pu999}
     \nonumber \frac{d}{dt} \int_\Omega e^{\chi_M v} u^p + \int_\Omega e^{\chi_M v} u^p & +\frac{4(p - 1)}{p} \int_\Omega e^{\chi_M v} |\nabla u^{\frac{p}{2}}|^2 \\[5pt] & \leq p \Gamma_1 \int_\Omega a u^p + p \Gamma_2(p, T) \int_\Omega u^p + \Gamma_3(p, T)
\end{align}
for all $t \in (0, \widehat{T})$, with
\begin{align*}
&\Gamma_1 := \alpha_1\Lambda_E + \delta_2\chi_M \Lambda_E\Lambda_M, \\[5pt]
&\Gamma_2(p,T) := \delta_1\chi_M \Lambda_M M_\infty(T) + \Lambda_M, \\[5pt]
& \Gamma_3(p,T) := \alpha_1\left(\frac{p-1}{p}\right)^{p-1} \Lambda_E a^{\ast}(T).
\end{align*}
As the spatial dimension is $N=3$ and no $L^2$ estimate is available for $u$, the approach employed in the proof of Lemma~\ref{3dp_lemma1} to derive the $L^p \; (p > 1)$ estimates for $a$ and $w$ is no longer applicable. Instead, we adopt the approach developed in \cite[Lemma~3.6-3.7]{zheng2022global} to establish the corresponding $L^p \; (p > 2)$ estimate for $u$. To this end, with exponents
\begin{align*}
    \frac{Np + 2}{Np} \quad \text{and} \quad \frac{Np + 2}{2},
\end{align*}
we apply the H{\"o}lder's inequality to the first term on the right-hand side of \eqref{pu999} in view of \eqref{awlp}, which yields
\begin{align*}
    p\Gamma_1 \int_\Omega a u^p & \leq  p\Gamma_1 \left(\int_\Omega a^{\frac{Np + 2}{2}}\right)^{\frac{2}{Np + 2}} \cdot \left(\int_\Omega u^{\frac{Np + 2}{N}}\right)^{\frac{Np}{Np + 2}}\\[5pt]
     & \leq \Gamma_4(p, T) \left(\int_\Omega u^{\frac{Np + 2}{N}}\right)^{\frac{Np}{Np + 2}} \\[5pt]
    & = \Gamma_4(p, T) \|u^{\frac{p}{2}}\|^2_{L^{\frac{2(Np + 2)}{Np}}(\Omega)} \qquad \text{for all}~ t \in (0, \widehat{T}). 
\end{align*}
Since $N=3$ and $p>2$, it follows immediately that
\begin{align*}
\frac{2}{p}\leq\frac{2(Np+2)}{Np}\leq\frac{2N}{(N-2)_+}.
\end{align*}
Hence, an application of the Gagliardo-Nirenberg inequality, together with \eqref{zL1_bound}, yields
\begin{align}
    \label{gaga00}
     \nonumber & \Gamma_4(p, T) \|u^{\frac{p}{2}}\|^2_{L^{\frac{2(Np + 2)}{Np}}(\Omega)}  \\[5pt]
     \nonumber & \leq \Gamma_5(p, T) \left(\|\nabla u^{\frac{p}{2}}\|^{2 \frac{Np - \frac{N^2 p}{Np + 2}}{2 - N + Np}}_{L^2(\Omega)} \|u^{\frac{p}{2}}\|^{2 - 2 \frac{Np - \frac{N^2 p}{Np + 2}}{2 - N + Np}}_{L^{\frac{2}{p}}(\Omega)} + \|u^{\frac{p}{2}}\|^2_{L^{\frac{2}{p}}(\Omega)}\right) \\[5pt]
     & \leq  \Gamma_6(p, T)\left(\|\nabla u^{\frac{p}{2}}\|^{2 \frac{Np - \frac{N^2 p}{Np + 2}}{2 - N + Np}}_{L^2(\Omega)} + 1\right)
\end{align}
for all $t \in (0, \widehat{T})$ with some positive constants $\Gamma_5(p, T)$ and $\Gamma_6(p, T)$. For $p > 2$ and $N =3$ we have
\begin{align*}
    2 \frac{Np - \frac{N^2 p}{Np + 2}}{2 - N + Np} < 2.
\end{align*}
Applying Young's inequality to the term on the right-hand side of \eqref{gaga00}, we obtain for some constant $\Gamma_7(p, T) > 0$
\begin{align}
    \label{aauo}
    \Gamma_4(p, T) \|u^{\frac{p}{2}}\|^2_{L^{\frac{2(Np + 2)}{Np}}(\Omega)} \leq \frac{2(p - 1)}{p} \int_\Omega e^{\chi_M v} |\nabla u^{\frac{p}{2}}|^2 + \Gamma_7(p, T) \qquad \text{for all}~ t \in (0, \widehat{T}). 
\end{align}
An analogous argument, together with some simplification, yields
\begin{align}
    \label{up999y}
     p\Gamma_2(p, T) \int_\Omega u^p \leq  \frac{2(p - 1)}{p} \int_\Omega e^{\chi_M v} |\nabla u^{\frac{p}{2}}|^2 + \Gamma_8(p, T) \qquad \text{for all}~ t \in (0, \widehat{T}).
\end{align}
Inserting \eqref{aauo} and \eqref{up999y} into \eqref{pu999} results in the following ordinary differential inequality
\begin{align*}
    \frac{d}{dt} \int_\Omega e^{\chi_M v} u^p + \int_\Omega e^{\chi_M v} u^p \leq  \Gamma_9(p, T)
\end{align*}
for all $t \in (0, \widehat{T})$.  An ODE comparison argument then yields
\begin{align*}
    \int_\Omega e^{\chi_M v} u^p \leq \Gamma_{10}(p, T) := \max \left\{\Gamma_9(p, T), \int_\Omega e^{-(p - 1)\chi_M v_0} \cdot \left(c^M_0\right)^p  \right\}
\end{align*}
for all $t \in (0, \widehat{T})$. This readily yields \eqref{ulp11} in view of \eqref{v_es1}.
\end{proof}
With the above $L^p$ estimates at hand, we are now in a position to apply the Moser-Alikakos iteration technique to obtain the following $L^\infty$ boundedness result. The approach follows that of \cite[Lemma 3.10]{zheng2022global}; however, we include the details here for the sake of completeness.
\begin{lemma}
\label{lemma_linf}
Let $T>0$. Suppose that all the assumptions of Theorem~\ref{3D_Class_Theorem} are satisfied. Then there exists a constant $C(T)>0$ such that
\begin{align}
\label{linf}
\|a(\cdot,t)\|_{L^\infty(\Omega)} + \|u(\cdot,t)\|_{L^\infty(\Omega)} + \|w(\cdot,t)\|_{L^\infty(\Omega)}
\leq C(T)
\qquad\text{for all }t\in(0,\widehat{T}),
\end{align}
where $\widehat{T}:=\min\{T,T_{\max}\}$.
\end{lemma}
\begin{proof}
 As a consequence of the arguments employed in the proofs of Lemmas~\ref{3dp_lemma1} and \ref{lemma_uLp}, we obtain the following differential inequality for some $\kappa_1 > 0$:
 \begin{align}
     \label{inff1}
      \nonumber & \frac{d}{dt} \int_\Omega e^{\chi_E v} a^p + \int_\Omega e^{\chi_E v} a^p + \frac{4(p - 1)}{p} \int_\Omega e^{\chi_E v} |\nabla a^{\frac{p}{2}}|^2 \\[5pt]
      \nonumber & + \frac{d}{dt} \int_\Omega e^{\chi_M v} u^p + \int_\Omega e^{\chi_M v} u^p +\frac{4(p - 1)}{p} \int_\Omega e^{\chi_M v} |\nabla u^{\frac{p}{2}}|^2 \\[5pt]
      \nonumber & + \frac{d}{dt} \int_\Omega e^{-\chi_F v} w^p + \int_\Omega e^{-\chi_F v} w^p + \frac{4(p-1)e^{-\chi_F}}{p}\int_\Omega |\nabla w^{\frac p2}|^2 \\[5pt]
      & \leq p \kappa_1 \left\{\int_\Omega a^{p + 1} + \int_\Omega u^{p + 1} + \int_\Omega w^{p + 1} + 1\right\}
 \end{align}
for all $t \in (0, \widehat{T})$. Assuming $p > 2N$ entails that
\begin{align*}
    2 \leq \frac{2(p + 1)}{p} \leq \frac{2 p - N + 2}{p - N},
\end{align*}
this allows us to once again use the Gagliardo-Nirenberg inequality with some constant $C_{GN}^{(3)} > 0$ to have
\begin{align*}
\|\varphi\|_{L^{\frac{2(p+1)}{p}}(\Omega)}^{\frac{2(p+1)}{p}}
\leq
C_{GN}^{(3)} \|\nabla \varphi\|_{L^2(\Omega)}^{\frac{2N(p+2)}{p(N+2)}}
\|\varphi\|_{L^1(\Omega)}^{\frac{2(2p-N+2)}{p(N+2)}}+ C_{GN}^{(3)} \|\varphi\|_{L^1(\Omega)}^{\frac{2(p+1)}{p}},
\qquad \text{for all } \varphi\in W^{1,2}(\Omega).
\end{align*}
The above-mentioned Gagliardo-Nirenberg inequality, in conjunction with Young's inequality, allows us to have,e for some constants $\kappa_2, \kappa_3, \kappa_4 > 0$
\begin{align}
    \label{aagan1}
    \nonumber & \kappa_1p \|a\|^{p + 1}_{L^{p + 1}(\Omega)}  \\[5pt]
    \nonumber & = \kappa_1 p \|a^{\frac{p}{2}}\|^{\frac{2(p + 1)}{p}}_{L^{\frac{2(p + 1)}{p}}(\Omega)}  \\[5pt]
    \nonumber & \leq \kappa_2 p \|\nabla a^{\frac{p}{2}}\|_{L^2(\Omega)}^{\frac{2N(p + 2)}{p(N + 2)}}\|a^{\frac{p}{2}}\|_{L^1(\Omega)}^{\frac{2(2p-N+2)}{p(N+2)}} + \kappa_3 p  \|\nabla a^{\frac{p}{2}}\|_{L^1(\Omega)}^{\frac{2(p+1)}{p}}\\[5pt]
    \nonumber & \leq \min\{1, e^{-\chi_F}\} \|\nabla a^{\frac{p}{2}}\|^2_{L^2(\Omega)} \\[5pt]
    \nonumber & + \frac{2(p - N)}{p(N + 2)}\left(\min\{1, e^{-\chi_F}\}\frac{p(N + 2}{N(p + 2)}\right)^{-\frac{(pN + 2N)}{2(p - N)}}\kappa_2^{\frac{p(N + 2)}{2(p - N)}} p^{\frac{p(N + 2)}{2(p - N)}}\|a^{\frac{p}{2}}\|^{\frac{2p - N + 2}{p - N}}_{L^1(\Omega)}+ \kappa_3 p\|a^{\frac{p}{2}}\|_{L^1(\Omega)}^{\frac{2(p+1)}{p}} \\[5pt]
    & \leq  \min\{1, e^{-\chi_F}\} \|\nabla a^{\frac{p}{2}}\|^2_{L^2(\Omega)} + \kappa_4 p^{N + 2} \max\left\{1, \|a^{\frac{p}{2}}\|_{L^1(\Omega)}\right\}^{\frac{2p - N + 2}{p - N}}
\end{align}
for all $t \in (0, \widehat{T})$, where we have used the fact that $\frac{p(N + 2)}{2(p - N)} \leq N + 2$ by $p \geq 2N$. A similar procedure allows us to have
\begin{align}
\label{aa09}
     & \kappa_1p \|u\|^{p + 1}_{L^{p + 1}(\Omega)} \leq  \min\{1, e^{-\chi_F}\} \|\nabla u^{\frac{p}{2}}\|^2_{L^2(\Omega)} + \kappa_5 p^{N + 2} \max\left\{1, \|u^{\frac{p}{2}}\|_{L^1(\Omega)}\right\}^{\frac{2p - N + 2}{p - N}}, \qquad \text{and} \\[5pt]
\label{aant}
    & \kappa_1p \|w\|^{p + 1}_{L^{p + 1}(\Omega)} \leq  \min\{1, e^{-\chi_F}\} \|\nabla w^{\frac{p}{2}}\|^2_{L^2(\Omega)} + \kappa_6 p^{N + 2} \max\left\{1, \|w^{\frac{p}{2}}\|_{L^1(\Omega)}\right\}^{\frac{2p - N + 2}{p - N}}
\end{align}
for all $t \in (0, \widehat{T})$. Inserting \eqref{aagan1}-\eqref{aant} into \eqref{inff1} results in
\begin{align}
    \label{aambg}
     \nonumber & \frac{d}{dt}\bigg\{ \int_\Omega e^{\chi_E v} a^p
   +\int_\Omega e^{\chi_M v} u^p
    + \int_\Omega e^{-\chi_F v} w^p \bigg\} + \bigg\{\int_\Omega e^{\chi_E v} a^p +
   \int_\Omega e^{\chi_M v} u^p
    + \int_\Omega e^{-\chi_F v} w^p \bigg\} \\[5pt]
    & \leq \kappa_7 p^{N + 2} \max \left\{1, \|a^{\frac{p}{2}}\|_{L^1(\Omega)} + \|u^{\frac{p}{2}}\|_{L^1(\Omega)} + \|w^{\frac{p}{2}}\|_{L^1(\Omega)}\right\}^{\frac{2p - N + 2}{p - N}}
\end{align}
for all $t \in (0, \widehat{T})$. Define
\begin{align*}
    p_i = 2^i p_0, \qquad
M_i = \max \left\{
1,\,
\sup_{t\in(0,\widehat{T})}
\left(
\|a^{\frac{p}{2}}\|_{L^1(\Omega)}
+
\|u^{\frac{p}{2}}\|_{L^1(\Omega)}
+
\|w^{\frac{p}{2}}\|_{L^1(\Omega)}
\right)
\right\},
\quad \text{for } i=0,1,2,\ldots.
\end{align*}
Since $p_0 > 2N$ we have a $\kappa_8 > 0$ such that
\begin{align}
    \label{aambg1}
     \nonumber \frac{d}{dt}\bigg\{ \int_\Omega e^{\chi_E v} a^p
   +\int_\Omega e^{\chi_M v} u^p
    & + \int_\Omega e^{-\chi_F v} w^p \bigg\} \\[5pt]
    & + \bigg\{\int_\Omega e^{\chi_E v} a^p +
   \int_\Omega e^{\chi_M v} u^p
    + \int_\Omega e^{-\chi_F v} w^p \bigg\} \leq \kappa_8 p_i^{N + 2}M_{i - 1}^{\frac{2p_i - N + 2}{p_i - N}} 
\end{align}
for all $t \in (0, \widehat{T})$. Consequently, there exists $\mu>1$, independent of $i$, such that
\begin{align}
\label{moser_iteration}
M_i
\leq
\max\left\{
\mu^i M_{i-1}^{2+\epsilon_i},
\left(
1+\|a_0\|_{L^\infty(\Omega)}
+\|u_0\|_{L^\infty(\Omega)} + \|w_0\|_{L^\infty(\Omega)}
\right)^{p_i}
\right\},
\end{align}
where $\epsilon_i := \frac{N+2}{p_i-N} \leq \frac{\kappa_9}{2^i}, \; i\geq1$ and $\kappa_9 = \frac{(N + 2)2^i}{2^i p_0 - N} \leq \frac{2(N + 2)}{2p_0 - N}$. In particular, $\sum_{i=1}^{\infty}\epsilon_i<\infty.$

If
\begin{align*}
\mu^iM_{i-1}^{2+\epsilon_i}
\leq
\left(
1+\|a^{\frac{p}{2}}\|_{L^1(\Omega)}
+\|u^{\frac{p}{2}}\|_{L^1(\Omega)}
+ \|w^{\frac{p}{2}}\|_{L^1(\Omega)}\right)^{p_i}
\end{align*}
for infinitely many $i \geq 1$, then \eqref{moser_iteration} directly yields the
desired boundedness. Otherwise, for all sufficiently large $i$,
\begin{align*}
M_i\leq\mu^iM_{i-1}^{2+\epsilon_i}.
\end{align*}
Iterating this inequality gives
\begin{align}
\label{moser_iteration2}
\nonumber
M_i \leq \mu^i\left(\mu^{i-1}M_{i-2}^{2+\epsilon_{i-1}}\right)^{2+\epsilon_i}
& =\mu^{\,i+(i-1)(2+\epsilon_i)}M_{i-2}^{(2+\epsilon_{i-1})(2+\epsilon_i)}\\[7pt]
&\leq\mu^{\,i+\sum_{j=2}^{i}(j-1)\prod_{k=j}^{i}(2+\epsilon_k)}M_0^{\prod_{k=1}^{i}(2+\epsilon_k)}.
\end{align}
Since $p_i=2^ip_0$ and $\sum_{i=1}^{\infty}\epsilon_i<\infty$, taking
$p_i$-th roots in \eqref{moser_iteration2} and passing to the limit
$i\to\infty$ yields some $C(T)>0$ such that
\begin{align*}
\sup_{t\in(0,\widehat{T})}\left\{\|a(\cdot,t)\|_{L^\infty(\Omega)}+\|u(\cdot,t)\|_{L^\infty(\Omega)}
+ \|w(\cdot,t)\|_{L^\infty(\Omega)}\right\}
\leq C(T).
\end{align*} 
\end{proof}

We are now in a position to establish the main result of this section, asserting 3D global classical solvability for \eqref{model3w}, \eqref{model2} and \eqref{model3}.

\begin{proof}[Proof of Theorem \ref{3D_Class_Theorem}]
In view of the extensibility criterion \eqref{extensibility_criterion2} in Lemma \ref{extensibility_imp}, it suffices to invoke the conclusion of Lemma \ref{lemma_linf} to obtain the desired result.
\end{proof}

\section{Numerical simulations}\label{sec5}
In this section, we present some indicative numerical simulations for model \eqref{model1}-\eqref{model3} in two spatial dimensions, as well as the modified model \eqref{model3w}, \eqref{model2} and \eqref{model3} in three spatial dimensions. The simulations incorporate the necessary conditions in Theorems \ref{2D_Class_Theorem} and \ref{3D_Class_Theorem} on the model parameters for the existence of solutions. 
The numerical solutions for the PDE systems are obtained using an IMEX Finite Volumes/Finite Differences method, introduced in \cite{SfakKolbe2016, hellmann2016mathematical, kolbe2021modeling, Kolbe2022glioma}, where numerical simulations were performed using MATLAB R2025a. More details on the numerical scheme are given in Appendix \ref{app:num_scheme}.

\subsection{Simulations for \texorpdfstring{\eqref{model1}-\eqref{model3}}{the system} in two spatial dimensions}
\label{sec:numsol2D}
Let us start with the model \eqref{model1}-\eqref{model3} in a two-dimensional square domain. We set the domain $\Omega = [0,4]\times[0,4]$. The EMT-MET switch functions are the ones given in \eqref{zeta_Sim} with  $L = 1$, $k = 5\times10^4$ and $b_T = 5\times10^{-3}$. For the initial conditions of the model, the density of ECCs is given by:
\begin{equation*}
    c^E_0(x,y) = 0.8\exp \left(  -\frac{(x-2)^2 + (y-2)^2}{0.04} \right), \quad \text{for } (x,y)\in\Omega,
\end{equation*}
while we assume there are no MCCs initially present i.e. $c^M_0 (x,y) = \mathbf{0}$. As far for the initial densities of the CAFs, MMPs, TGF-$\beta$ and ECM, we assume an initial smooth random configuration as shown in Figure \ref{fig:2Dsims}. The construction of these initial conditions is described in more detail in Appendix \ref{app:initial condition} and the values of the model parameters chosen for the simulation are shown in Table \ref{tab:params}, and satisfy the conditions of Theorem \ref{2D_Class_Theorem}. 

The simulation exhibits key model characteristics and captures the dynamical behavior of tumor growth and invasion of local tissue. Initially, EMT occurs in areas of high TGF- $\beta$ concentration, leading to the formation of a small population of MCCs that expands through the tissue via haptotaxis and diffusion. ECM degradation by cancer cells and MMPs creates ``local'' ECM gradients that promote cell migration and drive rapid MCC expansion to the tissue edges. In these damaged or low-ECM-density areas, we observe progressive CAF accumulation, which performs negative haptotaxis to repair the collagen matrix. The presence of CAFs in these low-ECM-density areas results in the overexpression of TGF-$\beta$, which, in effect, leads to the creation of a smaller MCC colony away from the primary region of expansion, as shown in the last column of Figure \ref{fig:2Dsims}. Let us note here that Theorem \ref{2D_Class_Theorem} shows the existence of solutions of the model when the haptotaxis coefficients are at most of the same order of magnitude as the diffusion of the cells, i.e., $ \approx \mathcal {O} (1)$. Hence, the spatial spreading of the cancer cells comes from a balanced combination of these two effects, with neither dominating the other.


\begin{figure}
    \centering
    \begin{tabular}{ccccc}
         & \hspace{-0.7cm}$t=0$ & \hspace{-0.7cm}$t=0.023$ & \hspace{-0.7cm}$t=0.114$ & \hspace{-0.7cm}$t=0.231$  \\
         \rotatebox{90}{\hspace{2.4em}{ ECM \hspace{1.7cm} MMPs \hspace{1.6cm} TGF-$\beta$ \hspace{1.7cm} CAFs \hspace{1.8cm} MCCs \hspace{1.7cm} ECCs}} 
         & 
         \includegraphics[width=0.22\linewidth, trim={37em 6em 35em 4em},clip]{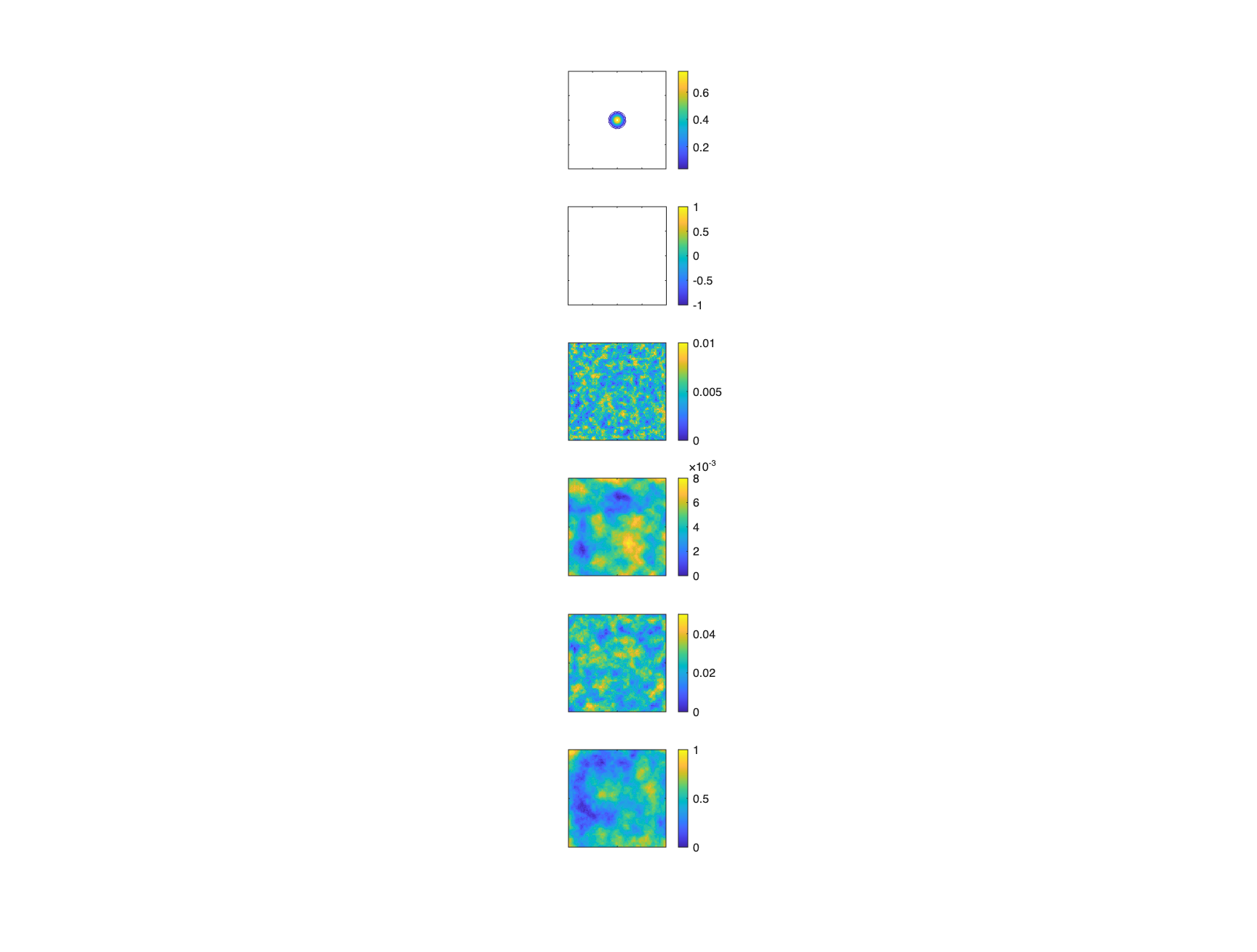} &
         \includegraphics[width=0.22\linewidth, trim={37em 6em 35em 4em},clip]{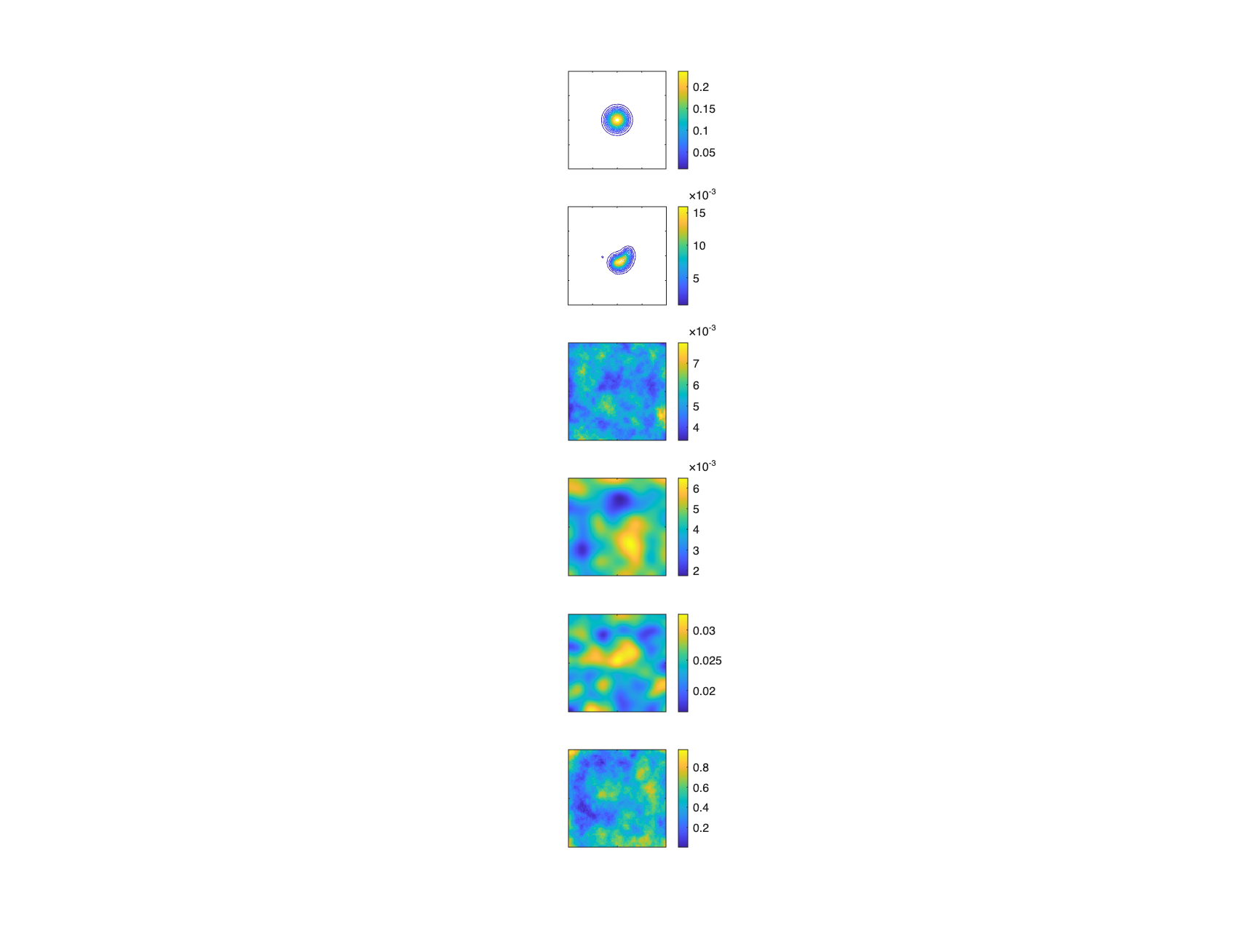} &
         \includegraphics[width=0.22\linewidth, trim={37em 6em 35em 4em},clip]{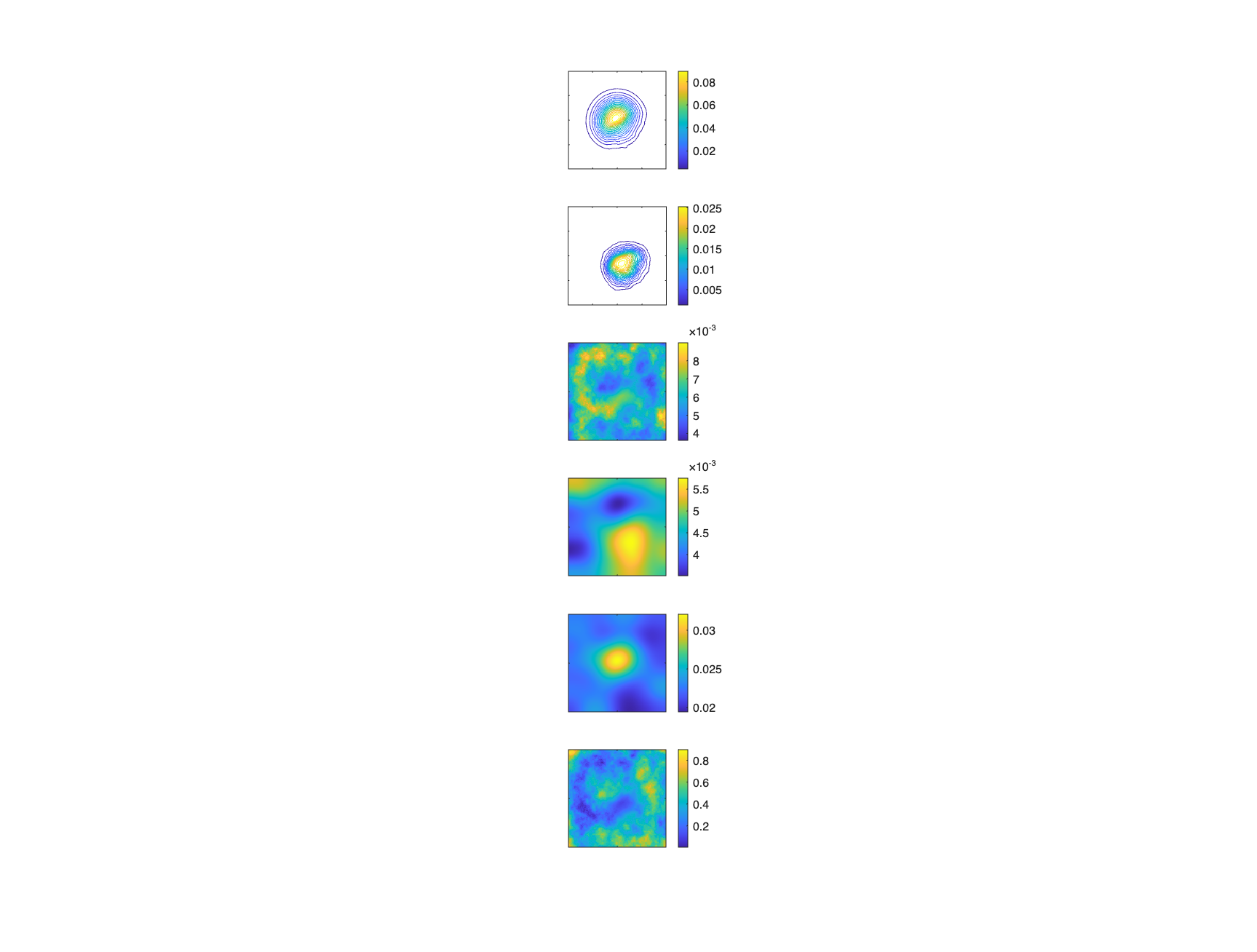} &
         \includegraphics[width=0.22\linewidth, trim={37em 6em 35em 4em},clip]{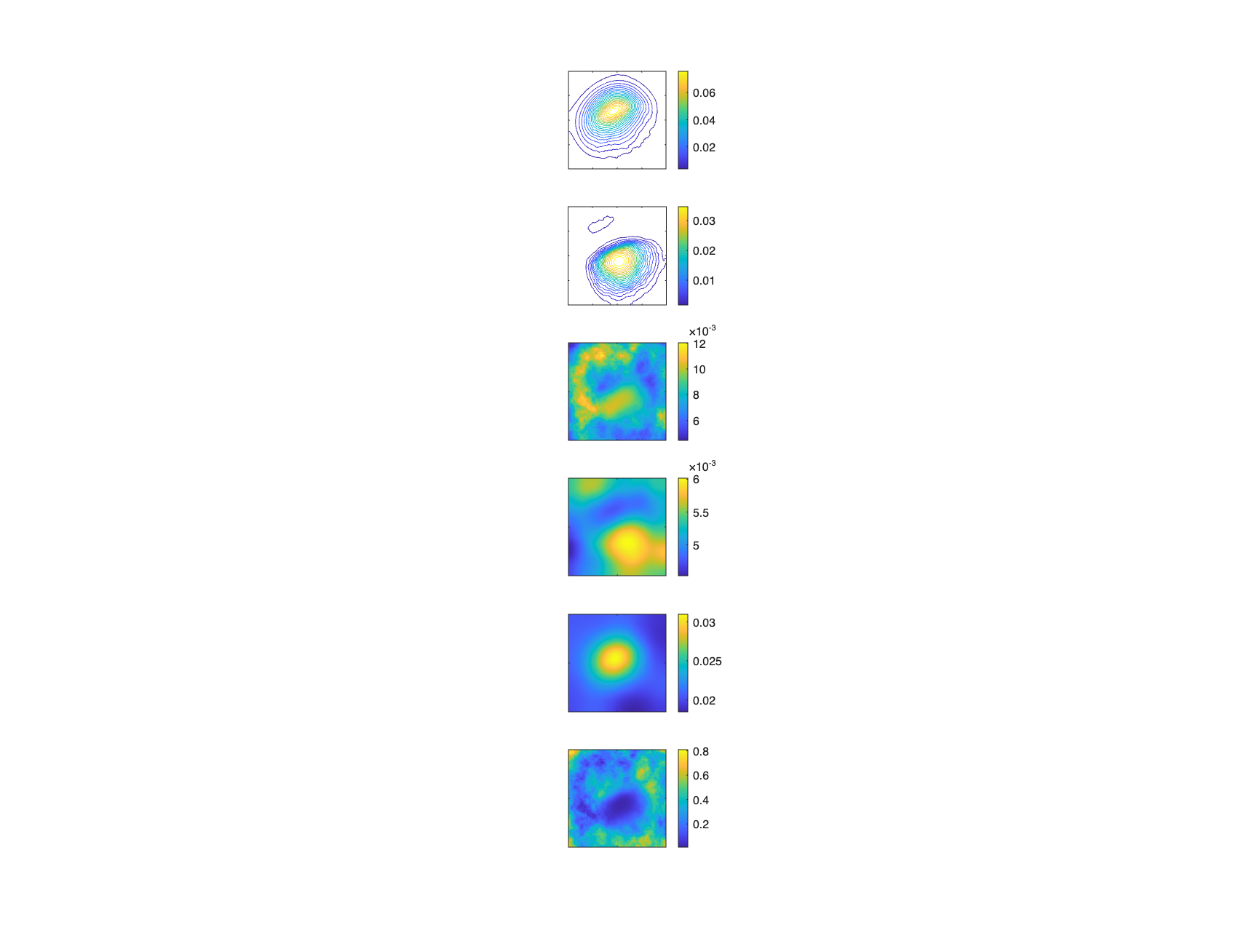} 
    \end{tabular}
    \caption{Numerical solution for \eqref{model1}-\eqref{model3} in a two dimensional spatial domain. We show the spatio-temporal evolution of the system for four different time snapshots (columns), for all the densities (rows).}
    \label{fig:2Dsims}
\end{figure}

\subsection{Simulations for \texorpdfstring{\eqref{model3w}, \eqref{model2} and \eqref{model3}}{the system} in three spatial dimensions}
\label{sec:numsol3D}
We continue with the model \eqref{model3w}, \eqref{model2} and \eqref{model3} in a three-dimensional domain,  
 where we now set $\Omega = [0,4]\times[0,4]\times[0,4]$, and the simulation parameters are shown in Table \ref{tab:params}. For the initial conditions, we yet again assume that there no MCCs in the domain, i.e. $c^M(x,y,z) = \mathbf{0}$, and the initial distribution of ECCs is given by:
\begin{equation*}
    c^E_0(x,y,z) = 0.8\exp \left(  -\frac{(x-2)^2 + (y-2)^2 + (z-2)^2}{0.04} \right), \quad \text{for } (x,y,z)\in\Omega.
\end{equation*}
For the initial density distributions of the CAFs, MMPs, and ECM, we use a similar strategy as before: we construct smooth random configurations as described in Appendix \ref{app:initial condition} and shown in Figure \ref{fig:Initial condition}.

\begin{figure}
    \centering
    \begin{tabular}{cc}
         & \hspace{-1.5cm} ECCs \hspace{6.2cm} MCCs   \\
         \rotatebox{90}{\hspace{8em}{$t=0.01$}} & \includegraphics[width=0.85\linewidth, trim={20em 8em 10em 3em},clip]{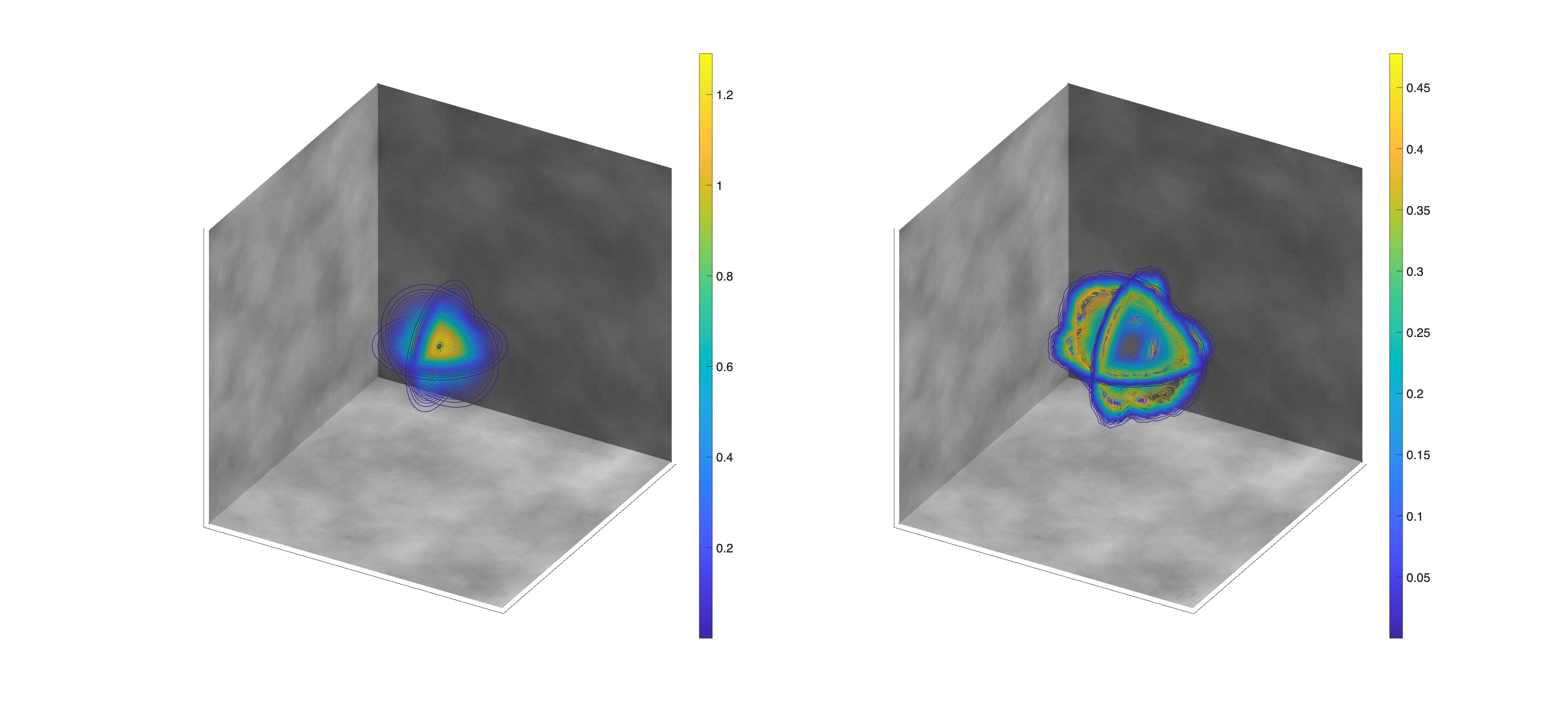} \\
         \rotatebox{90}{\hspace{8em}{$t=0.1$}} & \includegraphics[width=0.85\linewidth, trim={20em 8em 10em 3em},clip]{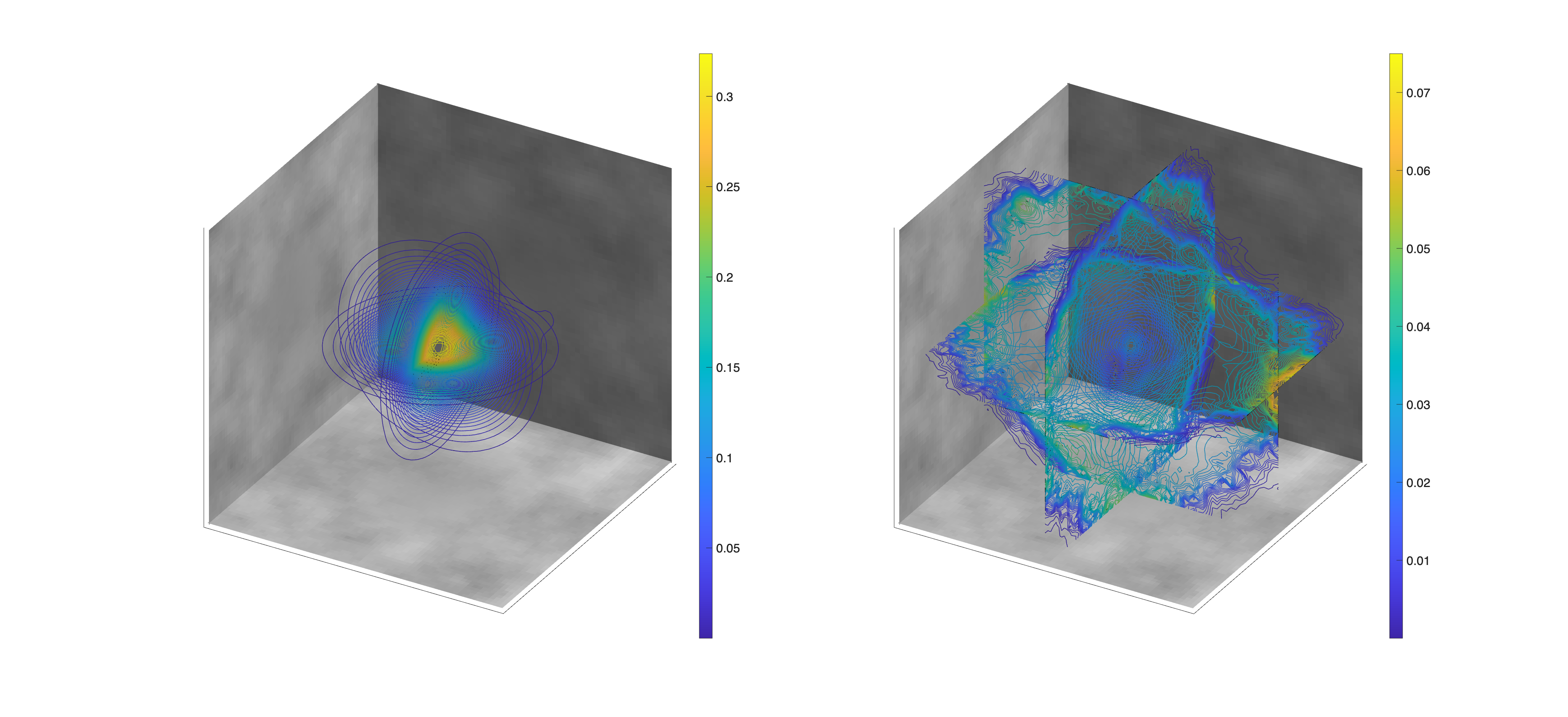} \\
         \rotatebox{90}{\hspace{8em}{$t=0.3$}} & \includegraphics[width=0.85\linewidth, trim={20em 8em 10em 3em},clip]{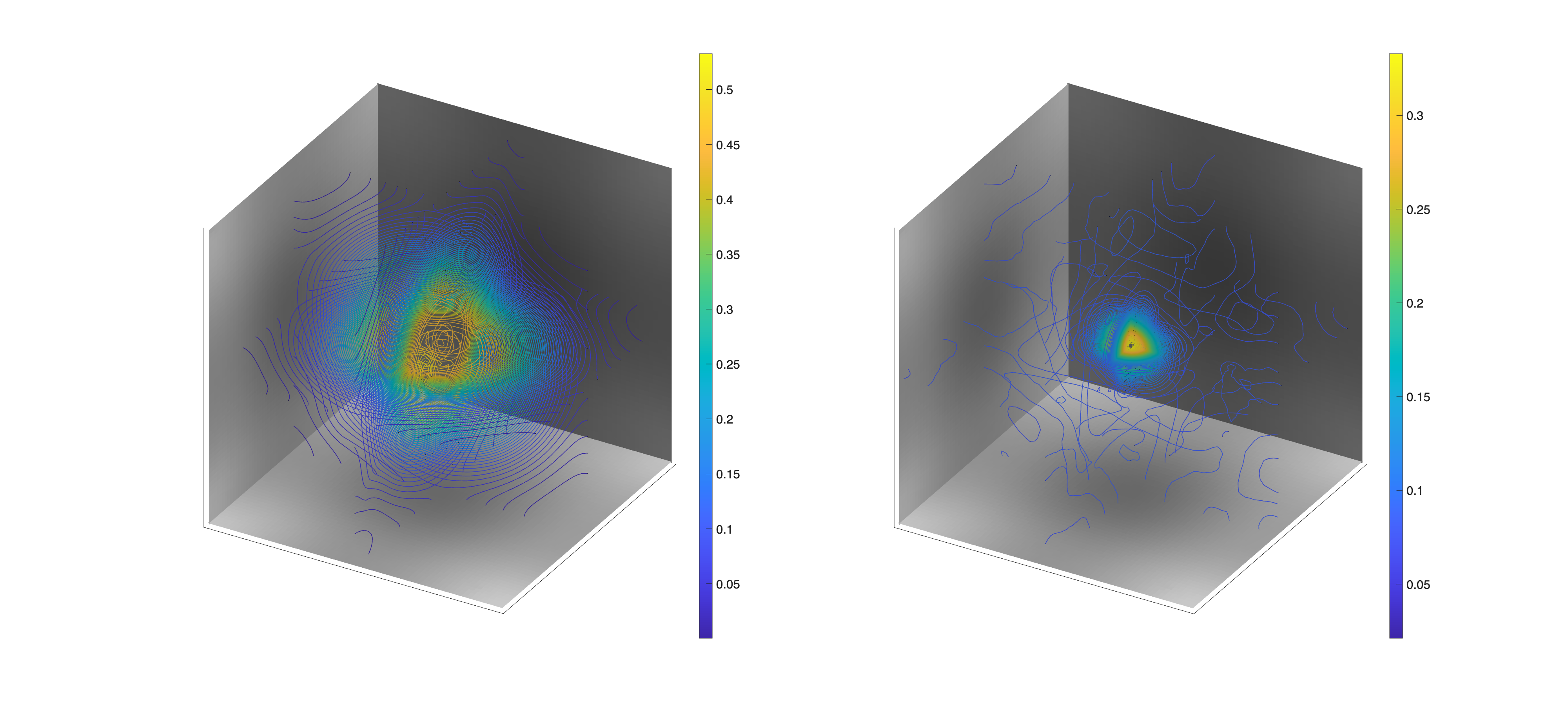} \\
    \end{tabular}
    \caption{Numerical solution for \eqref{model3w}, \eqref{model2} and \eqref{model3} in a three dimensional spatial domain. We show the spatio-temporal evolution of the system for three different time snapshots (rows), for the ECC and MCC densities (rows).}
    \label{fig:3Dsims}
\end{figure}

The simulation results are shown in Figure \ref{fig:3Dsims}. For illustration, we show only the spatiotemporal progression of the two cancer cell densities. In particular, we plot the contour lines of the two-dimensional planes, $x=2, y=2, z=2$, and track the evolution of the cancer cells in all directions of the three-dimensional domain. The grayscale background represents three different slices of the ECM density at the edges of the domain, i.e. $x=0, y=0, z=0$. The simplified model \eqref{model3w}, assumes that phenotypic transitions are happening at a constant rate in the entirety of the domain, neglecting their dependence on TGF-$\beta$. For that reason, EMT initially occurs in the center of the domain occupied by the ECCs, and the newly formed MCCs begin their haptotactic motion through the ECM. We note here that Theorem \ref{3D_Class_Theorem} doesn't pose any restrictions on the haptotaxis coefficient of the MCCs, so we have chosen the value $\chi_M=10$, which is one order of magnitude larger than the diffusion of the cells. Hence, the MCCs begin their ECM-dependent directed motion and reach the edges of the domain. The reverse phenotypic transition of MET now takes place in regions further away from the primary tumor site, which leads to the creation of small metastatic colonies of ECCs in the outer edges of the domain, as shown in Figure \ref{fig:3Dsims} at the final time $t=0.3$. Additionally, because more ECCs accumulate at the domain edges, we observe ECM degradation at the final simulation time.


\section{Discussion \& conclusions}\label{sec6}
 In this work, we study the existence of global classical solutions for a cancer invasion and metastasis model that includes phenotypic transitions of epithelial-like and mesenchymal-like cancer cells, in both two and three spatial dimensions. The model accounts for cancer cell migration due to the ECM, where cells perform haptotaxis and move toward denser regions of the matrix. One key component of the model is the CAFs, which, firstly, perform negative haptotaxis and move to low ECM density regions to repair the matrix and secondly, produce TGF-$\beta$, which is a triggering factor for the phenotypic transitions between the two cancer cell families; i.e., high TGF-$\beta$ concentration promotes EMT, whereas low concentration promotes the reverse process, MET. 

 One main challenge in the analysis is handling the nonproliferative MCC population, which undergoes haptotactic migration without logistic damping. Bidirectional EMT-MET transitions and different proliferation mechanisms across cell populations introduce further coupling, making the derivation of suitable a priori estimates more involved. 
 A further difficulty is the presence of MMPs, which introduces a nonlinear coupling among the ECCs, MCCs, CAFs, and the ECM. In particular, MMP production by ECCs and CAFs promotes ECM degradation, which in turn affects the gradients governing haptotactic migration. This feedback mechanism introduces additional nonlinear interactions, making it harder to derive and close the required a priori estimates. By deriving suitable uniform bounds and controlling these nonlinear interactions, we extend the existing global solvability results in \cite{TaoWinklerPreprint1, TaoWinkler2, Li2026Haptotaxis} to a more fully coupled cancer invasion model. The numerical results, shown in Section \ref{sec5}, illustrate the results of Theorems \ref{2D_Class_Theorem} \& \ref{3D_Class_Theorem} in two and three spatial dimensions, respectively, and showcase the spatiotemporal progression of a metastatic tumor colony and the creation of cancer cell islands away from the primary tumor. 

 However, the current version of the model includes several limitations that we would like to address and explore in future work. First, the model assumes linear diffusion for all cellular components, in contrast to the reference model in \cite{katsaounis2024genuinely}, which incorporates a nonlinear porous-medium diffusion term. In a cellular context, porous medium diffusion should be presumed as the spatial spreading of cells due to the effect of the pressure imposed by their neighbors. Additionally, in the original model in \cite{katsaounis2024genuinely}, the presence of high TGF-$\beta$ concentrations was only a necessary but not a sufficient condition for EMT and MET, and it additionally relied on a random stochastic process for the phenotypic transitions to take place. Finally, the model takes into account of only two families of cancer cells, ECCs and MCCs, assuming a binary instantaneous switch between them. Still, several studies suggest that EMT should be considered as a continuous process, where intermediate partial-epithelial-mesenchymal cells are present and contribute to the overall metastatic progression, \cite{Tsuji2009, Pastushenko2019}. Recent modeling attempts have included the continuous phenotypic spectrum in the form of phenotype-structure PDEs, but only a few analytical results exist for these problems \cite{Chisholm2016, Celora2023, Lorenzi2024, Lorenzi2025, Carrillo2025, Lorenzi2026, Katsaounis2026}.

 Regardless, this work takes an important step toward rigorous mathematical analysis of biologically realistic cancer invasion models. By establishing global classical solvability while retaining complex interactions among cell populations, phenotypic transitions, and ECM remodeling, our results provide a mathematical foundation for studying more elaborate models of cancer invasion and metastasis.

\subsubsection*{Acknowledgments}
NM would like to express his sincere gratitude to Prof.\ Dr.\ Michael Winkler, Institut f\"ur Mathematik, Universit\"at Paderborn, for kindly sharing the preprints \cite{TaoWinklerPreprint1, TaoWinkler2}. The generosity is deeply appreciated.

\noindent 
NM would also like to thank Prof.~Dr.~Christina Surulescu, RPTU Kaiserslautern-Landau, for valuable discussions during the development of this manuscript.

\noindent
NM and DK acknowledge funding from the Deutsche Forschungsgemeinschaft (DFG) via SPP 2311 \textit{Robust coupling of continuum-biomechanical in silico models to establish active biological system models for later use in clinical applications – Co-design of modeling, numerics, and usability} (Project Numbers: 465242756 \& 548864771).

\appendix

\section{Numerical solution of PDE systems}\label{app:numPDEs}

The advection-reaction-diffusion systems are solved numerically using a second-order {Implicit-Explicit Runge-Kutta Finite Differences, Finite Volumes} (IMEX-RK FD-FV) numerical method. This method extends a previous method developed and used in \cite{SfakKolbe2016, hellmann2016mathematical, kolbe2021modeling, Kolbe2022glioma}, and we refer there for most details. Here, we only discuss some of its components. 

\subsection{Numerical scheme}\label{app:num_scheme}
We consider, at first, a generic advection-reaction-diffusion system of the form 
\begin{equation}\label{ARD}
    \mathbf{w}_{t} = A(\mathbf{w}) + R(\mathbf{w}) + D(\mathbf{w}),
\end{equation}
where $\mathbf{w}$ represents the analytical solution vector of the system, and $A,R,$ and $D$ are the advection, reaction and diffusion operators respectively. After spatial discretizations have taken place, we denote the corresponding semi-discrete approximation by $\mathbf{w}_{h}$, where the index $h$ denotes the (maximal, if the space discretization is non-uniform) spatial grid diameter. The semi-discrete solution $\mathbf w_h$ satisfies the following system of ODEs
\begin{equation}\label{numARD}
    \partial_{t}\mathbf{w}_{h} = \mathcal{A}(\mathbf{w}_h) + \mathcal{R}(\mathbf{w}_h) + \mathcal{D}(\mathbf{w}_h),
\end{equation}
where the numerical operators $\mathcal{A}, \mathcal{R}$ and $\mathcal{D}$ are (spatially) discrete approximations of the advection-reaction, diffusion operators, $A, R$ and $D$ in equation \eqref{ARD}. Moreover, as the numerical scheme we employ is (partially) FV, raising its accuracy to the second order necessitates the use of {flux limiters} for the interface reconstruction of the numerical fluxes. Out of a large number of flux limiter options, we have found that the {Minimized-Central} (MC) limiter, see \cite{van1977towards}, constitutes a robust and efficient choice. 

Before solving \eqref{numARD}, we re-organize its terms in implicit and explicit (IMEX splitting) and, accordingly, \eqref{numARD} takes the form
\begin{equation}\label{IMEX}
    \partial_{t}\mathbf{w}_{h} = \mathcal{I}(\mathbf{w}_{h}) + \mathcal{E}(\mathbf{w}_{h}).
\end{equation} 
The actual IMEX splitting depends on the problem at hand, but in a typical case the advection terms $\mathcal{A}$ are treated explicitly in time, the diffusion terms $\mathcal{D}$ implicitly, and the reaction $\mathcal{R}$ terms either explicitly or implicitly, depending on their stiffness. In the problems we encounter in this scenario, we resolve all reaction terms explicitly in time.

The semi-discrete problem \eqref{IMEX} is now solved using a diagonally implicit RK method for the implicit part $\mathcal{I}({w}_{h})$, and an explicit RK for the explicit part $\mathcal{E}(\mathbf{w}_{h})$. Altogether, we solve \eqref{IMEX} using the additive RK scheme
\begin{equation}\label{IMEX-RK}
    \begin{dcases}
	\mathbf{w}^{*}_{i} = \mathbf{w}^{n}_{h} + \tau_{n} \sum^{i-2}_{j=1}\Bar{a}_{i,j}\mathbf{E}_{j} + \tau_{n}\Bar{a}_{i,i-1}\mathbf{E}_{i-1}, &\quad i = 1 \ldots s, \\
	\mathbf{w}_{i} = \mathbf{w}^{*}_{i} + \tau_{n} \sum^{i-1}_{j=1}a_{i,j}\mathbf{I}_{j} + \tau_{n}\Bar{a}_{i,i}\mathbf{I}_{i}, &\quad i = 1 \ldots s, \\
	\mathbf{w}^{n+1}_{h} = \mathbf{w}^{n}_{h} + \tau_{n} \sum^{s}_{i=1}\Bar{b}_{i}\mathbf{E}_{i} + \tau_{n}\sum^{s}_{i=1}b_{i}\mathbf{I}_{i},
	\end{dcases}
\end{equation}
where $s$ 
are the stages of the IMEX-RK method, $\mathbf{E}_{i} = \mathcal{\mathbf E}(\mathbf{w}_{i})$, ${\mathbf I}_{i} = \mathcal{\mathbf I}(\mathbf{w}_{i})$$, $\; $i = 1\ldots s$, and $\{\hat{\mathbf b},\hat{\mathbf A}\}, \; \{\mathbf b,\mathbf A\}$ are the coefficients for the explicit and implicit parts of the scheme, respectively. We refer the reader to \cite{christopher2001additive} for the values of these coefficients, where the Butcher's Tableau of both second- and third-order methods are given. As a final stage of this method, we solve the linear system in \eqref{IMEX-RK} using the {Iterative Biconjugate Gradient Stabilized Krylov subspace} method; see \cite{Krylov.1931, vdVorst.1992}.

\subsection{Initial conditions \& simulation parameters}\label{app:initial condition}
We give here a short description of the construction of the initial condition for the ECM, CAFs, TGF-$\beta$, and MMPs densities. A visual representation is shown in Figure \ref{fig:Initial condition}, but the construction process goes as follows, and more details can be found in \cite{franssen2021anovel3d, katsaounis2024genuinely}. To begin with, an $8 \times 8$ matrix is created with entries taken from a standard normal distribution, $N(0,1)$. We take a number of refinement steps until the matrix resolution reaches the desired (computational) resolution of the domain. At each stage, the matrix size is doubled to increase the resolution of each density. The entries of the new larger density are obtained by interpolating the values of the previous smaller density, with the addition of a small amount of Gaussian noise. Therefore, as each density is refined, it preserves the initial randomly chosen structure observed in the $8 \times 8$ grid, with areas of higher or lower density appearing in the same regions regardless of resolution. This procedure is extended into three dimensions. However, due to the increased computational time that three-dimensional simulations impose, a maximum refinement of $64 \times 64 \times 64$ will be used for all densities in this experiment. The simulation parameters used for the numerical experiments in Section \ref{sec5} are shown in Table \ref{tab:params}.
\begin{figure}[h!]
    \footnotesize
    \centering
    \begin{tabular}{ccc}
        \includegraphics[width=0.15\textwidth]{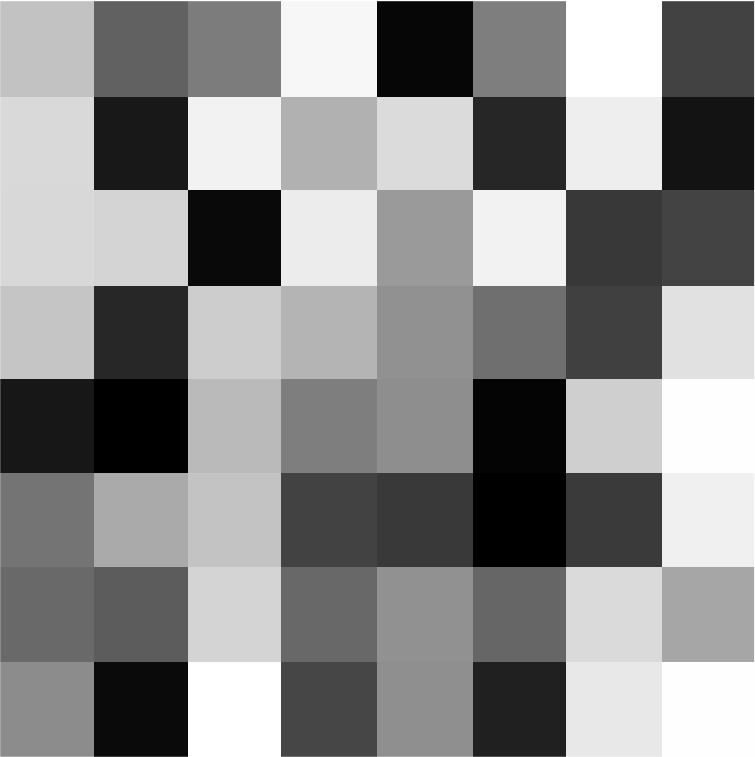} 
            &\includegraphics[width=0.15\textwidth]{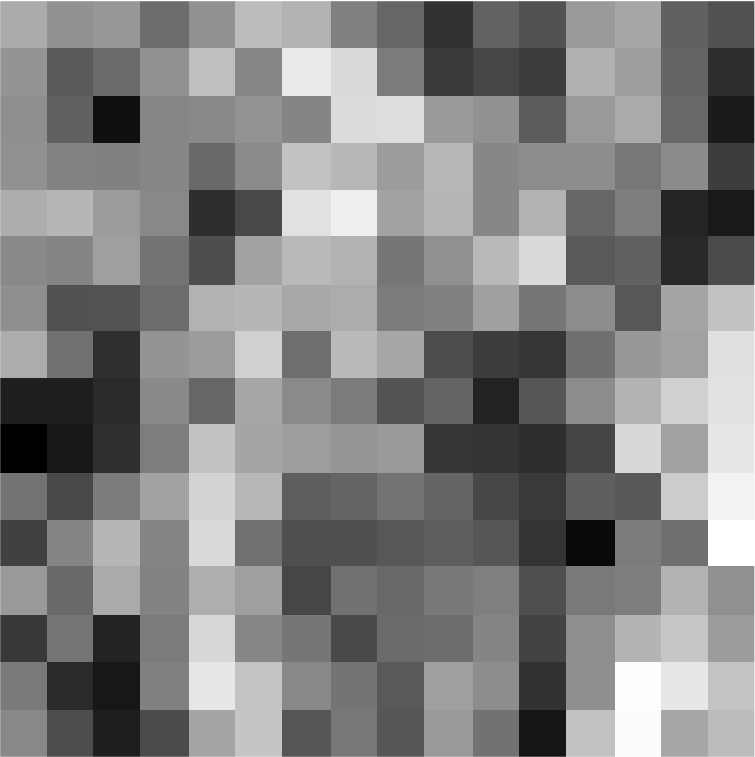}
            &\includegraphics[width=0.15\textwidth]{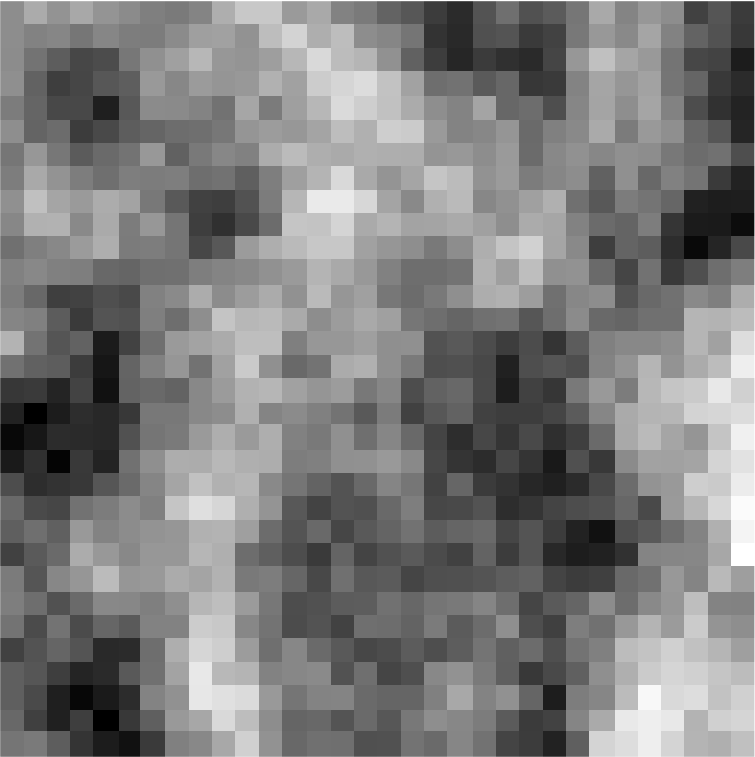}\\   {(a)~$8\times8$}&{(b)~$16\times16$}&{(c)~$32\times32$}\\
        \includegraphics[width=0.15\textwidth]{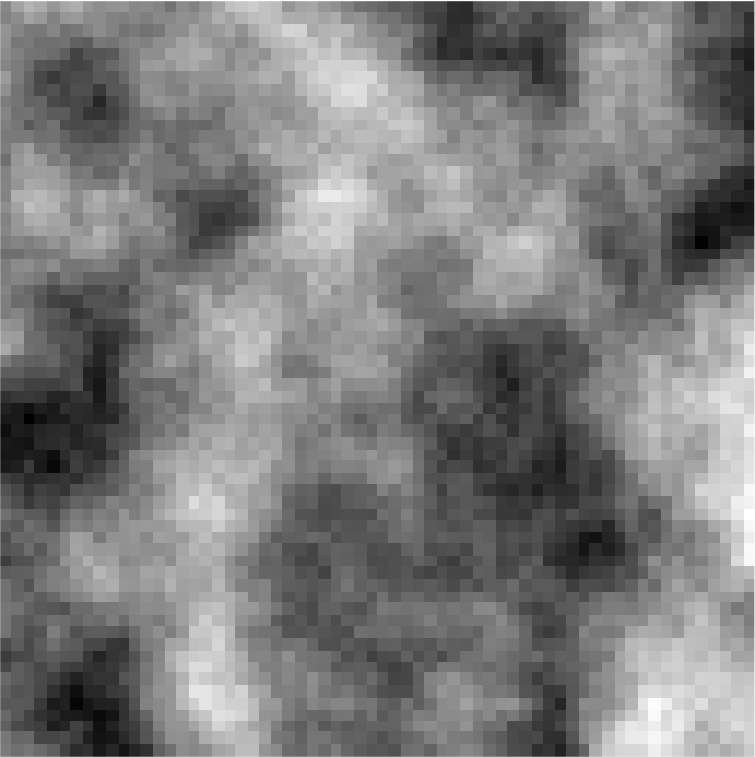}
            &\includegraphics[width=0.15\textwidth]{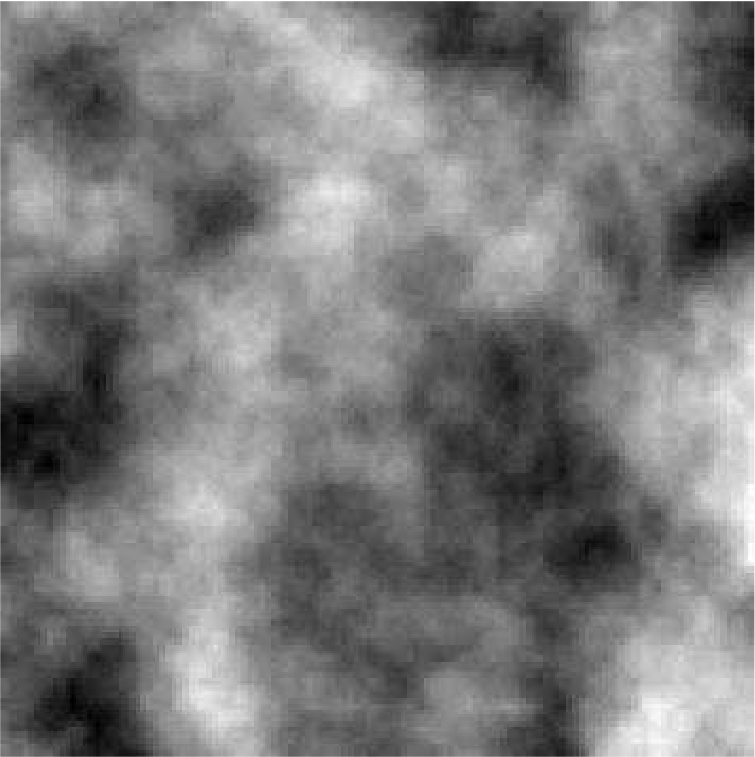}
            &\includegraphics[width=0.15\textwidth]{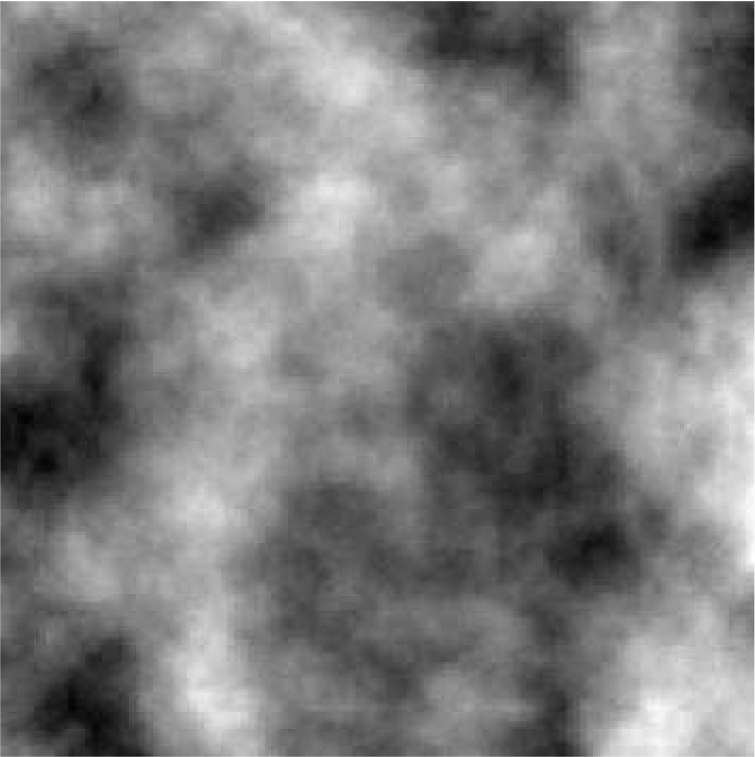}\\            {(d)~$64\times64$}&{(e)~$128\times128$}&{(f)~$256\times256$}
    \end{tabular}
    \caption{ \footnotesize Construction of the initial condition for the ECM, TGF-$\beta$, MMPs and CAFs densities employed in Experiments in Section \ref{sec5}. The process starts with random values over an $8\times8$ grid; these introduce the basic structure of each density. At each stage, the grid is bisected, and the new values are obtained by averaging neighboring values from the previous coarser grid, with added noise. The process stops when the required grid size is reached.}
    \label{fig:Initial condition}
 \end{figure}

\begin{table}[ht]
    \centering
    \begin{tabular}{|c|c|c|}
    \hline
        Parameters & ~~~~~ $2D$  ~~~~~ &  ~~~~~ $3D$ ~~~~~ \\
        \hline
        $\chi_E$ & 0.01 & 0.01\\
        $\chi_M$ & 1.99 & 10\\
        $\chi_F$ & 0.99 & 0.8\\
        $\rho_E$ & 10 & 10\\
        $\rho_F$ & 5 & 5\\
        $\rho_b$ & 1.5 & --\\
        $\rho^F_m$ & 1 & 1\\
        $\rho^E_m$ & 1 & 1\\
        $\rho_v$ & 6 & 6\\
        $\lambda_F$ & 1 & 1\\
        $\lambda_b$ & 1 & --\\
        $\lambda_m$ & 1.2 & 1.2\\
        $\delta_1$ & 40 & 40\\
        $\delta_2$ & 2 & 2\\
        $\delta_3$ & 500 & 0\\
        $\mu_E/\alpha_1$ & 10 & 10\\
        $\mu_M/\alpha_2$ & 20 & 20\\
    \hline
    \end{tabular}
    \caption{Simulations parameters for the numerical experiments shown in Section \ref{sec5}.}
    \label{tab:params}
\end{table}


\newpage
\bibliographystyle{unsrtinitials}
\bibliography{meta}

\end{document}